\documentclass[10pt,a4paper]{article}

\usepackage[utf8]{inputenc} %
\usepackage[T1]{fontenc} %

\usepackage{amssymb}
\usepackage{a4wide}
\usepackage{t1enc}
\usepackage{graphics}
\usepackage{graphicx}
\usepackage{epsfig}
\usepackage{verbatim}
\usepackage{times}
\usepackage{caption}
\usepackage{subcaption}
\usepackage{pifont}
\usepackage{shadow}
\usepackage{listings}
\usepackage{latexsym}
\usepackage{pdfsync}
\usepackage{amsthm}
\usepackage{bbm}
\usepackage{amsmath, amsthm, amsfonts}
\usepackage{fancyhdr}
\usepackage{dsfont}  
\usepackage{flafter}  
\usepackage{placeins}
\usepackage{mathrsfs}
\usepackage{xcolor}
\usepackage{paralist}
\usepackage{geometry} 
\usepackage[normalem]{ulem}
\usepackage{xspace}
\usepackage{nicefrac}

\usepackage{pgfplots}
\usepackage{tikz}
\usetikzlibrary{spy}
\usetikzlibrary{patterns}

\def\be{\mathbf{e}}

\def\R{\mathbb{R}}

\def\N{\mathbb{N}}
\def\Z{\mathbb{Z}}

\newcommand\mx{{\mathbf{x}}}
\newcommand\mX{{\mathbf{X}}}
\newcommand\mY{{\mathbf{Y}}}
\newcommand\mJ{{\mathbf{J}}}
\def\ttv{\mathtt{v}}
\def\tta{\mathtt{a}}
\def\ttb{\mathtt{b}}

\def\dsp{\displaystyle}

\newcommand{\xdrv}[2]{{\frac{\partial #1}{\partial #2}}}

\newtheorem{hypo}{Assumption}[section]

\newtheorem{theorem}[hypo]{Theorem}
\newtheorem{lemma}[hypo]{Lemma}

\newtheorem{proposition}[hypo]{Proposition}
\newtheorem{remark}[hypo]{Remark}

\newcommand{\bn}{{\mathbf n}}
\newcommand{\bx}{{\mathbf x}}
\newcommand{\bX}{{\mathbf X}}

\newcommand{\hole}{\text{\rm hole}}
\newcommand{\IR}{{\mathbb R}}
\newcommand{\IC}{{\mathbb C}}
\newcommand{\IN}{{\mathbb N}}
\newcommand{\IZ}{{\mathbb Z}}
\newcommand{\macro}{\text{macro}}

\newcommand\xnorm[2]{\left\lVert #1 \right\rVert_{#2}}
\newcommand\xjump[1]{\left[ #1 \right]}
\newcommand\xavrg[1]{\left< #1 \right>}

\newcommand{\OmegaTop}    {\Omega_{\mathrm{T}}}

\newcommand{\OmegaBottom} {\Omega_{\mathrm{B}}}

\newcommand{\OmegaLimit}  {\Omega^0}
\newcommand{\LTop}        {L'}
\newcommand{\LBottom}     {L }
\newcommand{\HTop}        {H_{\mathrm{T}}}
\newcommand{\HBottom}     {H_{\mathrm{B}}}
\newcommand{\eOne}        {\mathrm{e}_1}
\newcommand{\eTwo}        {\mathrm{e}_2}
\newcommand{\Ltwo}        {\mathrm{L}^2}
\newcommand{\LtwoG}       {\mathrm{L}^2(\Gamma)}

\newcommand{\Hone}        {\mathrm{H}^1}
\newcommand{\Honeloc}     {\mathrm{H}^1_{\text{loc}}}

\newcommand{\Honehalf}    {\mathrm{H}^{\nicefrac12}}
\newcommand{\HonehalfG}   {\mathrm{H}^{\nicefrac12}     (\Gamma)}

\newcommand{\eg}{\textit{e.\,g\mbox{.}}\xspace}
\newcommand{\ie}{\textit{i.\,e\mbox{.}}\xspace}

\newcommand{\cf}{\textrm{cf.}\xspace}

\newcommand\commu[2]{\left[ #1 , #2 \right]}
\definecolor{mydarkcyan}{rgb}{0,0.5,0.5}
\definecolor{mygreen}{RGB}{46,192,87}

\makeatletter

\@addtoreset{equation}{section}
\makeatother

\pgfplotscreateplotcyclelist{mylist}{%
blue,every mark/.append style={fill=blue!80!black},mark=*\\%
red,every mark/.append style={fill=red!80!black},mark=square*\\%
brown!60!black,every mark/.append style={fill=brown!80!black},mark=o\\%
black,mark=star\\%
blue,every mark/.append style={fill=blue!80!black},mark=diamond*\\%
red,densely dashed,every mark/.append style={solid,fill=red!80!black},mark=*\\%
brown!60!black,densely dashed,every mark/.append style={
solid,fill=brown!80!black},mark=square*\\%
black,densely dashed,every mark/.append style={solid,fill=gray},mark=otimes*\\%
blue,densely dashed,mark=star,every mark/.append style=solid\\%
red,densely dashed,every mark/.append style={solid,fill=red!80!black},mark=diamond*\\%
}

\begin{document}
\begin{center} {\Large{On the homogenization of the Helmholtz problem with
      thin perforated walls of finite length}}\\[0.5cm]

{\large{Adrien Semin$^{a,}$\footnote{This
  work was carried out where the author was at Research center
  Matheon, Institut f\"ur Mathematik, Technische Universit\"at Berlin,
  10623 Berlin, Germany.}, B\'erang\`ere Delourme$^{b,}$\footnote{Part of this work was
      carried out where the author was on research leave at Laboratoire
    POEMS, INRIA-Saclay, ENSTA, UMR CNRS 2706, France}, Kersten Schmidt$^{a,c,d}$}}\\[0.5cm]

{\small $a$: Brandenburgische Technische Universit\"at
  Cottbus-Senftenberg, Institut f\"ur
  Mathematik, 03046 Cottbus, Germany}\\
{\small $b$: Universit\'e Paris 13, Sorbone Paris Cit\'e, LAGA, UMR
  7539, 93430 Villetaneuse, France} \\
{\small $c$: Research center Matheon, 10623 Berlin, Germany}\\
{\small $d$: Institut f\"ur Mathematik, Technische Universit\"at Berlin, 10623 Berlin, Germany} \\
\end{center}

{\small \noindent  \textbf{Corresponding author:} Adrien Semin,
  Brandenburgische Technische Universit\"at
  Cottbus-Senftenberg\\
Address: Platz der deutschen Einheit 1,
Hauptgebäude (HG), Raum 4.01,
D-03046 Cottbus\\
E-mail: adrien.semin@b-tu.de\\
Tel: +49 (0)355 69 3735}\\[0.5cm]

% Table of contents is not present in the Laplace article, so I've
% removed it.
%%\tableofcontents

\textbf{Abstract}\\
In this work, we present a new solution representation for the
Helmholtz transmission problem in a bounded domain in $\mathbb{R}^2$
with a thin and periodic layer of finite length. The layer may
consists of a periodic pertubation of the material coefficients or it
is a wall modelled by boundary conditions with an periodic array of
small perforations.  We consider the periodicity in the layer as the
small variable $\delta$ and the thickness of the layer to be at the
same order. Moreover we assume the thin layer to terminate at
re-entrant corners leading to a singular behaviour in the asymptotic
expansion of the solution representation. This singular behaviour
becomes visible in the asymptotic expansion in powers of $\delta$
where the powers depend on the opening angle.  We construct the
asymptotic expansion order by order. It consists of a macroscopic
representation away from the layer, a boundary layer corrector in the
vicinity of the layer, and a near field corrector in the vicinity of
the end-points. The boundary layer correctors and the near field
correctors are obtained by the solution of canonical problems based,
respectively, on the method of periodic surface homogenization and on
the method of matched asymptotic expansions.  This will lead to
transmission conditions for the macroscopic part of the solution on an
infinitely thin interface and corner conditions to fix the unbounded
singular behaviour at its end-points. Finally, theoretical
justifications of the second order expansion are given and illustrated
by numerical experiments. The solution representation introduced in
this article can be used to compute a highly accurate approximation of
the solution with a computational effort independent of the small
periodicity $\delta$.



\section*{Introduction}	
The present work is dedicated to the iterative construction of a
second order asymptotic expansion of the solution to an Helmholtz
problem posed in a non-convex polygonal domain which excludes a set of
similar small obstacles equi-spaced along the line between two
re-entrant corners.  The distance between two consecutive obstacles,
which appear to be holes in the domain, and the diameter of the
obstacles are of the same order of magnitude $\delta$, which is
supposed to be small compared to the dimensions of the domain.  The
presence of this thin periodic layer of holes is responsible for the
appearance of two different kinds of singular behaviors. First, a
highly oscillatory boundary layer appears in the vicinity of the
periodic layer. Strongly localized, it decays exponentially fast as
the distance to the periodic layer increases. Additionally, since the
thin periodic layer has a finite length and ends in corners of the
boundary, corners singularities come up in the neighborhood of its
extremities. The objective of this work is to provide a practical
asymptotic expansion that takes into account these two types of
singular behaviors.

The boundary layer effect occurring in the vicinity of the periodic
layer is well-known. It can be described using a two-scale asymptotic
expansion (inspired by the periodic homogenization theory) that
superposes slowly varying macroscopic terms and periodic correctors
that have a two-scale behavior: these functions are the combination of
highly oscillatory and decaying functions (periodic of period $\delta$
with respect to the tangential direction of the periodic interface and
exponentially decaying with respect to $d/\delta$, $d$ denoting the
distance to the periodic interface) multiplied by slowly varying
functions. This boundary layer effect has been widely investigated
since the work of Panasenko~\cite{Panasenko81}, 
Sanchez-Palencia~\cite{RapportSanchezPalencia,SanchezPalencia},
Achdou~\cite{Achdou,AchdouCR} and
Artola-Cessenat~\cite{ArtolaCessenat,ArtolaCessenat2}. In particular,
high order asymptotics have been derived for the Laplace
equation~\cite{AchdouPironneauValentin,Madureira,CiupercaJaiPoignard,BreschMilisic2010}
and for the Helmholtz equation\cite{poirier2006impedance,Poirier}.

On the other hand, corner singularities appearing when dealing with
singularly perturbed boundaries have also been widely
investigated. Among the numerous examples of such singularly perturbed
problems, we can mention the cases of small inclusions~(see
chapter 2 of Ref.~\cite{LivreNazarov1} for the case of one inclusion and
Ref.~\cite{MR2573145} for the case of several inclusions), perturbed
corners\cite{DaugeTordeuxVialVersionLongue}, propagation of waves in
thin slots\cite{fente1,fente2}, propagation of waves across a thin
interface\cite{Claeys.Delourme:2013}, diffraction by
wires\cite{XavierArticle}, diffraction by a muffler containing
perforated ducts\cite{Bonnet.Drissi.Gmati:2005}, or the mathematical
investigation of patched antennas\cite{BendaliMakhloufTordeux}.
Again, this effect can be depicted using two-scale asymptotic
expansion methods that are the method of multi-scale expansion
(sometimes called compound method) and the method of matched
asymptotic expansions\cite{VanDyke,LivreNazarov1,Ilin}.
Following these methods, the solution of the perturbed problem may be
seen as the superposition of slowly varying macroscopic terms that do
not see directly the perturbation and microscopic terms that take into
account the local perturbation.

Recently, the authors investigated a Poisson problem in a polygonal
domain which excludes a set of similar small obstacles equi-spaced
along the line between two re-entrant
corners\cite{ResearchReportAKB,Delourme.Semin.Schmidt:2015}. In their
study, they have combined the two different kinds of asymptotic
expansions mentioned above in order to deal with both corner
singularities and the boundary layer effect. Based on the matched
asymptotic expansions, the authors constructed and justified a
complete asymptotic expansion. This asymptotic expansion relies on the analysis of the behaviour of the solutions of the Poisson problem in
an infinite cone with oscillating boundary with Dirichlet boundary
conditions by Nazarov\cite{Nazarov205}. In the present paper, we are
going to extend this work for the Helmholtz equation by constructing
explicitly and rigorously the terms of the expansions up to order $2$
(with Neumann boundary conditions on the perforations of the layer).

The remainder of the paper is organized as follows.  In
Section~\ref{sec:intr-sett-probl} we are going to define the problem,
show the main ingredients of the asymptotic expansion following the
method of matched asymptotic expansions, and give the main results.
The asymptotic expansion of the solution away from the corners is
  given in Section~\ref{SectionTransmissionConditionsBoundaryLayer},
  whereas the problem for the terms of the near field expansion and
  their behavior towards infinity, is analyzed in
  Section~\ref{sec:constr-singularities}. The terms of this expansion
takes into account the boundary layer effect due to the thin layer with
small perforations and satisfy transmission conditions. Then, the
matching of the far field and near field expansions and the iterative
construction of the terms of the asymptotic expansions are conducted
in Section~\ref{sec:constr-first-terms}. Finally, in
Section~\ref{sec:error_estimates} the asymptotic expansion is
justified with an error analysis.

\section{Description of the problem and main results}
\label{sec:intr-sett-probl}
\noindent In this section, we first define the problem under consideration (Section~\ref{SubsectionDescription}). Then, we give the Ansatz of the asymptotic expansion (Section~\ref{SubsectionAnsatz}). Finally, we give the main result of this paper, which states the existences of the terms of the asymptotic expansion and the convergence of the truncated series toward the exact solution and we show a numerical illustration of the result (Section~\ref{SubsectionMainResults}).
\subsection{Description of the problem}\label{SubsectionDescription}

%n this section we define the domain of interest
%$\Omega^\delta \subset \IR^2$ which contains a thin perforated wall of finite
%length as well as its limit when $\delta\to0$, and we introduce the Helmholtz
%problem on $\Omega^\delta$ to be studied in this article.

\subsubsection{Definition of the domain $\Omega^\delta$ with a thin perforated wall of finite
  length} 

Our domain of interest $\Omega^\delta$ consists of a (non-convex)
polygon $\Omega$ intersected with the complement of an array of
'small' similar obstacles, see~Fig.~\ref{fig:OmegaDelta}.  The polygon $\Omega$, represented on Figure~\ref{fig:Omega}, is the union of the rectangular domain $\OmegaTop$ and a symmetric
trapezoidal domain $\OmegaBottom$ (of height $\HBottom>0$) that share a common interface $\Gamma$ ($\Gamma$ corresponds to the upper side of $\OmegaBottom$ and the lower side of $\OmegaTop$). More precisely,  
\begin{equation}
\OmegaTop =  \left\{ \mathbf{x} = (x_1,x_2) \in \R^2, \mbox{such that}  -\LTop < x_1 < \LTop,\; \mbox{and} \;  0 < x_2 <  \HTop\right\}, \quad  (\LTop > \LBottom > 0, \; \HTop>0),
\end{equation}  
the common interface $\Gamma$ is given by
\begin{equation}
 \Gamma := \{ \mathbf{x} = (x_1,x_2) \in \R^2,  -\LBottom < x_1 < \LBottom \mbox{and}  \; x_2 =  0 \}
 \end{equation} 
 and 
 \begin{equation}
 \Omega = \OmegaBottom \cup \OmegaTop \cup \Gamma.
 \end{equation}
We point out that the polygon $\Omega$ has two re-entrant corners
$\mathbf{x}_{O}^\pm = ( \pm\LBottom,0)$ of angle of $\Theta > \pi$.
\begin{figure}[htbp]
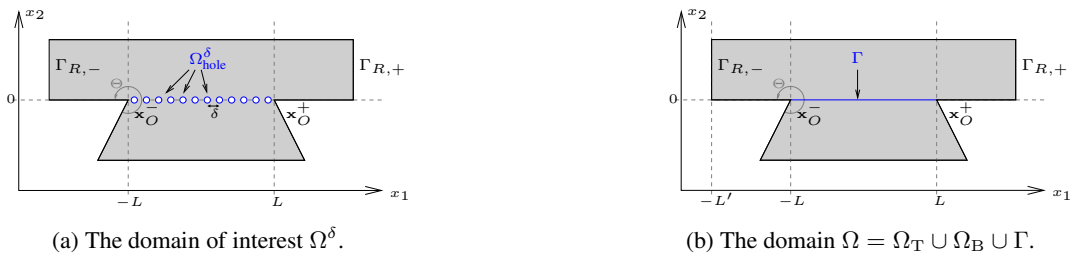

  \centering
  \begin{subfigure}[b]{0.44\textwidth}
    \null\hfill\input{domain.pspdftex}\hfill\null%Exported to 80%
    \caption{The domain of interest $\Omega^\delta$.} %$= \Omega \backslash \overline{\Omega^\delta_\hole}$.
    \label{fig:OmegaDelta}
  \end{subfigure}
  \qquad
  \begin{subfigure}[b]{0.44\textwidth}
    \null\hfill\input{domain2.pspdftex}\hfill\null%Exported to 80%
    \caption{The domain $\Omega = \OmegaTop \cup \OmegaBottom
      \cup \Gamma$.}
    \label{fig:Omega}
  \end{subfigure}
  \caption{Illustration of the polygonal domain $\Omega$ and the
    domain of interest $\Omega^\delta$.}
  \label{fig:dessinsDomaines}
\end{figure}

\noindent Besides, let $\widehat{\Omega}_{\hole} \in \IR^2$ be a {\emph {smooth}}
canonical bounded open set
(not necessarily connected) strictly included in the domain $(0,1)
\times (-1,1)$. Then, let $\N^\ast := \N \setminus \{0 \}$ denote the set of positive integers
and let
 $\delta$ be a positive real number (that is
supposed to be small) such that 
\begin{equation}
\frac{2\LBottom}{\delta} = q \in \N^\ast.
\end{equation}

\noindent %
Now, let $\Omega^\delta_{\hole}$ be the thin (periodic) layer
consisting of $q$ equi-spaced similar obstacles defined by scaling and
shifting the canonical obstacle $\widehat{\Omega}_{\hole}$ (see
Fig.~\ref{fig:OmegaDelta}):
\begin{equation}
  \label{eq:layer_Omega_hole}
  \Omega^\delta_{\hole} = \bigcup_{\ell = 1}^{q} 
  \left\{ -\LBottom \eOne + \delta \{ \widehat{\Omega}_{\hole} +
        (\ell-1) \eOne \} \,  \right\}.
\end{equation}
Here, $\eOne$ and $\eTwo$ denote the unit vectors of $\IR^2$ and $\delta$ is assumed to be smaller
than $\HTop$ and $\HBottom$ such that $\Omega^\delta_\hole$ does not touch the top or bottom boundaries of $\Omega$. %
Finally, we define our domain of interest as
\begin{align*}
  \Omega^\delta = (\OmegaBottom \cup \OmegaTop \cup \Gamma) \backslash \overline{\Omega^\delta_{\hole}}.
\end{align*}
Its boundary $\partial\Omega^\delta$ consists of the union of
  three sets (see Figure~\ref{fig:dessinsDomaines}): 
  \begin{enumerate} 
  \item[-] the set of holes
  $\Gamma^\delta = \partial\Omega^\delta_{\hole}$,
  \item[-]  the lateral boundaries  $\Gamma_{R,\pm} = \left\lbrace \bx \in \partial \Omega^\delta \ / \
    x_1 = \pm \LTop \right\rbrace$ of $\OmegaTop$:
    $$\Gamma_R = \Gamma_{R,-} \cup \Gamma_{R,+},$$  
  \item[-] the remaining part
  $\Gamma_N = \partial \Omega^\delta \setminus \left( \Gamma^\delta
    \cup \Gamma_R \right) = \partial \Omega \setminus \Gamma_R$, namely  the boundaries of $\OmegaBottom$ except $\Gamma$ and the upper boundary $\OmegaTop$.
    \end{enumerate}
Note, that in the limit $\delta\to 0$ the repetition of holes
degenerates to the interface $\Gamma$, the domain $\Omega^\delta$ to
the domain $\OmegaLimit := \OmegaTop \cup \OmegaBottom = \Omega\setminus\Gamma$, and
its boundary $\partial\Omega^\delta$ to $\partial\Omega \cup \Gamma$.

\begin{remark}
  Note that the asymptotic analysis that will be employed in this
  article can be simply transferred to similar domains with thin
  periodic layers and different boundary conditions away from the
  layer.  For example, the upper subdomain $\OmegaTop$ can be replaced
  by a half space where radiation conditions are imposed at infinity.
\end{remark}

\subsubsection{The Helmholtz problem with a thin perforated wall of finite length.}
On the domain $\Omega^\delta$ we introduce the Helmholtz transmission
problem to be considered in this article. Let $k_0>0$ be a given positive number, and let \mbox{$u_{\text{inc}} = \exp(\imath k_0 (x_1 - L'))$}  be an incident plane wave of wavenumber $k_0$ coming from
the left, we seek
$u^\delta$ as solution of the total field problem
\begin{equation}
  \label{eq:perturbed_Helmholtz}
  \left\lbrace\quad
    \begin{aligned}
      - \Delta u^\delta -(k^\delta)^2(\bx) u^\delta &= 0, &\quad&
      \text{in }
      \Omega^\delta,\\
      \nabla u^\delta \cdot \bn &= 0, &&\text{on } \Gamma^\delta,\\
      \nabla (u^\delta - u_{\text{inc}}) \cdot \bn - \imath k_0
      (u^\delta - u_{\text{inc}}) &= 0,
      &&\text{on }\Gamma_R^-,\\
      \nabla u^\delta \cdot \bn - \imath k_0
      u^\delta  &= 0  && \text{on } \Gamma_R^+,\\ 
      \nabla u^\delta \cdot \bn &= 0, &&\text{on }\Gamma_N.
    \end{aligned}
  \right.
\end{equation}
\noindent In the previous system of equations, $\bn$ stands for the outward unit normal  vector of $\partial\Omega^\delta$. %
In the first equation of~\eqref{eq:perturbed_Helmholtz}, $k^\delta(\bx)$  is given by
$$
k^\delta(\bx) = \begin{cases}
k_0 & \mbox{if} \; \bx \quad  \in \Omega^\delta \setminus (-L,L) \times (-\delta,\delta),\\[1ex]
\dsp \widehat{k}(\frac{x_1}{\delta}, \frac{x_2}{\delta}) & \mbox{otherwise},
\end{cases} 
$$ 
 where the function $\widehat{k}$ (defined on  $\R^2$) is a smooth, positive function that is $1$-periodic with respect to its first
variable $s$. We also assume that there exists $\eta \in (0,1)$ such that $ k(s,t)= k_0$ for $|t| > \eta$ or $|s|>\eta$. In other words, $k^\delta$ is a smooth function that is  constant equal to $k_0$ outside the thin layer  $ (-L,L) \times (-\delta,\delta)$ and periodic of period $\delta$ in the vicinity of it. In particular, $k^\delta$ is bounded from above and from below independently of $\delta$ and $k^\delta$ tends almost everywhere to $k_0$.\\

\noindent  The model~\eqref{eq:perturbed_Helmholtz} can be seen as a Helmholtz transmission
problem in an infinite wave-guide with Neumann boundary conditions on
the (rigid) walls, especially, on $\Gamma^\delta$ and $\Gamma_N$,
which is truncated to a finite domain using first-order absorbing
boundary conditions of Robin's type on $\Gamma_R$ (see \eg
Ref.~\cite{Goldstein:1982}). The following well-posedness result,
  based on the Fredholm alternative (Theorem 6.6 in
  ~\cite{BrezisAnglais}), is standard (see for instance Lemma 3.4
  in\cite{fente2} - Proposition 11.3 in \cite{Claeys.Delourme:2013}). 
: \begin{proposition}[Existence, uniqueness and stability]
  \label{prop:existence_uniqueness_u_delta}
  For any $\delta>0$ there exists a unique solution $u^\delta$ of
  problem~\eqref{eq:perturbed_Helmholtz} in
  $\Hone(\Omega^\delta)$. Moreover, there exists a constant $C$ (independent
  of~$\delta$) such that
  \begin{equation}
    \label{eq:prop_norm_u_norm_f}
    \xnorm{u^\delta}{\Hone(\Omega^\delta)} \leqslant
    C \| \nabla u_{\text{\em inc}}\cdot \bn - \imath k_0 u_{\text{\em inc}} \|_{(\Honehalf(\Gamma_R^-))'}.
  \end{equation}
\end{proposition}
\noindent For the sake of completeness, the proof of the previous is written in Appendix~\ref{AppendixStability}. We remark that the constant $C$ appearing in the stability estimates~\eqref{eq:prop_norm_u_norm_f} is independent of $\delta$ but depends on $k_0$, $\hat{k}$, and $\hat{\Omega}_{\text{hole}}$. \\

\noindent The objective of this paper is to describe the behaviour of
$u^\delta$ as $\delta$ tends to $0$. Our work relies on a construction
of an asymptotic expansion of $u^\delta$ as $\delta$ tends to $0$. %

\subsection{Ansatz of the asymptotic expansion}\label{SubsectionAnsatz}

As mentioned in the introduction, due to the presence of both the
periodic layer and the two re-entrant corners, it seems not possible to write
a simple asymptotic expansion valid in the whole domain. We have to
take into account both the boundary layer effect in the vicinity of
$\Gamma$ and the additional corner singularities appearing in the
neighborhood of the two re-entrant corners $\mx_{O}^ \pm$. To do so, we shall
distinguish differents areas where the expansions are different:
\begin{enumerate}
\item[-] a {\em far field area} located 'far' from the 
corners $\mx_{O}^ \pm$  (hached area on Fig.~\ref{DessinAsymptotic}),
\item[-] two {\em near field zones} located in their
vicinities (grey areas on Fig.~\ref{DessinAsymptotic}).
\end{enumerate}
 The far and near field areas intersect in the (non-empty) matching zone.

\begin{figure}[h!]
  \centering
  \input{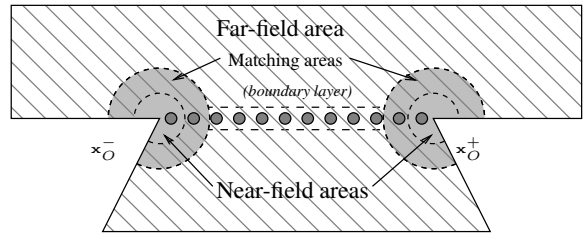}
  \caption{Schematic representation of the overlapping subdomains for the asymptotic expansion.
  The far field area ({\em hatched}) away from the corners
  $\mx_{O}^ \pm$ is overlapping the near field area ({\em gray}) in the matching zone.}
\label{DessinAsymptotic}
\end{figure}

\subsubsection{Far field expansion} 
In this section, we write an asymptotic expansion valid away from the two corners $\mx_{O}^\pm$ (hatched area
in~Fig.~\ref{DessinAsymptotic}). We shall decompose $u^\delta$ as the
superposition of a {\em macroscopic part} (that contains no rapid
oscillation) and a {\em boundary layer contribution} localized in the
neighborhood of the thin periodic layer. In the present case the solution $u^\delta$ is then expanded in powers of $\delta$, where each power
is the sum of an integer and a so-called singular exponent $\lambda_n$ given by
 \begin{equation}\label{DefinitionLambda}
 \lambda_n = n\lambda, \quad \lambda = \frac{\pi}{\Theta}.
 \end{equation}
More precisely, we choose the ansatz
\begin{multline}
  \label{FFExpansion}
  u^\delta(\mx) = %
    u_{\text{FF},0,0}^\delta(\mx) %
    + \delta^{\lambda_1} \,u_{\text{FF},1,0}^\delta(\mx) %
    + \delta \,u_{\text{FF},0,1}^\delta(\mx) %
    \\+ \delta^{\lambda_2} \,u_{\text{FF},2,0}^\delta(\mx) %
    + \delta^{\lambda_1+1} \, u_{\text{FF},1,1}^\delta(\mx) %
    + O(\delta^{\min(\lambda_3,2)}),
\end{multline}
where each term takes the form
\begin{equation} \label{def_uFFnq}
  u_{\text{FF},n,q}^\delta (\mathbf{x}) =  
  \begin{cases}
    %(1- (1-\chi\big(\frac{|x_1| - \LBottom}{\delta}\big)) (1 - \chi\left( \frac{x_2}{\delta} \right))) 
    u_{n,q}^\delta (\mx) & \mbox{if} \; |x_1| > \LBottom+2\delta, \\
    \chi\left( \frac{x_2}{\delta} \right) 
    u_{n,q}^\delta (\mathbf{x})
    + %\chi\big(\frac{|x_1| - \LBottom}{\delta}\big) 
    \Pi_{n,q}^\delta(x_1,
    \frac{\mathbf{x}}{\delta}) & \mbox{if} \; |x_1|< \LBottom - 2\delta,
  \end{cases} 
\end{equation}
with a smooth transition for $L-2\delta < |x_1| < L + 2\delta$, which
is detailed later in the article (see
Section~\ref{sec:error_estimates},
  Ref.~\cite{ResearchReportAKB} and
  Ref.~\cite{Delourme.Semin.Schmidt:2015}).  Here,
$u_{\text{FF},n,q}^\delta$, $(n,q) \in \IN^2$, is a combination of
{\em macroscopic terms} $u_{n,q}^\delta$ and {\em boundary layer
  correctors} $\Pi_{n,q}^\delta$, and $\chi : \R \mapsto (0,1)$
denotes a smooth cut-off function satisfying
\begin{equation}
  \label{defchi}
  \chi(t) = \begin{cases}
    1 & \mbox{if} \; |t| >2, \\[1ex]
    0 & \mbox{if} \; |t|<1.
  \end{cases}
\end{equation}
The superscript~$\delta$ in $u_{n,q}^\delta$ and $\Pi_{n,q}^\delta$ indicates that they may depend on $\delta$, however, this 
dependence is only polynomial in $\ln \delta$.
% \bdsn{\\Faut il vraiment mettre cette remarque a cette place ?? peut etre qu on pourrait la mettre plus loin ? je suis pour l'enlever}
% \begin{remark} For a given $N_0 \in \R^+$, it is possible to generalize the previous formula by truncating the
% series~\eqref{FFExpansion} at order~$\delta^{N_0}$, taking a sum of
% terms of the form $\delta^{\lambda_n + q} u_{\text{FF},n,q}^\delta$,
% $(n,q)\in \IN_{N_0}$, where $\IN_{N_0}$ denotes the set of indexes
% $(n,q)\in\N^2$ such that $\lambda_n+q < N_0$. In the present case, we
% aim to construct a second order approximation of
% $u^\delta$, which leads to consider
% $\N_2 := \lbrace (0,0), (1,0), (0,1), (2,0), (1,1), (3,0)\rbrace$ if
% $\Theta > \frac{3\pi}{2}$ and
% $\N_2 := \lbrace (0,0), (1,0), (0,1), (2,0), (1,1)\rbrace$
% otherwise.
% \end{remark}
In the next three paragraphs, we shall write the equations satisfied by the macroscopic terms, the boundary layer correctors and the transmissions conditions liking the two kinds of terms. The detailed derivation of these equations is done in Section~\ref{SubSectionDerivationu0u1}.

\textbf{Macroscopic equations.}
The macroscopic terms $u_{n,q}^\delta$ are defined in the limit domain
$\OmegaTop \cup \OmegaBottom$. Based on the usual decay assumption (see \eg
Ref.~\cite{poirier2006impedance} and Ref.~\cite{Poirier}) on the boundary layer correctors
we find that the macroscopic terms satisfy the homogeneous Helmholtz
equation
\begin{equation}\label{FFVolum}
  - \Delta u_{n,q}^\delta - k_0^2 u_{n,q}^\delta = 0 \; \mbox{in} \; \OmegaTop \cup \OmegaBottom,
\end{equation}
which is completed with prescribed boundary conditions on $\Gamma_R$
and $\Gamma_N$
\begin{equation}
  \label{eq:FF_macro_bc}
  \begin{aligned}
   \nabla (u_{0,0}^\delta - u_{\text{inc}}) \cdot \bn - \imath k_0
    (u_{0,0}^\delta - u_{\text{inc}}) &= 0 ,
    &&\text{on }\Gamma_R^-,\\
     \nabla u_{0,0}^\delta \cdot \bn - \imath k_0
    u_{0,0}^\delta  &= 0 ,
    &&\text{on }\Gamma_R^+,\\
    \nabla u_{n,q}^\delta \cdot \bn - \imath k_0
    u_{n,q}^\delta &= 0, \quad (n,q) \not= (0,0),
    &&\text{on }\Gamma_R,\\
    \nabla u_{n,q}^\delta \cdot \bn &= 0, &&\text{on }\Gamma_N.
  \end{aligned}
\end{equation}
A priori, they are not continuous across~$\Gamma$ and
may become unbounded when approaching the corners
$\mathbf{x}_{O}^\pm$. Hence, the macroscopic terms are not entirely
defined: 
\begin{itemize}
\item[-] we first have to prescribe transmission conditions
across the interface $\Gamma$ (for instance the jump of their trace
and the jump of their normal trace across $\Gamma$). This information
will appear to be a consequence of the boundary layer equations (see
the Paragraph 'transmission conditions' below).
\item[-] we also have to prescribe the behaviour of the macroscopic terms
in the vicinity of the two corner points $\mathbf{x}_O^\pm$. %
This information will be given through the matching conditions and
will be provided through the iterative construction of the first terms
(see  Section~\ref{SectionMainMatching}, and
Section~\ref{sec:constr-first-terms}).
\end{itemize}

\textbf{Boundary layer corrector equations.} The boundary layer
correctors $\Pi_{n,q}^\delta(x_1,X_1,X_2)$ (also sometimes denoted as
{\em periodic correctors}) are assumed, as usual in the periodic
homogenization theory, to be $1$-periodic with respect to the scaled
tangential variable $X_1$. They are defined in the infinite
periodicity cell
$\mathcal{B} = \left\{ (0,1) \times \R \right\} \setminus
\overline{\widehat{\Omega}_{\hole}}$
(\cf~Fig.~\ref{fig:PeriodicityCell}) and satisfy
\begin{equation}
  \label{PeriodicCorrectorEquations}
  \left\lbrace
    \begin{aligned}
      - \Delta_{\bX} \Pi_{n,q}^\delta(x_1,\bX)  & = &
      F_{n,q}^\delta(x_1,\bX) & \quad \mbox{in} \;
      \mathcal{B}, \\
      \partial_{\bn} \Pi_{n,q}^{\delta}(x_1,\bX) &  = & \;\; - \partial_{x_1} \Pi_{n,q-1}^{\delta}(x_1,\bX)  \mathbf{e}_1 \cdot \bn & \quad \mbox{on} \;  \partial \widehat{\Omega}_{\hole},
    \end{aligned}
  \right.
\end{equation}
in which $\bn$ denotes the normal vector on
$\partial \widehat{\Omega}_{\hole}$. The source terms
$F_{n,q}^\delta$, depending on the macroscopic terms $u_{n,p}^\delta$
for $p\leq q$ (see
\ref{SectionTransmissionConditionsBoundaryLayer}), are given
by
\begin{align*}
  F_{n,0}^\delta(x_1, \mX) &= \sum_{\pm} u_{n,0}^\delta(x_1,0^\pm) \chi_\pm''(X_2), \\
  F_{n,1}^\delta(x_1, \mX) &= \sum_{\pm} \Big\{ \partial_{x_2} u_{n,0}^\delta(x_1,0^\pm) (2 \chi'_\pm(X_2) + X_2
    \chi''_\pm+(X_2)) \\ & \quad + u_{n,1}^\delta(x_1,0^\pm) \chi''_\pm(X_2) \Big\},
\end{align*}
where the cut-off function $\chi_+$
(resp. $\chi_-$) is the restriction of $\chi$ for $t \in \IR_+$
(resp. $t \in \IR_-$),~\ie
\begin{equation}
  \label{eq:defchipm}
  \chi_\pm(t) = \chi(t) \mathds{1}_{\IR_\pm}(t).
\end{equation}

\noindent In addition,
the periodic correctors are required to be super-algebraically decaying
as the scaled variable
$X_2$ tends to $\pm \infty$ (they decay faster than any power of $X_2$). More precisely, for any $(k, \ell) \in \N^2$, we impose that 
\begin{equation}
  \label{eq:exponential_decaying}
  \lim_{|X_2| \rightarrow + \infty}
  X_2^k \partial_{X_2}^\ell
  \Pi_{n,q}^\delta =0.
\end{equation}

\textbf{Transmission conditions.}
Enforcing the decaying condition~\eqref{eq:exponential_decaying} leads
to the missing transmission conditions for the macroscopic terms
$u_{n,q}^\delta$ on $\Gamma$. The complete procedure to obtain these
transmission conditions is classical and is fully described in
Section~\ref{SectionTransmissionConditionsBoundaryLayer}.  In this paragraph,
we restrict ourselves to the statement of the results. To do so, we
introduce the definition of the jump and mean values of a function $u$
across $\Gamma$ (for a sufficiently smooth function $u$ defined in a
vicinity of~$\Gamma$):
\begin{align}\label{definitionMeanJump}
  \xjump{u}_{\Gamma}(x_1) = \lim_{h\rightarrow 0^+}
  \left(  u(x_1, h) -  u(x_1, -h) \right), 
  \xavrg{u}_{\Gamma}(x_1) = \frac{1}{2} \lim_{h\rightarrow 0^+
  } \left(  u(x_1,h) +  u(x_1, -h) \right).
\end{align}

\noindent For $n \in \{0,1,2,3\}$, we obtain that the terms $u_{n,0}^\delta$ do not jump across $\Gamma$,  \emph{i.e.}
\begin{equation}\label{eq:macroscopic_u_n_0}
\left[ u_{n,0}^\delta \right]_{\Gamma} = \left[ \partial_{x_2}  u_{n,0}^\delta \right]_{\Gamma} = 0 \quad \mbox{on} \, \Gamma.
\end{equation}
By contrast, for $n \in \{0,1,2\}$, the terms $u_{n,1}^\delta$ satisfy non-homogeneous jump conditions: 
 \begin{equation}
  \label{eq:macroscopic_u_n_1}
  \left \lbrace
   \begin{array}{r@{\;}ll}
      \left[ u_{n,1}^\delta \right]_{\Gamma} &=  \mathcal{D}_{1}
      \, \partial_{x_1}  \langle u_{n,0}^\delta \rangle_{\Gamma}   \, + \, \mathcal{D}_{2}
      \, \langle \partial_{x_2} u_{n,0}^\delta \rangle_{\Gamma}&\quad \mbox{on} \,
      \Gamma ,\\[0.5em]
      \left[ \partial_{x_2} u_{n,1}^\delta \right]_{\Gamma} & = \mathcal{N}_{1} \, \langle
      u_{n,0}^\delta \rangle_{\Gamma}  + \mathcal{N}_{2} \, \partial_{x_1}^2 \langle
      u_{n,0}^\delta \rangle_{\Gamma} + \mathcal{N}_{3} \, \partial_{x_1} \langle \partial_{x_2}
      u_{n,0}^\delta \rangle_{\Gamma}   &\quad \mbox{on} \, \Gamma.
    \end{array}
  \right.
\end{equation}
Here, the quantities $ \mathcal{D}_{i}$ ($i \in \{ 1,2\}$) and
$\mathcal{N}_{i}$ ($i \in \{ 1,2,3\}$), defined by
\eqref{Sautu1}-\eqref{Sautdu1} are complex-valued constants coming
from the periodicity cell
problems~\eqref{PeriodicCorrectorEquations}. They only depend on
$\hat{k}$ and on the geometry of the periodicity cell.

\subsubsection{Near field expansions}
Let us now describe the asymptotic expansion valid in the  two near field zones, namely in the vicinity of the two reentrant corners $\mx_{O}^\pm$ (dark gray areas
in~Fig.~\ref{DessinAsymptotic}).  In these areas, the solution varies rapidly in all
directions. 
Therefore, we shall see that
\begin{multline}
  \label{NFExpansion}
  u^\delta(\mx) = U_{0,0,\pm}^\delta\left(\frac{\mathbf{x}-\mathbf{x}_{O}^\pm}{\delta}\right) + \delta^{\lambda_1} \,
  U_{1,0,\pm}^\delta\left(\frac{\mathbf{x}-\mathbf{x}_{O}^\pm}{\delta}\right) + \delta \,
  U_{0,1,\pm}^\delta\left(\frac{\mathbf{x}-\mathbf{x}_{O}^\pm}{\delta}\right) \\ + \delta^{\lambda_2} \,
  U_{2,0,\pm}^\delta\left(\frac{\mathbf{x}-\mathbf{x}_{O}^\pm}{\delta}\right) + \delta^{\lambda_1+1} \,
  U_{1,1,\pm}^\delta\left(\frac{\mathbf{x}-\mathbf{x}_{O}^\pm}{\delta}\right) + \delta^{\lambda_3} \,
  U_{3,0,\pm}^\delta\left(\frac{\mathbf{x}-\mathbf{x}_{O}^\pm}{\delta}\right) + O(\delta^2)
\end{multline}
for some near field terms $U_{n,q,\pm}^\delta$ defined in the fixed unbounded domains 
\begin{equation}
  \label{definitionOmegaHatpm}
  \widehat{\Omega}^- =\mathcal{K}^- \setminus  \bigcup_{\ell \in \N}
  \left\{ \overline{\widehat{\Omega}_{\hole}} + \ell \be_{1} \right\},
  \quad  \widehat{\Omega}^+ =\mathcal{K}^+ \setminus  \bigcup_{\ell \in \N^\ast}
  \left\{  \overline{\widehat{\Omega}_{\hole}} - \ell \be_{1} \right\}
\end{equation}
shown in Figure~\ref{fig:HatOmegaMoins} and \ref{fig:HatOmegaPlus},
where $\mathcal{K}^\pm$ are the conical domains 
\begin{equation}
  \label{eq:definition_mathcal_K}
  \mathcal{K}^\pm = \left\{ \mathbf{X} = R^\pm  ( \cos \theta^\pm, \sin \theta^\pm),  R^\pm\in \R^\ast_+, \theta^\pm \in I^\pm\right\} \subset\, \R^2
\end{equation}
of angular sectors $I^+ = (0, \Theta)$ and $I^- = (\pi - \Theta, \pi)$. The domains  $\widehat{\Omega}^\pm$ consist of the angular domains  $\mathcal{K}^\pm$ minus a infinite half line of equi-spaced similar canonical obtacles. In particular, 
if the domain $\widehat{\Omega}_{\hole}$ is symmetric with respect to the
axis $X_1 = 1/2$, then the domain $\widehat{\Omega}^-$ is nothing but the domain
$\widehat{\Omega}^+$ mirrored with respect to the axis $X_1=0$.
However, this is not the case in general. \\

Similarly to the far field terms the
 near field terms $U_{n,q,\pm}^\delta$ might also have a polynomial dependence with respect to $\ln
\delta$.\\
\begin{figure}[htbp]
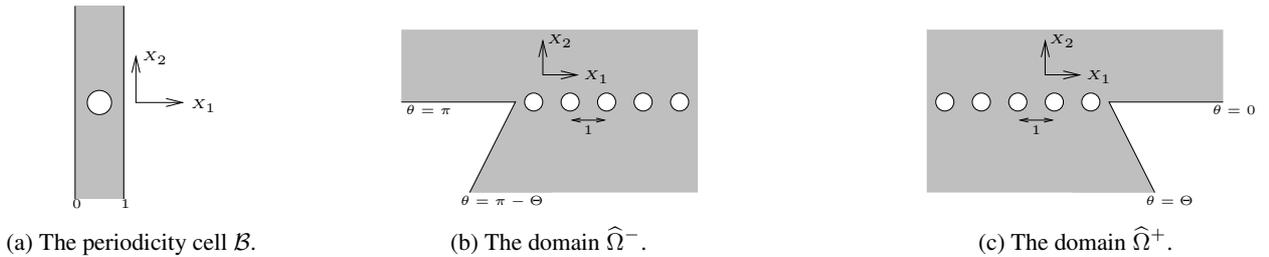

  \centering
  \begin{subfigure}[b]{0.22\textwidth}
    \centering
    \input{domain_canonical.pspdftex}% Exported to 80%
    \caption{The periodicity cell $\mathcal{B}$.}
    \label{fig:PeriodicityCell}
  \end{subfigure}
  \begin{subfigure}[b]{0.38\textwidth}
    \centering
    \input{domain_nf_omega_minus.pspdftex}% Exported to 80%
    \caption{The domain $\widehat{\Omega}^-$.}
    \label{fig:HatOmegaMoins}
  \end{subfigure}
  \begin{subfigure}[b]{0.38\textwidth}
    \centering
    \input{domain_nf_omega_plus.pspdftex}% Exported to 80%
    \caption{The domain $\widehat{\Omega}^+$.}
    \label{fig:HatOmegaPlus}
  \end{subfigure}
  \caption{The periodicity cell $\mathcal{B}$ and the normalized
    domains $\widehat{\Omega}^\pm$.}\label{fig:dessinsDomaines2}
\end{figure}

 Inserting the near field ansatz~\eqref{NFExpansion} into the Helmholtz
equation~\eqref{eq:perturbed_Helmholtz} and separating
formally the different powers of $\delta$, it is
easily seen that the near field term $U_{n,q}^\delta$ satisfies
\begin{equation}\label{NearFieldEquation}
\left \lbrace
\begin{aligned}
  - \Delta_{\mX} U_{n,q,\pm}^\delta & = (\widehat{k}^\pm)^2(\bX)
  U_{n,q-2,\pm}^\delta & & \quad \mbox{in} \;
  \widehat{\Omega}^{\pm}, \\
  \partial_\bn U_{n,q,\pm}^\delta & = 0 & & \quad \mbox{on}\,
  \partial \widehat{\Omega}^\pm,
\end{aligned}
\right.
\end{equation}  
where the perturbed wave number $\widehat{k}^\pm(\bX)$ is given by
\begin{equation}
  \label{eq:perturbation_wavenumber}
  \widehat{k}^\pm(\bX) =
  \begin{cases}
    \widehat{k}(\bX) \; & \mbox{if } \pm X_1 < 0,\\
    k_0 \; & \mbox{otherwise}.
  \end{cases}
\end{equation}
Again, Equation~\eqref{NearFieldEquation} does not define $U_{n,q,\pm}^\delta$ entirely because its (possibly increasing) behaviour towards infinity is missing. %
This behaviour will be given through the matching conditions.

\subsubsection{Matching principle}\label{SectionMainMatching} To link the far and near fields expansions~\eqref{FFExpansion} and \eqref{NFExpansion}, we
assume that they are both valid in two intermediate areas
${\Omega}^{\delta,\pm}_\mathcal{M}$ (dark shaded in
Fig.~\ref{DessinAsymptotic}) of the following form:
\begin{equation}
  \label{eq:matching_areas}
  {\Omega}^{\delta,\pm}_{\mathcal{M}} = \left\{ \mathbf{x} = (x_1,x_2) \in \Omega^{\delta},
    \sqrt{\delta} \leq d(\mathbf{x}, \mathbf{x}_O^\pm) \leq 2\,\sqrt{\delta} \right\},
\end{equation}
where $d$ denotes the usual Euclidian distance. The reader might just
keep in mind that they correspond to a neighborhood of the corners
$\mathbf{x}_{O}^\pm$ of the re-entrant corners for the far field terms
(macroscopic and boundary layer correctors) and to a neighborhood
of infinity, \ie, $R^\pm \to \infty$, for the near field terms
(expressed in the scaled variables). \\

\noindent In practice, for a given order $N_0\geq0$,  we make a \emph{formal}
identification between \eqref{FFExpansion} and \eqref{NFExpansion}: 
\begin{equation}
  \label{eq:matching_conditions_global_main}
  \sum_{\lambda_n+q < N_0} \delta^{\lambda_n+q} \, u_{\text{FF},n,q}^\delta
  (\mx) \approx \sum_{\lambda_n+q < N_0}
  \delta^{\lambda_n+q} \, U_{n,q,\pm}^\delta 
  \left(\frac{\mx-\mx_{O}^\pm}{\delta}\right). % +O(\delta^{N_0}) \, .
\end{equation}
The previous relation can be seen in two different scales (the
macroscopic scale and the near field scale) and will relate, on the
one hand, the regular part of the far field terms to the increasing
behaviour of the near field terms, and, on the other hand, the
decreasing behaviour of the near field terms to the singular behaviour
of the far field terms. The matching will be conducted for the first
terms order by order in Section~\ref{sec:constr-first-terms}.
  
\begin{remark} A crucial point for the matching procedure is that we
  match only the far and near field expansions away from the layer,
  \ie, $\theta^- \not= 0$ and $\theta^+ \not= \pi$.  Indeed, thanks to
  the linearity of the canonical cell problem, the periodic correctors
  appear to be a by-product of the macroscopic terms (see
  Section~\ref{SectionTransmissionConditionsBoundaryLayer}). As a consequence,
  as soon as the two series match away from the layer, they also match
  in the vicinity of the layer (see~Section~\ref{SectionRaccordCouche}).
\end{remark}

\subsection{Main results}\label{SubsectionMainResults}
\label{sec:intr-sett-probl:outlook}

\subsubsection{Error estimates}
\label{sec:error_estimates_main}
Collectiong the macroscopic
problems~\eqref{FFVolum}-\eqref{eq:FF_macro_bc}-\eqref{eq:macroscopic_u_n_0}-\eqref{eq:macroscopic_u_n_1},
the boundary layer problems \eqref{PeriodicCorrectorEquations}, the
near field problems~\eqref{NearFieldEquation}, and the matching
conditions~\eqref{eq:matching_conditions_global_main} permits us to
define in step by step the first terms of the asymptotic expansion up
to order $2$ (see Section~\ref{sec:constr-first-terms}). Then, our
main theoretical result deals with the convergence of the truncated
macroscopic series in a domain that excludes the two corners and the
periodic thin layer:
\begin{theorem}[Error estimates of the truncated macroscopic expansion]
  \label{theo:error_estimate_optimal}
 Let $\Theta \in (\pi, 2\pi)$, and,  for a given number $\alpha > 0$, let
  \begin{equation*}
    \Omega_\alpha = \Omega^\delta \setminus (-\LBottom-\alpha,\LBottom+\alpha) \times
    (-\alpha,\alpha).
  \end{equation*}
  There exists a constant $\delta_0 > 0$, a constant $C >0$ and a  integer  $\kappa \in \{ 0, 1\}$ 
  such that for any
  $\delta \in (0, \delta_0)$,
  \begin{align}
    \label{eq:EstimationOptimale_0}
    \xnorm{u^\delta - u_{0,0}}{\Hone(\Omega_\alpha)} &\leqslant C \delta, \\
    \label{eq:EstimationOptimale_1}
    \xnorm{u^\delta - u_{0,0} - \delta u_{0,1}}{\Hone(\Omega_\alpha)} &\leqslant C \delta^{\lambda_2}.
  \end{align}
 and,  
  \begin{align}\label{EstimationOptimale}
   \xnorm{u^\delta - u_{0,0} - \delta u_{0,1} -
    \delta^{\lambda_2} u_{2,0}}{\Hone(\Omega_\alpha)} & \leqslant C \delta^2
    (\ln \delta)^\kappa, &  \mbox{if} \quad \Theta \leq \frac{3\pi}{2}, \\
  %If $\Theta \in ( \frac{3\pi}{2}, 2 \pi)$, there exists a constant $C_3 > 0$
    \xnorm{u^\delta - u_{0,0} - \delta u_{0,1} -
      \delta^{\lambda_2} u_{2,0} - \delta^{\lambda_3}
      u_{3,0}}{\Hone(\Omega_\alpha)} & \leqslant C \delta^2 & \mbox{if} \quad \Theta \in (\frac{3\pi}{2}, 2\pi).
  \end{align} 
\end{theorem}  
The proof of the previous theorem, although rather classical (see \eg
Chapter 4 in Ref.~\cite{LivreNazarov1}), is conducted in Section~\ref{SectionErrorEstimates}: it is based on the construction of an approximation global approximation (defined in~\eqref{eq:global_approximation})  of $u^\delta$  defined in the whole domain $\Omega^\delta$.

\subsubsection{Numerical justification}
\label{sec:numer-just}

We illustrate numerically the results of
Theorem~\ref{theo:error_estimate_optimal} using the finite elements
method with the numerical C++ library
Concepts\cite{conceptsweb,Frauenfelder.Lage:2002}. For both, the exact
and macroscopic problems, we rely on meshes geometrically refined
towards the corners and varying polynomial
degree\cite{Schwab1998,Schmidt.Kauf:2009}. We consider the geometry
sketched in the left part of Figure~\ref{fig:error_estimates} for
$\delta = 0.25$, for which the inner angle $\Theta = \frac{3\pi}{2}$
at the two corners $\mathbf{x}_{O}^\pm$.  The upper rectangle
representing a wave-guide is $\OmegaTop = (-2.5,2.5) \times (0,1)$ and
the lower one representing a chamber is
$\OmegaBottom = (-0.5,0.5) \times (-1,0)$.  The canonical hole
$\widehat{\Omega}^{\pm}_{\hole}$ is the disk centered at $(0.5,0)$
with diameter equal to $0.3$. We consider a homogeneous wave number
$k^\delta = k_0 = 5 \pi$.
In Figure~\ref{fig:error_estimates} we show the difference  between the
exact solution $u^\delta$ and the macroscopic expansion of different
order, using that $u_{1,0}=u_{1,1}=0$, in the
$L^2(\Omega_\alpha)$-norm for $\alpha=0.25$ as a function of $\delta$
where $\delta = 1/4, 1/8, ..., 1/128$. As might be expected, we exactly recover the convergence rate stated in Theorem~\ref{theo:error_estimate_optimal}.
%The numerical approximation
%verifies that the order of convergence of the macroscopic expansions
%stated in Theorem~\ref{theo:error_estimate_optimal} is sharp, and the error
%behaves in powers of $\delta$ as the first neglected term.
\begin{figure}[htbp!]
  \centering
    \begin{tikzpicture}
      \pgfplotsset{ legend style={
          at={(1.00,0.2)},
          font=\footnotesize, anchor=east, align=left, row sep=0.2cm} %
      } %
      \pgfmathparse{4}\let\n\pgfmathresult;
      \pgfmathparse{2. / \n}\let\dd\pgfmathresult;
      \pgfmathparse{0.15 * \dd}\let\r\pgfmathresult;
       
      % Rectangles
      \fill[gray!75!white, xshift=3.2cm, yshift=5.1cm,scale=0.5] (-5 ,0) rectangle (5 ,2);
      \fill[gray!75!white, xshift=3.2cm, yshift=5.1cm,scale=0.5] (-1,-2) rectangle (1, 0);
      \draw[black, xshift=3.2cm, yshift=5.1cm,scale=0.5] (-5,2) -- (5,2) -- (5,0) -- (1,0) -- (1,-2) -- (-1,-2) -- (-1,0) --(-5,0) -- (-5,2);
      \node[xshift=4.5cm, yshift=5.6cm] (0,0) {$\Omega^\delta$};
       
      % Interior circles
      \foreach \x in {1,...,\n}
       {            
 	\pgfmathparse{\dd*(\x-0.5)-1}\let\y\pgfmathresult;
 	\filldraw[fill=white, draw=black, xshift=3.2cm, yshift=5.1cm,scale=0.5] (\y,0) circle (\r);
      }

      \begin{axis}[width=0.9\textwidth, height=0.54\textwidth, %
        xmin=0.002, xmax=0.5, ymin=0.00005, ymax=1.5,
        xlabel=Distance $\delta$ between two consecutive holes,
        ylabel=Modelling error,
        xmode=log,
        ymode=log,
        cycle list name=mylist,
        legend style={draw=none}, legend cell align=left
        ]
        \addplot table [x expr = 1. / \thisrow{N}, y = error0] {homogeneous.dat};
        \addlegendentry{$\left\lVert u^\delta - u_{0,0} \right\rVert_{L^2(\Omega_{0.25})}$};
        \addplot table [x expr = 1. / \thisrow{N}, y = error1] {homogeneous.dat};
        \addlegendentry{$\left\lVert u^\delta - u_{0,0} - \delta
            u_{0,1} \right\rVert_{L^2(\Omega_{0.25})}$};
        \addplot table [x expr = 1. / \thisrow{N}, y = error2] {homogeneous.dat};
        \addlegendentry{$\left\lVert u^\delta - u_{0,0} - \delta
            u_{0,1} - \delta^{4/3} u_{2,0} \right\rVert_{L^2(\Omega_{0.25})}$};
        % \addplot[red!80!black,dashed] table [x expr = 1. / \thisrow{N}, y = error3] {datas/homogeneous.dat};
        % \addlegendentry{$\left\lVert u^\delta - u_{0,0} - \delta
        %     \tilde{u}_{0,1} \right\rVert_{L^2(\Omega_{0.25})}$};
        % \addplot[red!80!black,mark=+, only marks] table [x expr = 1. / \thisrow{N}, y = error3] {datas/homogeneous.dat};

        % Axis
        %  64 0.0135221  64 0.0078556  32 6.5547e-003
        % 256 0.0135221  64 0.0012993  32 4.5856e-004
        % 256 0.0039325 256 0.0012993 128 4.5856e-004

        % Axis 1 (slope 0.96) : 0.0076567 / 2^(0.96) = 0.003936
        \draw
        (axis cs:0.00390625,0.003936)
        |- (axis cs:0.0078125,0.0076567)
        node[near start,left]{{\small$0.96$}};
        % Axis 2 (slope 1.33)
        % Starting point : 
        %    x = 0.0078125
        %    y = 0.0029982
        % End point :
        %    x = 1/64*sqrt(2) = 0.022097
        %    y =  0.0029982*(2*sqrt(2))^1.33 = 0.011951
        \draw
        (axis cs:0.022097,0.011951)
        |- (axis cs:0.0078125,0.0029982)
        node[near start,right]{{\small$1.33$}};
        % Axis 3 (slope 1.92) : 1.7337e-3/4^(1.92) = 1.2447e-04
        \draw
        (axis cs:0.015625,1.7337e-3)
        |- (axis cs:0.0039062,0.00012107) 
        node[near start,right]{{\small$1.92$}};
        % Axis 1 (slope 0.96) : 0.0076567 / 2^(0.96) = 0.003936
%         \draw
%         (axis cs:0.00390625,0.003936)
%         |- (axis cs:0.0078125,0.0076567)
%         node[near start,left]{{\small$0.96$}};
      \end{axis}
    \end{tikzpicture}
  \caption{The numerically computed errors of macroscopic expansions 
    truncated at different orders in dependence of $\delta$. The computational domain $\Omega^\delta$ 
    is sketched for $\delta = 0.25$.
  }
  \label{fig:error_estimates}
\end{figure}
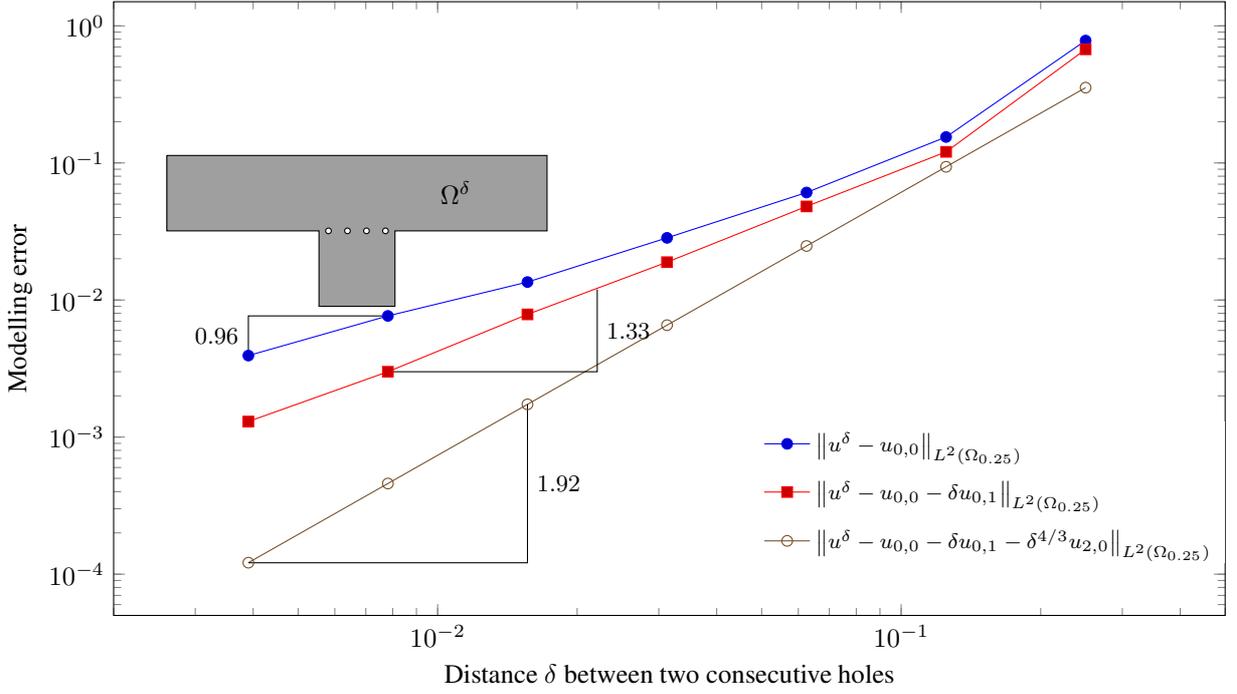

\section{Analysis of the far field problems: transmission problem,  boundary layer problems and derivation of the transmission
  conditions}
\label{SectionTransmissionConditionsBoundaryLayer}

This section is dedicated to the analysis of the far-field problems. In Section~\ref{sec:gener-results-exist} and Section~\ref{SubSectionPropBL}, we first recall the functional frameworks that will allow us to define the macroscopic terms and boundary layer terms. Then,  Section~\ref{SubSectionDerivationu0u1} is dedicated to the \emph{formal} derivation of the  transmission conditions~\eqref{eq:macroscopic_u_n_0}-\eqref{eq:macroscopic_u_n_1}
for the macroscopic fields $u_{n,q}^\delta$ across $\Gamma$.

\subsection{General results of existence for transmission problem}
\label{sec:gener-results-exist}
The macroscopic fields satisfy transmission problems of the following form (cf.~\eqref{FFVolum}-\eqref{eq:FF_macro_bc}-\eqref{eq:macroscopic_u_n_0}-\eqref{eq:macroscopic_u_n_1}):
\begin{equation}
    \label{eq:MacroscopicProblemModel}
    \left\lbrace
      \begin{aligned}
        - \Delta u - k_0^2 u &= \mathfrak{f} &\quad& \text{in } \OmegaTop \cup \OmegaBottom\ ,\\
        \left[ u \right]_{\Gamma} &= \mathfrak{g} &\quad&\text{on }\Gamma\ ,\\
        \left[ \partial_{x_2} u \right]_{\Gamma} &= \mathfrak{h}
        &\quad&\text{on }\Gamma\ , \\
        \nabla u\cdot \bn - \imath k_0 u^\delta &= \mathfrak{j} , &\quad& \text{on
        }\Gamma_R,\\
        \nabla u \cdot \bn &= 0, &\quad& \text{on }\Gamma_N.
      \end{aligned}
    \right.
  \end{equation}
To solve this transmission problem, we consider the space $\Hone(\OmegaTop \cup \OmegaBottom)$ defined by
\begin{equation}
  \label{eq:space_Hone_Omega}
  \Hone(\OmegaTop \cup \OmegaBottom) = \left\lbrace v \in
    \Ltwo(\Omega) \text{ such that } v_{|\OmegaTop} \in
    \Hone(\OmegaTop) \text{ and } v_{|\OmegaBottom} \in
    \Hone(\OmegaBottom) \right\rbrace,
\end{equation}
which incorporates discontinuous functions over $\Gamma$ (see
Figure~\ref{fig:Omega}). We denote by $\HonehalfG$ the restriction of
the trace of the functions $\Hone(\OmegaTop)$ to $\Gamma$. Naturally,
the space $\HonehalfG$ is also the restriction of the trace of the
functions of $\Hone(\OmegaBottom)$. We point out that general
transmission problems are investigated in~\cite{Nicaise} using the
Kondratev theory. In particular  the following well-posedness result
is proved (Theorem~3.4 and Theorem~3.5 in Ref.~\cite{Nicaise},
Proposition 3.6.1 in Ref.~\cite{TheseDelourme}). 
\begin{proposition}
  \label{prop:existence_uniqueness_macro}
  Let $\mathfrak{f} \in \Ltwo(\Omega)$, $\mathfrak{g} \in \HonehalfG$, $\mathfrak{h} \in \Ltwo(\Gamma)$,  and $\mathfrak{j} \in H^{-1/2}(\Gamma_R)$. Then, Problem~\eqref{eq:MacroscopicProblemModel}
  has a unique solution $u$ belonging to
  $\Hone(\OmegaTop \cup \OmegaBottom)$.
  \end{proposition}

\subsection{Existence and uniqueness result for the boundary layer problem}
\label{SubSectionPropBL}

\noindent The boundary layer correctors satisfy problems of the form (see~\eqref{PeriodicCorrectorEquations})
%\noindent Based on this functional framework, we consider the following problem: 
%find $\Pi \in \mathcal{V}^-(\mathcal{B})$ such that
\begin{equation}
  \label{CanoniqueBoundaryLayer}
  \left \lbrace
    \begin{array}{rcll}
      \dsp - \Delta_\mX \Pi & = & F & \quad \text{in } \mathcal{B}, \\[1ex]
      \dsp \partial_\bn \Pi & = & {G} &\quad \text{on } \partial \widehat{\Omega}_{\hole}, \\[1ex]
      \dsp \partial_{X_1} \Pi(0,X_2) & =&  \partial_{X_1} \Pi(1,X_2), & \quad X_2 \in \R.
    \end{array}
  \right. % F \in \left(\mathcal{V}^-(\mathcal{B})\right)' \quad  G \in L^2(\partial \widehat{\Omega}_{\hole}).
\end{equation}
together with the super-algebric decaying condition~\ref{eq:exponential_decaying}. 
In this section, we give a standard result of existence and uniqueness associated with this problem.  To do so, 
we introduce the two weighted Sobolev spaces
\begin{equation}
  \label{eq:prop_layer_existence_uniqueness_space}
  \mathcal{V}^\pm(\mathcal{B}) = \left\lbrace \Pi \in \mathrm{H}_{\text{loc}}^1(\mathcal{B}),
    \Pi(0,X_2) = \Pi(1,X_2),  \text{ and }  \left(\Pi\, w_e^\pm\right)  \in
    \Hone(\mathcal{B}) \right\rbrace, 
\end{equation}
where the weighting functions
$w_e^\pm(X_1,X_2)= \chi(X_2) \exp( \pm \frac{| X_2|}{2})$. The
functions of $\mathcal{V}^-(\mathcal{B})$ correspond to the periodic
(w.r.t. $X_1$) functions of $\Hone_{\text{loc}}(\mathcal{B})$ that grow
slower than $\exp( \frac{| X_2|}{2})$ as $X_2$ tends to $\pm \infty$.
By contrast, the functions of $\mathcal{V}^+(\mathcal{B})$ correspond
to the periodic functions of $\Hone_{\text{loc}}(\mathcal{B})$ decaying
faster than $\exp(- \frac{| X_2|}{2})$ as $X_2$ tends to $\pm \infty$.
Note also that
$\mathcal{V}^+(\mathcal{B}) \subset
\mathcal{V}^-(\mathcal{B})$.\\

%\noindent Based on this functional framework, we consider the following problem: 
%find $\Pi \in \mathcal{V}^-(\mathcal{B})$ such that
%\begin{equation}
%  \label{CanoniqueBoundaryLayer}
%  \left \lbrace
 %   \begin{array}{rcll}
 %     \dsp - \Delta_\mX \Pi & = & F & \quad \text{in } \mathcal{B}, \\[1ex]
 %     \dsp \partial_\bn \Pi & = & {G} &\quad \text{on } \partial \widehat{\Omega}_{\hole}, \\[1ex]
 %     \dsp \partial_{X_1} \Pi(0,X_2) & =&  \partial_{X_1} \Pi(1,X_2), & \quad X_2 \in \R.
 %   \end{array}
 % \right. F \in \left(\mathcal{V}^-(\mathcal{B})\right)' \quad  G \in L^2(\partial \widehat{\Omega}_{\hole}).
%\end{equation}
\noindent As soon as $F \in \left(\mathcal{V}^-(\mathcal{B})\right)' $ and $G \in L^2(\partial \widehat{\Omega}_{\hole})$, it is known that Problem~\eqref{CanoniqueBoundaryLayer} has (several)
solutions in $\mathcal{V}^-(\mathcal{B})$ (cf Proposition~2.2 of
Ref.~\cite{Nazarov205} and Section~5 of
Ref.~\cite{Claeys.Delourme:2013}).  More specifically,
Problem~\eqref{CanoniqueBoundaryLayer} has a finite dimensional kernel
of dimension $2$, spanned by the functions
$\mathcal{N} = \mathds{1}_{\mathcal{B}}$ and $\mathcal{D}$, where
$\mathcal{D}$ is the unique harmonic function of
$\mathcal{V}^-(\mathcal{B})$ such that there exists
$\mathcal{D}_\infty \in \R$ such that
$$\widetilde{\mathcal{D} }(X_1, X_2) = \mathcal{D}(X_1, X_2)-\chi_+(X_2) ( X_2 + 
\mathcal{D}_\infty) - \chi_-(X_2) (X_2 - \mathcal{D}_\infty) $$
belongs to $\mathcal{V}^+(\mathcal{B})$ ($\chi_\pm$ defined
by~\eqref{eq:defchipm}). \\

\noindent The following proposition provides necessary
and sufficient conditions for the existence of an exponentially
decaying solution (see also Proposition~2.2 of
Ref.~\cite{Nazarov205} and Section~5 of
Ref.~\cite{Claeys.Delourme:2013} for the proof):
\begin{proposition}
  \label{prop:layer_existence_uniqueness_problem_strip}
  Assume that $F \in \left(\mathcal{V}^-(\mathcal{B})\right)' $ and $G \in L^2(\partial \widehat{\Omega}_{\hole})$. Problem~\eqref{CanoniqueBoundaryLayer} has a unique solution $\Pi \in \mathcal{V}^+(\mathcal{B})$ if and only if ${(F,G)}$   satisfies the following two conditions \begin{align}
      \label{eq:prop_layer_existence_uniqueness_compatibility_D}
      \tag{$\mathcal{C}_{\mathcal{D}}$}
      \int_{\mathcal{B}} F(\mX) \mathcal{D}(\mX) d\mX
      {+ \int_{\partial \widehat{\Omega}_{\hole}} G(\mX) \mathcal{D}(\mX) d\sigma(\mX)}& = 0, \\
      \label{eq:prop_layer_existence_uniqueness_compatibility_N}
      \tag{$\mathcal{C}_{\mathcal{N}}$}
      \int_{\mathcal{B}} F(\mX) \mathcal{N}(\mX) d\mX
      {+ \int_{\partial \widehat{\Omega}_{\hole}} G(\mX) \mathcal{N}(\mX) d\sigma(\mX)} & = 0.
    \end{align}
  
 % \item Conversely, if problems (\ref{CanoniqueBoundaryLayer}) admits
%    a solution $\Pi \in \mathcal{V}^+(\mathcal{B})$, then it satisfies
%    the compatibility conditions
%    \eqref{eq:prop_layer_existence_uniqueness_compatibility_D},
%    \eqref{eq:prop_layer_existence_uniqueness_compatibility_N}.
\end{proposition}
\subsection{Derivation of the boundary layer correctors problems and the transmission conditions for the macroscopic problems}
\label{SubSectionDerivationu0u1}
The previous framework will allows us to derive \emph{formally} the transmission conditions~\eqref{eq:macroscopic_u_n_0}-\eqref{eq:macroscopic_u_n_1}
for the macroscopic fields $u_{n,q}^\delta$ across $\Gamma$. This procedure
turns out to be independent of the index $n$ and of the superscript $\delta$ (of
$u_{n,q}^\delta$) so that we shall omit the index $n$ and the
superscript $\delta$ in this section. To do so, we completely ignore the corners $\mathbf{x}_{O}^\pm$ and we proceed as if the periodic layer were infinite.  For a given  $a \in (0, \LBottom)$, we restrict the domain $\Omega^\delta$ to
$  \Omega^\delta_a = \lbrace \mx \in \Omega^\delta \text{ such that }
  \lvert x_1 \rvert < a \rbrace$,
and we call $\Omega_a$ the limit domain
as $\delta \to 0$, \ie$
  \label{eq:domain_Omega_a}
  \Omega_a = \lbrace \mx \in \OmegaTop \cup \OmegaBottom \text{ such that }
  \lvert x_1 \rvert < a \rbrace$.
%\begin{figure}[h!]
%  \centering
%  \input{domain_equiv.pspdftex}
%  \caption{Local domain $\Omega^\delta_a$ considered for writing the
%    jump conditions along the layer.}
%  \label{DessinLocalLayer}
%\end{figure}
We start from a (given) term $u_0$ in $\Omega_a$ that is
solution of the homogeneous Helmholtz equation
$$
- \Delta u_0 - k_0^2 u_0 = 0 \quad  \mbox{in} \; \OmegaTop \cap \Omega_a \; \mbox{and} \; \OmegaBottom \cap \Omega_a.
$$ 
Then,  using
the method of homogenization\cite{PavliotisStuart2008}, we extend
$u_0$ to a function $v^\delta$ of the form 
$$\chi(x_2 / \delta) (u_0 + \delta u_1 + \delta^2 u_2) + (1-\chi(x_2 / \delta)) (\Pi_0 + \delta \Pi_1 + \delta^2 \Pi_2))$$
 that is defined in $\Omega^\delta_a$
and that satisfies the original Helmholtz problem 
\eqref{eq:perturbed_Helmholtz}  up to a given order (ignoring the lateral boundaries $\Omega_a^\delta$) :
 $$ - \Delta v^\delta - (k^\delta)^2 v^\delta \approx 0 \quad \mbox{in} \; \Omega_a^\delta  \quad \mbox{and}  \quad \partial_\mathbf{n} v^\delta  \approx 0
 \;\;\mbox{on} \; \Gamma^\delta \cap \partial \Omega_a^\delta.$$
 \begin{remark} The periodic boundary layer being considered as infinite, we point
 out that the following analysis is entirely
 classical\cite{RapportSanchezPalencia,ArtolaCessenat,AchdouCR,Ammari}. Moreover, we emphasize that the upcoming iterative
procedure is formal in the sense that we shall provide necessary
transmission conditions for the macroscopic terms $u_q$ (without
questioning their existence yet).

 \end{remark}

\subsubsection{Step 0: $[u_0]_\Gamma$  and $\Pi_0$}\label{SubsubStep0}

We start with the ansatz
\begin{equation}
  \label{eq:deriv_ansatz_0_ff}
  v^\delta(\mx) = u_0(\mx) \chi(x_2 / \delta), \quad \text{in } \Omega^\delta_a.
\end{equation}
The choice of the cut-off function $\chi(x_2 / \delta)$ is intended 
 %$  \lbrace \mx \in \Omega^\delta_a \text{ such that }\chi(x_2 / \delta)
%  \not= 0\rbrace = \lbrace \mx \in \Omega_a \text{ such that }\chi(x_2 / \delta)
%  \not= 0\rbrace$
%and,
since $k^\delta(\mx) = k_0$ on the support of $\chi( x_2 /
\delta)$. Reminding that $\big( - \Delta u_0 - k_0^2 u_0 \big)
\chi(x_2 / \delta)=0$, we see that 
\begin{equation}
  \label{eq:deriv_ansatz_0_equation}
     - \Delta v^\delta -\widehat{k}^2(\frac{\mx}{ \delta})  v^\delta  =-
    \frac{1}{\delta^2} u_0(\mx) \chi''(\frac{x_2}{\delta}) \; -\;
    \frac{2}{\delta} \xdrv{u_0}{x_2}(\mx) \chi'(\frac{x_2}{ \delta}) \quad \mbox{and} \quad \partial_\mathbf{n} v^\delta = 0\;\mbox{on} \; \Gamma^\delta \cap \partial \Omega_a^\delta.
\end{equation}
In \eqref{eq:deriv_ansatz_0_equation}, the leading
order term  is in $\delta^{-2}$ and is supported
in a vicinity of the limit interface $\Gamma_a = (-a,a) \times \lbrace 0
\rbrace$. To correct it, it is rational to add to $v^\delta$ an exponentially decaying  periodic corrector
$\Pi_0(x_1,\mx/\delta)$:
\begin{equation}
  \label{eq:deriv_ansatz_0_ff_with_bc}
  v^\delta(\mx) = u_0(\mx) \chi(x_2 / \delta) + \Pi_0(x_1,\mx/\delta) , \quad \text{in } \Omega^\delta_a
\end{equation}
We note that \begin{multline}
  \label{eq:deriv_ansatz_0_equation_with_bc}
    - \Delta v^\delta - \widehat{k}^2(\mx / \delta)  v^\delta  =
    \frac{1}{\delta^2} \left( - u_0(\mx) \chi''(\frac{x_2}{
      \delta}) - \Delta_{\mX} \Pi_0(x_1, \frac{\mx}{\delta}) \right) \\
     + \frac{1}{\delta} \left(- 2 \xdrv{u_0}{x_2}(\mx) \chi'(\frac{x_2}{ \delta})
    - 2 \partial_{x_1} \partial_{X_1} \Pi_0(x_1, \frac{\mx}{ \delta}) \right)
     - \partial_{x_1}^2 \Pi_0(x_1, \frac{\mx}{\delta}) - \widehat{k}^2(\mx / \delta)  \Pi_0(x_1, \frac{\mx}{ \delta}).
\end{multline}
Then, making the change of scale
 $\mX = \mx / \delta$ and using a Taylor expansion of
$u_0(x_1, \delta X_2)$ for $\delta$ small and for $X_2 \not= 0$, the leading term of order $\delta^{-2}$ vanishes if  $\Pi_0$ satisfies
\begin{equation}\label{ProblemPi0}
\left\lbrace
\begin{array}{rcll}
- \Delta_{\mX} \Pi_{0}(x_1, \mX) &= & F_0(x_1, \mX)  &\; \mbox{in} \; \mathcal{B}, \\
\partial_\bn \Pi_{0} & = & 0  &\; \mbox{on} \;  \partial \widehat{\Omega}_\hole,
\end{array}
\right. F_0(x_1, \mX) = \sum_{\pm}u_0(x_1,0^\pm) \chi_\pm''(X_2).
\end{equation}
Problem~\eqref{ProblemPi0} is a partial differential equation with
respect to the microscopic variables $X_1$ and $X_2$, wherein the
macroscopic variable $x_1$ plays the role of a parameter. For a fixed
$x_1$ in $(-a, a)$ (considered as a parameter),
$F_0(x_1, \cdot)$ belongs to $(\mathcal{V}^-(\mathcal{B}))'$ since it
is compactly supported.  Then, in view of
Proposition~\ref{prop:layer_existence_uniqueness_problem_strip}, there
exists an exponentially decaying solution
$\Pi_{0}(x_1, \cdot) \in \mathcal{V}^+(\mathcal{B})$ if and only if
the two compatibility conditions
(\ref{eq:prop_layer_existence_uniqueness_compatibility_D},
\ref{eq:prop_layer_existence_uniqueness_compatibility_N})
(Prop.~\ref{prop:layer_existence_uniqueness_problem_strip}) are
satisfied.  The condition
(\ref{eq:prop_layer_existence_uniqueness_compatibility_N}) is always
satisfied while the condition
(\ref{eq:prop_layer_existence_uniqueness_compatibility_D}) gives
$
  \left[ u_0 \right]_{\Gamma_a}(x_1) =0$.
Taking formally in this relation the limit $a = \LBottom$ gives
\begin{equation}\label{Sautu0}
  \left[ u_0 \right]_{\Gamma}(x_1) =0.
\end{equation}
The previous equality provides a first transmission condition for the limit
macroscopic term $u_0$ (a transmission condition for $\left[ \partial_{\mathbf{n}}u_0 \right]_{\Gamma}$ is still needed). In addition, under the previous condition,
$F_0(x_1, \mX) = \chi''(X_2) \langle u_0 \rangle_{\Gamma}(x_1)$, and,
using the linearity of Problem~\eqref{ProblemPi0}, we can obtain a
tensorial representation of
$\Pi_0$,
in which macroscopic and microscopic variables are separated:
\begin{equation}\label{decompositionPi0}
\Pi_0(x_1, \mX) =   \langle u_0 \rangle_{\Gamma}(x_1) \,  V_{0}(\mX).
\end{equation}
Here the profile function $V_{0}(\mX)$ is the unique function of
$\mathcal{V}^+(\mathcal{B})$ satisfying  
 \begin{equation}\label{ProblemW0t}
\left \lbrace
\begin{array}{rcll}
- \Delta_{\mX} V_{0}(\mX) &= & F_{V_0}(\mX)  &\; \mbox{in} \; \mathcal{B}, \\
\partial_\bn V_{0} &= & 0  &\; \mbox{on} \;  \partial
\widehat{\Omega}_\hole,\\
\partial_{X_1} V_{0} (0,X_2) &= &\partial_{X_1} V_{0}(1,X_2),& \; X_2 \in
\R,
\end{array} 
\right. \quad F_{V_0}(\mX) = \chi''(X_2).\\
\end{equation}
A direct calculation shows that $
V_{0}(\mX) = 1  - \chi(X_2)$.
%Finally, inserting the expression of $\Pi_0$ into
%\eqref{eq:deriv_ansatz_0_equation_with_bc}, we get
%\begin{equation}
%  \label{eq:deriv_ansatz_0_equation_with_bc2}
%  \begin{aligned}
%    - \Delta v^\delta - \big( \widehat{k}(\mx / \delta) \big)^2 v^\delta & =
%    \frac{1}{\delta^2} \left( u_0(x_1,0) - u_0(\mx) \right) \chi''(x_2
%    /
%    \delta) \\
%   & - \frac{2}{\delta} \partial_{x_2} u_0(\mx) \chi'(x_2 / \delta)
 %   - (\partial_{x_1}^2 +k_0^2) u_0(x_1,0) (1  - \chi(x_2 / \delta)) .
 % \end{aligned}
%\end{equation}
%The first line of the right-hand side in
%\eqref{eq:deriv_ansatz_0_equation_with_bc2} can be written, using a
%Taylor expansion of $u_0$ for $x_2$ small and not equal to~0:
%\begin{equation}
%  \label{eq:deriv_ansatz_0_equation_with_bc2_l1}
%  u_0(x_1,0) - u_0(\mx) = - \partial_{x_2} u_0(x_1,0^\pm) x_2
%  - \partial_{x_2}^2 u_0(x_1,0^\pm) \frac{x_2^2}{2} + o(x_2^2)
%\end{equation}
%Using the variable change $x_2 = \delta X_2$, we can rewrite relation
%\eqref{eq:deriv_ansatz_0_equation_with_bc2_l1} under the form
%\begin{equation}
%  \label{eq:deriv_ansatz_0_equation_with_bc2_l1_2}
%  u_0(x_1,0) - u_0(\mx) = - \delta \partial_{x_2} u_0(x_1,0^\pm) X_2
 % - \delta^2 \partial_{x_2}^2 u_0(x_1,0^\pm) \frac{X_2^2}{2} + o(\delta^2)
%\end{equation}
%so that the next term we have to add in $v^\delta$ is of order $\delta$.

\subsubsection{Step 1: $[\partial_{x_2}
  u_0]_\Gamma$, $[u_1]_\Gamma$, and $\Pi_1$}\label{SubsubStep1}

By definition of $\Pi_0$, the leading part in the right hand side of~\eqref{eq:deriv_ansatz_0_equation_with_bc} is of order $\delta^{-1}$.
 To cancel these terms,  we correct $v^\delta$ defined by \eqref{eq:deriv_ansatz_0_ff_with_bc},
adding a first order corrector, \emph{both} in a vicinity of the layer
and away from the layer:
\begin{equation}
  \label{eq:deriv_ansatz_1_ff_with_bc}
    v^\delta(\mx)  = u_0(\mx) \chi(x_2 / \delta) +
    \Pi_0(x_1,\mx/\delta) 
     + \delta u_1(\mx) \chi(x_2 / \delta) + \delta \Pi_1(x_1,\mx/\delta),
    \quad \text{in } \Omega^\delta_a.
\end{equation}
Adding the term $\Pi_1$ is natural (indeed, the remaining term in
\eqref{eq:deriv_ansatz_0_equation_with_bc} is located in the vicinity of the interface $\Gamma$. It is of order $1/\delta$,
that can be seen as $\delta$ (order of the remaining term) times
$\delta^{-2}$ (order of differentiation after the change of scale)). By contrast, the addition of the term $u_1$ might be surprising but
appears to be mandatory to ensure the exponential decay of $\Pi_1$. Then,
\begin{equation}
  \label{eq:deriv_ansatz_1_equation_with_bc}
 \begin{aligned}
    - \Delta v^\delta -  \widehat{k}^2(\frac{\mx}{\delta})  v^\delta  =   &  -
    \delta (\Delta u_1 + k_0^2 u_1) \chi\left(\frac{x_2 }{\delta}\right)   +
    \frac{1}{\delta^2} \left( u_0(x_1,0) - u_0(\mx) \right) \chi''\left(\frac{x_2 }{\delta}\right) \\
   &  - \frac{2}{\delta} \partial_{x_2} u_0(\mx) \chi'\left(\frac{x_2 }{\delta}\right) -
    \frac{1}{\delta} u_1(\mx) \chi''\left(\frac{x_2 }{\delta}\right) - \frac{1}{\delta}
    \Delta_{\mX} \Pi_1(x_1, \frac{\mx}{\delta}) 
    \\ &  + \left(\partial_{x_1}^2+\widehat{k}^2(\frac{\mx}{\delta})\right)
    u_0(x_1,0) (1 - \chi\left(\frac{x_2 }{\delta}\right)) - 2 \partial_{x_2} u_1(\mx)
    \chi'\left(\frac{x_2 }{\delta}\right) \\&   - 2 \partial_{x_1} \partial_{X_1}
    \Pi_1(x_1, \frac{\mx}{ \delta}) 
     - \delta \left(\partial_{x_1}^2+ \widehat{k}^2(\frac{\mx}{\delta})\right) \Pi_1(x_1, \frac{\mx}{ \delta}).
  \end{aligned}
\end{equation}
and
\begin{equation}\label{NeumannOrdre1}
 \partial_\mathbf{n} v^\delta(\delta \mathbf{X}) =  \partial_\mathbf{n} \Pi_1(x_1, \mathbf{X}) + \partial_{x_1} \langle u_0(x_1, 0) \rangle \mathbf{e}_1\cdot \mathbf{n}.
\end{equation}
For a given $\mx$ such that $x_2 \not = 0$, dividing
  \eqref{eq:deriv_ansatz_1_equation_with_bc} by $\delta$ and taking
  the limit as $\delta \to 0$ in~\eqref{eq:deriv_ansatz_1_equation_with_bc} leads to
  \begin{equation}
  \label{eq:deriv_ansatz_1_helm}
  - \Delta u_1 - k_0^2 u_1 = 0 \quad \mbox{in} \;\OmegaTop \cap \Omega_a \; \mbox{and} \; \OmegaBottom \cap \Omega_a
\end{equation}
Indeed, the terms that contains $1-\chi$, $\chi'$ and $\chi''$ are compactly supported
and vanish for $\lvert x_2 \rvert > 2 \delta$, and, by assumption the terms related
to $\Pi_1$ are exponentially decaying towards
$x_2 / \delta \to \infty$.  To defined $\Pi_1$, we make the change of scale $\mathbf{X} = \frac{\mathbf{x}}{\delta}$ in~\eqref{eq:deriv_ansatz_1_equation_with_bc} (using Taylor expansions of $u_0$ and $u_1$ in the vicinity of $\Gamma$) and we enforce the term in $\delta^{-1}$ in~\eqref{eq:deriv_ansatz_1_equation_with_bc} to vanish. Together with the Neumann boundary condition~\eqref{NeumannOrdre1}, it is rational to construct $\Pi_1$ as a solution to
\begin{equation}\label{ProblemPi1}
\left \lbrace
\begin{array}{rcll}
  \dsp - \Delta_{\mX} \Pi_{1}(x_1, \mX) & = & F_1(x_1, \mX)  & \; \mbox{in} \; \mathcal{B}, \\
  \dsp \partial_\bn \Pi_{1} & = & {-\partial_{x_1} \langle u_0(x_1, 0) \rangle  \, \be_1
                                  \cdot \bn}  &\; \mbox{on} \;  \partial
  \widehat{\Omega}_\hole, \\
  \partial_{X_1} \Pi_1 (0,X_2) &= &\partial_{X_1} \Pi_1(1,X_2),& \; X_2 \in
  \IR,
\end{array} 
\right.
\end{equation}
where
\begin{equation}\label{G1}
  F_1(x_1, \mX) = \sum_{\pm}  \left( \partial_{x_2} u_0(x_1,0^\pm) (2 \chi'_\pm(X_2) + X_2
  \chi''_\pm(X_2))  +  u_1(x_1,0^\pm) \chi''_\pm(X_2) \right) .
\end{equation}
As for $\Pi_0$, Problem~\eqref{ProblemPi1} is a partial differential equation
with respect to the microscopic variables $X_1$ and $X_2$, where the
macroscopic variable $x_1$ plays the role of a parameter. For a fixed $x_1$ in
$(-\LBottom, \LBottom)$, $F_1(x_1, \cdot)$ is compactly supported in
$\mathcal{B}$, and, consequently, belongs to $(\mathcal{V}^-(\mathcal{B}))'$.
Then, thanks to
Proposition~\ref{prop:layer_existence_uniqueness_problem_strip}, there exists
an exponentially decaying solution $\Pi_1(x_1,\cdot) \in
\mathcal{V}^+(\mathcal{B})$ if and only if the two compatibility
conditions~(\ref{eq:prop_layer_existence_uniqueness_compatibility_D},
\ref{eq:prop_layer_existence_uniqueness_compatibility_N}) are satisfied.
A direct calculation shows that the compatibility condition
\eqref{eq:prop_layer_existence_uniqueness_compatibility_N} is fulfilled if and
only if
\begin{equation}\label{Sautdu0}
  \left[\partial_{x_2} u_0 \right]_\Gamma(x_1) = 0,
\end{equation}
and the compatibility condition
\eqref{eq:prop_layer_existence_uniqueness_compatibility_D} is fulfilled if and
only if
\begin{multline}\label{Sautu1}
  \left[ u_1 \right]_\Gamma(x_1) = \mathcal{D}_1 \,
  \partial_{x_1} \langle u_0 \rangle_{\Gamma}(x_1) +   \mathcal{D}_2
  \, \langle \partial_{x_2} u_0 \rangle_{\Gamma}(x_1),  \\  \mathcal{D}_1 = - \int_{\partial
  \widehat{\Omega}_\hole} \mathcal{D} \be_1 \cdot \bn , \; \mathcal{D}_2 = \int_{\mathcal{B}} (2 \chi'(X_2) + X_2 \chi''(X_2)) \mathcal{D}. 
\end{multline}
Under the two conditions~\eqref{Sautdu0}-\eqref{Sautu1},
Problem~\eqref{ProblemPi1} has a unique solution $\Pi_1 \in
  \mathcal{V}^+(\mathcal{B})$ that can be written as 
  \begin{equation}\label{decompositionPi1}
  \Pi_1(x_1, \mX) = \langle u_1 \rangle_\Gamma(x_1)\,
  V_{0}(\mX)  \; {+\; \partial_{x_1} \langle u_0 \rangle_\Gamma(x_1) \, V_{1,1}(\mX)}  \;+ \; \langle \partial_{x_2} u_0
  \rangle_\Gamma(x_1) \, V_{1,2}(\mX) .
\end{equation}
Here, $V_{1,1} \in \mathcal{V}^+(\mathcal{B})$ and $V_{1,2} \in \mathcal{V}^+(\mathcal{B})$ are the unique
exponentially decaying solutions to the following problems:
\begin{equation}\label{ProblemW1t}
  \left \lbrace
  \begin{array}{rcll}
    - \Delta_{\mX} V_{1,1}(\mX) & = & \mathcal{D}_1
                                                   \frac{
                                                   \chi_+''(X_2) -
                                                   \chi_-''(X_2) }{2}
                                                    & \quad \text{in } \mathcal{B}, \\
    \partial_\bn V_{1,1} &= & - \be_1 \cdot \bn & \quad \text{on } \partial
                                                               \widehat{\Omega}_\hole,\\
    \partial_{X_1} V_{1,1} (0,X_2) &= & \partial_{X_1}
                                                     V_{1,1}(1,X_2), &\quad X_2 \in \R, 
    \end{array}
  \right. 
\end{equation}

\begin{equation}\label{ProblemW1n}
  \left \lbrace
  \begin{array}{rcll}
   % \begin{aligned}
      - \Delta_{\mX} V_{1,2}(\mX) & = & F_{V_{1,2}}  + \mathcal{D}_2 \frac{\chi_+''(X_2) -
      \chi_-''(X_2))}{2} & \quad \text{in } \mathcal{B}, \\
      \partial_\bn V_{1,2} &= & 0 & \quad \text{on } \partial
      \widehat{\Omega}_\hole,\\
      \partial_{X_1} V_{1,2} (0,X_2) &= & \partial_{X_1}
      V_{1,2}(1,X_2), &\quad X_2 \in \R, 
    % \end{aligned}
    \end{array}   \right. F_{V_{1,2}}= 2 \chi'(X_2) + X_2 \chi''(X_2)
 \end{equation}

\subsubsection{Step 2: $[\partial_{x_2}
  u_1]_\Gamma$ ($[u_2]_\Gamma$ and $\Pi_2$)}
\label{SubsubStep2}

To define completely  $u_1$, we need to go one order further into the asymptotic expansion. We then correct $v^\delta$ defined by \eqref{eq:deriv_ansatz_1_ff_with_bc},
adding a second order corrector:
\begin{equation}
  \label{eq:deriv_ansatz_2_ff_with_bc}
  \begin{aligned}
    v^\delta(\mx) & = u_0(\mx) \chi(\frac{x_2}{\delta}) +
    \Pi_0(x_1,\frac{\mx}{\delta})  + \delta ( u_1(\mx) \chi(\frac{x_2}{ \delta}) +  \Pi_1(x_1,\frac{\mx}{\delta})) + 
     + \delta^2 ( u_2(\mx) \chi(\frac{x_2 }{\delta}) +  \Pi_2(x_1,\frac{\mx}{\delta})).
     \end{aligned}
\end{equation}
Again, we apply the Helmholtz operator on $v^\delta$. Then extracting
the macroscopic $\delta^2$ order and the $\delta^0$ order close to the
layer gives the equations for $u_2$ and $\Pi_2$. The term $u_2$ is solution of
the homogeneous Helmholtz equation
\begin{equation}
  \label{eq:deriv_ansatz_2_helm}
  - \Delta u_2 - k_0^2 u_2 = 0
\end{equation}
in $\OmegaTop \cap \Omega_a$ and $\OmegaBottom \cap \Omega_a$. The
periodic corrector $\Pi_2$ satisfies the following equation
\begin{equation}\label{ProblemPi2}
\left \lbrace
\begin{array}{rcll}
  \dsp - \Delta_{\mX} \Pi_{2}(x_1, \mX) & = & F_2(x_1, \mX)  &\; \mbox{in} \; \mathcal{B}, \\
  \dsp \partial_\bn \Pi_{2} &=& {-\partial_{x_1} \Pi_1 \, \be_1
                                \cdot \bn}  &\;\mbox{on} \;  \partial
                                              \widehat{\Omega}_\hole, \\
\partial_{X_1} \Pi_2(0,X_2) &= &\partial_{X_1} \Pi_2(1,X_2), &\; X_2 \in \R.
\end{array}
\right. 
\end{equation}
Here,
\begin{multline}
  \label{G2}
    F_2(x_1, \mX)  = \sum_{\pm} u_2(x_1,0^\pm) \chi_\pm''(X_2) + \frac{ \left( \left(\chi_+'(X_2)  -
      \chi_-'(X_2)\right) X_2 \right)' }{2} \left[\partial_{x_2} u_1
    \right]_\Gamma(x_1) \\ +  \; F_{V_{1,2}}\, \langle \partial_{x_2} u_1
    \rangle_{\Gamma}(x_1)
   + F_{V_{2,1}}(\mX) \, \langle u_0 \rangle_\Gamma(x_1) 
     \;  \\ + \; F_{V_{2,2}}(\mX)
    \, \partial_{x_1}^2 \langle u_{0}\rangle_{\Gamma}(x_1) \; + \;
   \; F_{V_{2,3}}(\mX)
    \, \partial_{x_1} \langle \partial_{x_2}
    u_{0}\rangle_{\Gamma}(x_1).
\end{multline}
$F_{V_0}$ and $F_{V_{1,2}}$ are given by~\eqref{ProblemW0t}-\eqref{ProblemW1n}, and,
\begin{multline}
  \nonumber
  F_{V_{2,1}}(\mX) = k_0^2  g(X_2) +
  \Big( \widehat{k}^2 - k_0^2 \Big) \ , \quad
 F_{V_{2,2}}(\mX) =  g(X_2)
  {+ 2 \, \partial_{X_1} \, V_{1,1}(\mX)} \ , \\
  F_{V_{2,3}}(\mX) = 2 \, \partial_{X_1} \, V_{1,2}(\mX) , \quad g(X_2) =  \left( \frac{(X_2)^2}{2} (1-\chi(X_2)) \right)''
\end{multline}
To obtain formula~\eqref{G2}, we have replaced $\Pi_0$ and $\Pi_1$
with their tensorial
representations~(\ref{decompositionPi0}),(\ref{decompositionPi1}), we
have replaced $ -\partial_{x_2}^2 u_0(x_1, 0^\pm)$ by
$\partial_{x_1}^2u_0(x_1, 0^\pm) +
k_0^2 u_0(x_1,0^\pm)$.\\

\noindent For a fixed $x_1 \in (-\LBottom, \LBottom)$, it is easily verified that $F_2(x_1, \cdot)$ belongs to
$(\mathcal{V}^-(\mathcal{B}))' $. %
Then again, the existence of an exponentially
decaying corrector $\Pi_2(x_1, \cdot) \in \mathcal{V}^+(\mathcal{B})$
results from the orthogonality condirtions~\eqref{eq:prop_layer_existence_uniqueness_compatibility_N}-\eqref{eq:prop_layer_existence_uniqueness_compatibility_D}. As previously, enforcing the compatibility condition
\eqref{eq:prop_layer_existence_uniqueness_compatibility_N}  provides
the transmission condition for the jump
 of the normal trace of $u_1$ across $\Gamma$:
\begin{equation}\label{Sautdu1}
[\partial_{x_2} u_1 ]_\Gamma = \mathcal{N}_{1} \, \langle
u_0 \rangle_{\Gamma} +   \mathcal{N}_{2} \, \partial_{x_1}^2 \langle
u_0 \rangle_{\Gamma} + \mathcal{N}_{3} \, \partial_{x_1} \langle \partial_{x_2}
u_0 \rangle_{\Gamma}  , 
\end{equation}
where
\begin{equation}\label{DefN2tN2n}
  \mathcal{N}_{1} = - \int_{\mathcal{B}} 
  F_{V_{2,1}}(\mX),  
  \quad  \mathcal{N}_{2} = - \int_{\mathcal{B}}
  F_{V_{2,2}} + \int_{\partial
    \widehat{\Omega}_\hole} V_{1,1} \mathbf{e}_1 \cdot \mathbf{n}, 
  \quad  \mathcal{N}_{3} = - \int_{\mathcal{B}}
  F_{V_{2,3}} + \int_{\partial
    \widehat{\Omega}_\hole} V_{1,2} \mathbf{e}_1 \cdot \mathbf{n}.
\end{equation}
Then, enforcing the compatibility
condition~\eqref{eq:prop_layer_existence_uniqueness_compatibility_D}
provides the jump $[u_2]_\Gamma$, and the existence of $\Pi_2$ is
proved. Naturally, an explicit expression of $[u_2]_\Gamma$ and a
tensorial representation of $\Pi_2$ can be written, but, for the sake
of concision and the relevance of this article, we do not write it
here.

\begin{remark}
  In the case of a symmetric hole (\ie $(X_1,X_2) \in \mathcal{B}
  \iff (1-X_1,X_2) \in \mathcal{B}$), $V_{1,2}$ is symmetric with respect to the axis $X_1 = \frac{1}{2}$, and, consequently, 
  $
  \mathcal{D}_1 = \mathcal{N}_3 = 0$.
\end{remark}

\section{Analysis of singular behavior of near field terms}
\label{sec:constr-singularities}
The (first order) near field terms satisfy Laplace problems (see~\eqref{NearFieldEquation}) and might grow at infinity. 
This consideration motivates us to introduce two families of so-called \emph{near field singularities} $S_n^\pm$ ($n \in \N$) that satisfy the following homogeneous near field problems
\begin{equation}\label{eq:Sn}
  \left\lbrace\quad
    \begin{aligned}
      -\Delta S_n^\pm & = 0  \quad \mbox{in} \; \widehat{\Omega}^\pm,\\
      \partial_\bn U & = 0 \quad \mbox{on} \;  \partial \widehat{\Omega}^\pm
    \end{aligned}
  \right. 
\end{equation}
and behaves like $(R^\pm)^{\lambda_n}$ for large $R^\pm$. 
%% Non-classical
%\subsection{The Weighted Sobolev spaces $\mathfrak{V}_{\beta, \gamma
%}^\ell(\widehat{\Omega}^\pm)$}
%For the statement of the next results, following the works of Nazarov \cite{Nazarov205}, we need to consider 
% the weighted Sobolev space $\mathfrak{V}_{\beta, \gamma
%}^\ell(\widehat{\Omega}^\pm)$, $\ell \in \{ 0,1,2\}$ defined as the completion of
%$C_c^\infty(\overline{\widehat{\Omega}^\pm})$ with respect to the norm
%\begin{equation}
%  \label{eq:definition_weighted_Sobolev_spaces}
%  \xnorm{v}{\mathfrak{V}_{\beta, \gamma }^\ell(\widehat{\Omega}^\pm)} =
%  \sum_{p=0}^{\ell} \xnorm{(1+R^\pm)^{\beta-\gamma-\delta_{p,0}}
%    (\rho^\pm)^{\gamma-\ell+p+\delta_{p,0}} \nabla^p
%    v}{\Ltwo(\widehat{\Omega}^\pm)},
%\end{equation}
%with $
%  \rho^+ = 1 + R^+ \, \lvert \theta^+
%  - \pi \rvert$ and $\rho^- = 1 + R^- \, \lvert \theta^- \rvert$.
%The norm
%$\xnorm{\cdot}{\mathfrak{V}_{\beta, \gamma
%  }^\ell(\widehat{\Omega}^\pm)}$
%is a non-uniform weighted norm, the weight varying with the angle
%$\theta^\pm$. Away from the layer,  we recover the classical
%Sobolev norm $
%  \xnorm{v}{V_{\beta}^\ell(\mathcal{K}^+)} =
%  \sum_{p=0}^{\ell} \xnorm{(1+R^+)^{\beta-\ell+p}
%    \nabla^p v}{\Ltwo(\mathcal{K}^+)}$,
%while close to the layer, the global weight in
%\eqref{eq:definition_weighted_Sobolev_spaces} is given by
%$(1+R^\pm)^{\beta-\gamma-\delta_{p,0}}$.
%

\subsection{Singular asymptotic blocks}
\label{sec:sing-block}

In absence of the periodic layer, \ie \, $\widehat{\Omega^\pm} =
\widehat{\mathcal K^\pm}$, the function $\ln
R^\pm$ and, for $n \in \Z\setminus\{ 0 \}$,  the functions $(R^+)^{\lambda_n} \, \cos (\lambda_n \theta^+)$  (resp.  $(R^-)^{\lambda_n}  \cos \lambda_n (\theta^--\pi)$)  are  particular solutions of the homogeneous Laplace equation with Neumann
boundary conditions on $\partial \widehat{\mathcal K^\pm}$. However,
these functions do not satisfy the homogenous problem~\eqref{eq:Sn}
since they do not fulfill the homogeneous Neumann boundary conditions
on the obstacles of the periodic layer. Nevertheless, as done in
  Section 3 of~\cite{Nazarov205}, for any $n \in
\N$, starting from the function
\begin{multline}
  \label{eq:defintion_w_0_star}
  w_{0,0,\pm}(\ln R^\pm, \theta^\pm) = \ln R^\pm,   \quad
  w_{n,0,+}(\theta^+) = \cos (\lambda_n \theta^+), \\ \quad
  w_{n,0,-}( \theta^-) = \cos \left( \lambda_n (\theta^--\pi)\right),
\end{multline}
it is possible to build iteratively a so-called asymptotic block $\mathcal{U}_{n,p,+}$ (for any $p \in
\IN$) of the form 
\begin{multline}
  \label{eq:ansatz_U_star_p}
  \mathcal{U}_{n,p,\pm} = \chi(R^\pm) \sum_{q=0}^p \big(
  \chi_{\text{macro},\pm}(X_1^\pm,X_2^\pm) (R^\pm)^{\lambda_n-q} w_{n,q,\pm}(\ln
  R^\pm, \theta^+) \\[-0.5em] %
  + \chi_\mp(X_1^\pm) \lvert X_1^+
  \rvert^{\lambda_n-q} p_{n,q,\pm}(\ln \lvert X_1^\pm \rvert, X_1^\pm,
  X_2^\pm) \big),
\end{multline}
that 'almost' satisfies problem~\eqref{eq:Sn} for large $R^\pm$.
In~\eqref{eq:ansatz_U_star_p}, the cut-off function $\chi_-$ has been defined in
\eqref{eq:defchipm} and is represented on the right part of
Figure~\ref{fig:chi_macro_plus}. The cut-off function
$\chi_{\text{macro},+}$, represented on the left part of
Figure~\ref{fig:chi_macro_plus}, is a smooth function that satisfies
\begin{equation}
  \label{eq:chi_macro_property}
  \chi_{\text{macro},+}(X_1^+,X_2^+) = \chi(X_2^+), \quad X_1^+ < -1.
\end{equation}
and the function $\chi_{\text{macro},-}(X_1^-,X_2^-) =
\chi_{\text{macro},+}(-X_1^-,X_2^-)$. 
\begin{figure}[!hbtp]
  \centering
  \input{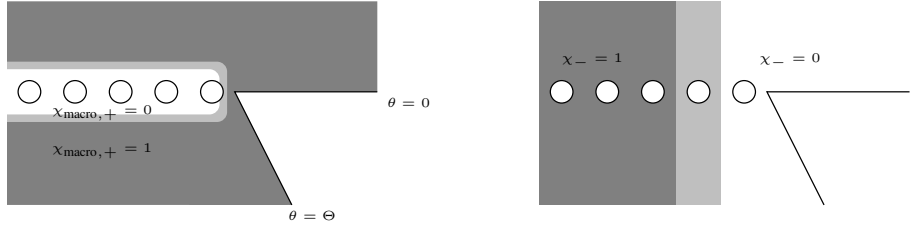}
  \caption{Graphic representation of the cut-off functions
    $\chi_{\text{macro}, +}$ (left part) and $\chi_-$ (right part).}
  \label{fig:chi_macro_plus}
\end{figure}

The definition of the functions $w_{n,q,\pm}$ and $p_{n,q,\pm}$ is given in
Appendix~\ref{sec:defin-prof-funct-wstar}. The functions $w_{n,q,\pm}$ are polynomials in
$\ln R^\pm$. The functions
$p_{n,q,\pm}$   are polynomials in $\ln \lvert X_1^\pm \rvert$, 
periodic with respect to $X_1^\pm$ and exponentially decaying as $X_2^\pm$
tends to $\pm \infty$.   The construction of these functions is done in such a way that their {Laplacian} and their Neumann trace become more and
more decaying at infinity as $p \to +\infty$: more precisely, we can prove that , for any $\varepsilon>0$,
\begin{equation} 
  \Delta  \mathcal{U}_{n,p,\pm}  = o\left( (R^\pm)^{{\lambda_n-p-1+\varepsilon}}\right)
  \quad \mbox{and} \quad  \partial_n  \mathcal{U}_{n,p,\pm}  = o\left( (R^\pm)^{{\lambda_n-p-1+\varepsilon}}\right) \text{on} \; \partial \widehat{\Omega}^\pm.
\end{equation}

We point out that the usage of the cut-off functions $\chi_{\text{macro},\pm}$,
$\chi_\mp(X_1^\pm)$ and $\chi(R^\pm)$ in \eqref{eq:ansatz_U_star_p}
is only a technical way to
construct functions defined on the whole domain $\widehat{\Omega}^\pm$.  \\

 \noindent The asymptotic blocks $\mathcal{U}_{n,p,\pm}$ turn out to be useful to construct the near field singularities $S_n^\pm$ and to describe their asymptotic for large $R^\pm$. 

\subsection{The families  $S_n^\pm$}

We are now in a position to write the main result of this subsection,
which proves the existence of the two families $S_n^\pm$ and give their behaviour at infinity. 
\begin{proposition}
  \label{prop:existence_uniqueness_NF_Sn}
  Let $n \in \N^\ast$, $p(n)=\max(1,1 + \lceil \lambda_n \rceil)$,  and
  \begin{equation}
    \label{eq:C_n}
    C_n^\pm := \begin{cases} 
    \displaystyle{ - \frac{1}{\Theta} \left( \int_{\widehat{\Omega}^\pm} \Delta \mathcal{U}_{n,p(n),\pm} -
      \int_{\partial \widehat{\Omega}^\pm} \partial_\bn
      \mathcal{U}_{n,p(n),\pm} \right) }& \mbox{if} \; \lambda_n \in \N,\\
      0 & \mbox{otherwise.}
      \end{cases}
  \end{equation}

     There exists a unique function $S_n^\pm \in \Hone_{\text{loc}}(\widehat{\Omega}^\pm)$ satisfying the homogeneous problem
  \eqref{eq:Sn} such that the function
  \begin{equation*}
    \tilde{S}_n^\pm = S_n^\pm - \mathcal{U}_{n,1+\lceil \lambda_n
      \rceil,\pm} - C_n^\pm \, \mathcal{U}_{0,1,\pm},  \end{equation*}
  tends to $0$
    as $R^\pm$ goes to infinity. Moreover, $S_n^\pm$ admits the following block
  decomposition for large $R^\pm$: for any $k \in \N^\ast$,
  \begin{equation}
    \label{eq:block_decomposition_Sn}
    \begin{aligned}
      S_n^\pm & = \mathcal{U}_{n,1+\lceil \lambda_{n+k} \rceil,\pm} +
      \sum_{m=1}^k \mathscr{L}_{-m}(S_n^\pm) \,
      \mathcal{U}_{-m,1+\lceil \lambda_{k-m} \rceil,\pm} + o\left(
        (R^\pm)^{-\lambda_k}\right) \\ & \quad \text{if
      }\lambda_n \not\in \IN,\\
      S_n^\pm & = \mathcal{U}_{n,1+\lceil \lambda_{n+k} \rceil,\pm} +
      \sum_{m=1}^k \mathscr{L}_{-m}(S_n^\pm) \,
      \mathcal{U}_{-m,1+\lceil \lambda_{k-m} \rceil,\pm} + o\left(
        (R^\pm)^{-\lambda_k}\right) \\ & + C_n^\pm \,
      \mathcal{U}_{0,1+\lceil \lambda_k \rceil,\pm} \quad \text{if
      }\lambda_n \in \IN.
    \end{aligned}
  \end{equation}
\end{proposition}
In the previous proposition $\lceil a \rceil$ denotes the ceiling of a
real number $a$. As demonstrated in
\ref{appendix_singularities}, for $\lambda_n \notin \N$, the
quantity $\int_{\widehat{\Omega}^\pm} \Delta \mathcal{U}_{n,p(n),\pm}
- \int_{\partial \widehat{\Omega}^\pm} \partial_\bn
\mathcal{U}_{n,p(n),\pm}$ vanishes
(Lemma~\ref{lema:compatibility_condition_Unp}), which explains why
$C_n=0$ in this case. The asymptotic
formula~\eqref{eq:block_decomposition_Sn} shows that, for large
$R^\pm$, $S_n^\pm$ can be decomposed as a sum of 'macroscopic'
contributions of the form $(R^\pm)^{\lambda_m -q} s_{m,q}(\theta^\pm,
\ln R^\pm)$ modulated by exponentially decaying (in $X_2^\pm$)
periodic (in $X_1$) functions of the form $|X_1^\pm|^{\lambda_m -q}
p_{m,q}(\ln |X_1^\pm|, X_1^\pm, X_2^\pm)$ in the vicinity if the
periodic layer. 
\begin{proof}
  The existence of the function $\tilde{S}_n^\pm$ results from the
  application of Proposition~\ref{PropositionProblemeChampProche} (or
  Corollary~3.23 of Ref.~\cite{CalozVial}), noting that the
  compatibility condition~\eqref{eq:CompatibilityProblemeChampProche}
  (due to the Neumann boundary condition) is satisfied : for
  $\lambda_n \in \N$, the addition of
  $C_n^\pm \, \mathcal{U}_{0,1,\pm}$ is required in order to fulfill
  this condition (note that, as shown in the proof of
  Lemma~\ref{lema:no_singularity_alone},
  $\int_{\widehat{\Omega}^\pm} \Delta \mathcal{U}_{0,1,\pm} -
  \int_{\partial \widehat{\Omega}^\pm} \partial_\bn
  \mathcal{U}_{0,1,\pm} = \Theta$).
  The asymptotic~\eqref{eq:block_decomposition_Sn} then follows from
  the application of the results of Nazarov\cite{Nazarov205} (see also
  Section~4 of Ref.\cite{Delourme.Semin.Schmidt:2015} for a
  detailed description of this decomposition). A rigorous estimation
  of the remainder $o\left( (R^\pm)^{-k}\right)$ can be done through
  the introduction of non-uniform weigthed Sobolev spaces\cite{Nazarov205}.
\end{proof}
\begin{remark}\label{RemAbsLog}
  We point out that it is not possible to construct a function in
  $\Honeloc(\widehat{\Omega}^\pm)$ satisfying the homogeneous problem
  \eqref{eq:Sn} and behaving like $\ln R^\pm$ at infinity (See
  Lemma~\ref{lema:no_singularity_alone} in
  Appendix~\ref{sec:absence-logar-sing}).
\end{remark}

We complete the family $(S_n^\pm)_{n>0}$ defined in
Proposition~\ref{prop:existence_uniqueness_NF_Sn} defining the
function 
\begin{equation}
  \label{eq:S_0_pm}
  S_0^\pm= 1,
\end{equation}
which obviously satisfies the homogeneous Laplace equation on
$\widehat{\Omega}^\pm$.

\section{Iterative construction of the first terms of the expansion}
\label{sec:constr-first-terms}

In this section, we propose a step by step iterative procedure to
construct the first terms of the expansion up to order
$\delta^2$. Since $\theta \in (\pi, 2\pi)$,
$0<\lambda_1 < 1 < \lambda_2 < \lambda_1 + 1 < \lambda_3$.  It follows
that we shall consider the indexes $(n,q)$ (associated with increasing
powers of $\delta^{\lambda_n +q}$) in the following order: $(0,0)$,
$(1,0)$, $(0,1)$, $(2,0)$, $(1,1)$ and, in the case of
$\Theta > \frac{3\pi}{2}$, the couple $(3,0)$.

\subsection{Construction of the limit terms $u_{0,0}^\delta$, $\Pi_{0,0}^\delta$ and $U_{0,0,\pm}^\delta$}
\label{sec:constr-limit-terms}
The macroscopic term $u_{0,0}^\delta$ and the near field terms $U_{0,0,\pm}^\delta$  satisfy the following problems
\begin{equation}
  \label{eq:macroscopic_u_0_0}
  \left \lbrace
    \begin{array}{r@{\;}ll}
      -\Delta u_{0,0}^\delta - k_0^2 u_{0,0}^\delta & = 0 &\quad  \mbox{in} \; \OmegaTop \cup \OmegaBottom,\\[0.5em]
      \left[ u_{0,0}^\delta \right]_{\Gamma} = \left[ \partial_{x_2} u_{0,0}^\delta \right]_{\Gamma} &= 0 &\quad \mbox{on} \, \Gamma,\\[0.5em]
    %  \left[ \partial_{x_2} u_{0,0}^\delta \right]_{\Gamma} & = 0 &\quad \mbox{on} \, \Gamma, \\[0.5em]
      \nabla u_{0,0}^\delta \cdot \bn & = 0 &\quad \mbox{on} \, \Gamma_N,
      \\[0.5em]
      \nabla u_{0,0}^\delta \cdot \bn - \imath k_0 u_{0,0}^\delta & = 0 &\quad \mbox{on} \, \Gamma_{R,+},\\[0.5em]
      \nabla u_{0,0}^\delta \cdot \bn - \imath k_0 u_{0,0}^\delta & = -2\imath
      k_0 &\quad \mbox{on} \, \Gamma_{R,-},
    \end{array}
  \right. \ \mbox{and} \ 
  \left\lbrace\quad
    \begin{array}{rcll}
      -\Delta U_{0,0,\pm}^\delta & = &  0 &  \mbox{in} \; \widehat{\Omega}^\pm,\\
      \partial_\bn U_{0,0,\pm}^\delta & = & 0 & \mbox{on} \; \partial
                                              \widehat{\Omega}^\pm,
    \end{array}
  \right. 
\end{equation}
coupled by the matching condition
\eqref{eq:matching_conditions_global_main} (written here by  only identifying the term of  order $0$ in the two series)
\begin{equation}\label{MatchingConditionOrder0}
U_{0,0,\pm}^\delta(\frac{x}{\delta}) \approx u_{0,0}^\delta (x).
\end{equation} 
\subsubsection{Construction of the macroscopic term $u_{0,0}^\delta$}
As well-known, the limit term $u_{0,0}^\delta$ is regular. In fact, $u_{0,0}^\delta = u_{0,0}$ (it does not depend on $\delta$) is defined as the unique solution of~\eqref{eq:macroscopic_u_0_0}-(left) belonging to  $\Hone(\Omega)$. The absence of singular behavior in $u_{0,0}$  can be understood by the following \emph{formal} argument:  a singular term in~$u_{0,0}$ of the form $r^{-s}$ would necessary counterbalance a term of the form $(R^\pm)^{-s}$, $s>0$ in $U_{0,0}^\delta$, which, written in terms of the macroscopic variable $r^\pm$,  become $\delta^{s} (r^\pm)^{-s}$, and can therefore not be canceled at order $0$.  Similarly, due to Remark~\ref{RemAbsLog} a singular term of the form $\ln r$ is excluded at this stage.

\begin{remark} \label{RemarkConstruction1}More generally, the previous
  argument shows that for any $(n,q) \in \N^2$, a singular term
  in~$u_{n,q}^\delta$ of the form $r^{-s}$ cannot counterbalance a
  regular term of the near field term of the same order
  $U_{n,q}^\delta$.
\end{remark}     
It is well-known that $u_{0,0}$ admits the following expansion in the
matching zones
\begin{equation}
  \label{eq:macroscopic_u_0_0_general_representation2}
  u_{0,0}(r^\pm, \theta^\pm) = \ell_0^\pm(u_{0,0}) J_0 (k_0 r^\pm) + 
  \sum_{m=1}^\infty \ell_m^\pm(u_{0,0}) J_{\lambda_m} (k_0 r^\pm)
  w_{m,0,\pm}(\theta^\pm),
\end{equation}
where the functions $w_{m,0,\pm}$ are defined by
\eqref{eq:defintion_w_0_star}, the functions $J_{\lambda_m}$ are the
Bessel functions of first kind (see \eg Section~9.1 of
Ref.~\cite{MR0167642}) and the quantities $\ell_m^\pm(u_{0,0})$ are
complex constants.  Using the radial decomposition of $J_{\lambda_m}$,
we see that
\begin{multline}\label{Expansionu00}
u_{0,0}(r^\pm, \theta^\pm) = \ell_{0}^\pm(u_{0,0}) + \frac{\ell_{1}^\pm(u_{0,0})
    (k_0/2)^{\lambda_1}}{\Gamma(\lambda_1+1)}
  \left( r^\pm \right)^{\lambda_1}
  w_{1,0,\pm}(\theta^\pm) + \\ \frac{\ell_{2}^\pm(u_{0,0})
    (k_0/2)^{\lambda_2}}{\Gamma(\lambda_2+1)} 
  \left( r^\pm \right)^{\lambda_2}
  w_{2,0,\pm}(\theta^\pm)  + \frac{\ell_{3}^\pm(u_{0,0})
    (k_0/2)^{\lambda_3}}{\Gamma(\lambda_3+1)}   \left( r^\pm \right)^{\lambda_3}
  w_{3,0,\pm}(\theta^\pm) + O(r^2).
  \end{multline}
\subsubsection{Construction of $U_{0,0,\pm}^\delta$}
We now turn to the definition of the near field term $U_{0,0,\pm}^\delta$. In view of~\eqref{Expansionu00}, writting the matching condtions~\eqref{MatchingConditionOrder0} in term of the microscopic variable gives 
$\ell_0^\pm(u_{0,0}) \approx U_{0,0,\pm}^\delta$.
As a result, $U_{0,0,\pm}^\delta$ should  behave like $\ell_0^\pm(u_{0,0})$ in the matching zones (\ie for $R^\pm$ large). Consequently,  it is natural to define $U_{0,0,\pm}^\delta$  as
\begin{equation}\label{definitionU00}
U_{0,0,\pm}^\delta = U_{0,0,\pm} = \ell_0^\pm(u_{0,0}).
\end{equation}

\subsubsection{Construction of the periodic corrector  $\Pi_{0,0}^\delta$} Finally, using then relations \eqref{decompositionPi0},  the periodic boundary layer corrector is
\begin{equation}
  \label{eq:corrector_Pi_0_0}
  \Pi_{0,0}^\delta(x_1,\mX) =  \Pi_{0,0}(x_1,\mX) = \langle u_0 \rangle_{\Gamma}(x_1) \, (1  - \chi(X_2)).
\end{equation}
\subsection{Construction of the terms $u_{1,0}^\delta$, $\Pi_{1,0}^\delta$ and $U_{1,0,\pm}^\delta$}
\label{sec:constr-terms-u_10}
Reminding that $u_{1,0}^\delta$ fulfills the jump conditions \eqref{eq:macroscopic_u_n_0} (see also Section~\ref{SectionTransmissionConditionsBoundaryLayer}),
$u_{1,0}^\delta$ and $U_{1,0,\pm}^\delta$ satisfy
\begin{equation}
  \label{eq:macroscopic_u_1_0}
  \left \lbrace
    \begin{array}{r@{\;}ll}
      -\Delta u_{1,0}^\delta - k_0^2 u_{1,0}^\delta & = 0 &\quad  \mbox{in} \; \OmegaTop \cup \OmegaBottom,\\[0.5em]
      \left[ u_{1,0}^\delta \right]_{\Gamma} = \left[ \partial_{x_2} u_{1,0}^\delta \right]_{\Gamma} &= 0  &\quad \mbox{on} \, \Gamma ,\\[0.5em]
%      \left[ \partial_{x_2} u_{1,0}^\delta \right]_{\Gamma} & = 0  &\quad \mbox{on} \, \Gamma, \\[0.5em]
      \nabla u_{1,0}^\delta \cdot \bn & = 0 &\quad \mbox{on} \, \Gamma_N,
      \\[0.5em]
      \nabla u_{1,0}^\delta \cdot \bn - \imath k_0 u_{1,0}^\delta & = 0 &\quad \mbox{on} \, \Gamma_R,
    \end{array}
  \right. \quad \mbox{and} \quad
   \left \lbrace
  \begin{array}{rcll}
      -\Delta U_{1,0,\pm}^\delta & = &  0 &\mbox{in} \; \widehat{\Omega}^\pm,\\
      \partial_\bn U_{1,0,\pm}^\delta & = & 0  & \mbox{on} \; \partial
      \widehat{\Omega}^\pm,\\
   \end{array}
   \right.
   \end{equation}
   together with the matching condition
   \eqref{eq:matching_conditions_global_main} written up to order
   $\delta^{\lambda_1}$.  Outside the thin periodic layer, and thanks
   to~\eqref{Expansionu00} and \eqref{definitionU00}, one can verify
   that this matching condition can be rewritten as
\begin{equation}
  \label{eq:matching_conditions_order_lambda1_pm_y1_positive2}
     \frac{\ell_{1}^\pm(u_{0,0})
    (k_0/2)^{\lambda_1}}{\Gamma(\lambda_1+1)}  \left( r^\pm \right)^{\lambda_1} w_{1,0,\pm}(\theta^\pm)+  \delta^{\lambda_1} u_{1,0}^\delta
  (r^\pm, \theta^\pm) \approx  
  \delta^{\lambda_1} U_{1,0,\pm}^\delta  \left(\frac{r^\pm }{{\delta}}, \theta^\pm\right).
\end{equation}

Analogously to Section~\eqref{sec:constr-limit-terms}, we will start with the construction of the macroscopic far field $u_{1,0}^\delta$. Then, we will define the near field term $U_{1,0}^\delta$ and, finally, we will define the associated boundary layer corrector $\Pi_{1,0}^\delta$. 
\subsubsection{Construction of the macroscopic term $u_{1,0}^\delta$}
First, it is reasonable to construct $u_{1,0}^\delta$ as a regular function. Indeed, a singular behaviour in  $u_{1,0,\pm}^\delta$ of the form $r^{-s}$ or (resp. $\ln r$) would counterbalance a regular term of the right hand side of \eqref{eq:matching_conditions_order_lambda1_pm_y1_positive2}.  This singular term would necessary come from a regular  term in  $U_{1,0,\pm}^\delta$,  which, thanks to Remark~\ref{RemarkConstruction1} (resp. Remark~\ref{RemAbsLog}) cannot be cancelled at this stage. It is then reasonable (see Proposition~\ref{prop:existence_uniqueness_macro}) to define $u_{1,0}^\delta$ as
\begin{equation}\label{Defintionu10delta}
u_{1,0}^\delta  = u_{1,0} := 0.
\end{equation}
\subsubsection{Construction of $U_{1,0,\pm}^\delta$} Taking into
account~\eqref{Defintionu10delta} and writing the matching
condition~\eqref{eq:matching_conditions_order_lambda1_pm_y1_positive2}
in term of the microscopic variables gives
\begin{equation*}
\delta^{\lambda_1} \frac{\ell_{1}^\pm(u_{0,0})
    (k_0/2)^{\lambda_1}}{\Gamma(\lambda_1+1)}  \left( R^\pm \right)^{\lambda_1} w_{1,0,\pm}(\theta^\pm)  \approx  
  \delta^{\lambda_1} U_{1,0,\pm}^\delta  \left(\frac{r^\pm }{{\delta}}, \theta^\pm\right).
\end{equation*}
Then, $U_{1,0,\pm}^\delta$ has to grow like
$ \frac{\ell_{1}^\pm(u_{0,0})
  (k_0/2)^{\lambda_1}}{\Gamma(\lambda_1+1)} \left( R^\pm
\right)^{\lambda_1} w_{1,0,\pm}(\theta^\pm)$
towards infinity.  Of course, the term
$\big( R^\pm \big)^{\lambda_1} w_{1,0,\pm}(\theta^\pm)$ does not
satisfies the homogeneous
problem~\eqref{eq:macroscopic_u_1_0}-(right). However,
Proposition~\ref{prop:existence_uniqueness_NF_Sn} ensures the
existence of a function $S_1^\pm$, that satisfies
\eqref{eq:macroscopic_u_1_0}-(right) and behaves like
$\left( R^\pm \right)^{\lambda_1} w_{1,0,\pm}(\theta^\pm)$ at
infinity. Then, it is natural to define $U_{1,0,\pm}^\delta$ as
\begin{equation} \label{DefinitionU10}
  U_{1,0,\pm}^\delta =  U_{1,0,\pm} = \frac{\ell_{1}^\pm(u_{0,0})
    (k_0/2)^{\lambda_1}}{\Gamma(\lambda_1+1)}  S_1^\pm.
\end{equation}
In view of the asymptotic formula~\eqref{eq:block_decomposition_Sn} for $S_1^\pm$ ($\lambda_1 \notin \mathbb{N}$), outside the periodic layer, the asymptotic of $U_{1,0,\pm}$ is given by
\begin{multline}\label{ExpansionU10}
  U_{1,0,\pm} = \frac{\ell_{1}^\pm(u_{0,0})
    (k_0/2)^{\lambda_1}}{\Gamma(\lambda_1+1)}  \Big\{ (R^\pm)^{\lambda_1} w_{1,0,\pm}(\theta^\pm)  +  (R^\pm)^{\lambda_1-1}  w_{1,1,\pm}(\theta^\pm) \\ + (R^\pm)^{-\lambda_1}  \mathscr{L}_{-1}(S_1^\pm) w_{-1,0,\pm}(\theta^\pm)  \\
    + (R^\pm)^{-\lambda_2}  \mathscr{L}_{-2}(S_1^\pm) w_{-2,0,\pm}(\theta^\pm) 
  \Big\} + O(R^{\lambda_1 -2} \ln R).
\end{multline} 
Here we use the fact that $\lambda_1 -1$ is not a mutliple of $\lambda_1$ so that $w_{1,1,\pm}$ is independent of $\ln R^\pm$. 

\subsubsection{Construction of the periodic corrector  $\Pi_{1,0}^\delta$} Thanks to the relation~\eqref{decompositionPi0}, and since $u_{1,0}^\delta$ vanishes, its associated boundary corrector also vanishes, and we have
\begin{equation}
\Pi_{1,0}^\delta = \Pi_{1,0} = 0.
\end{equation}

\subsection{Construction of the terms $u_{0,1}^\delta$ and $U_{0,1,\pm}^\delta$}
\label{sec:constr-terms-u_01}

Reminding that $u_{0,1}^\delta$ fulfills the jump conditions
\eqref{eq:macroscopic_u_n_1} (cf. Appendix~\ref{SectionTransmissionConditionsBoundaryLayer}),
$u_{0,1}^\delta$ and $U_{0,1,\pm}^\delta$ satisfy the following problems
\begin{equation}
  \label{eq:macroscopic_u_0_1}
  \left \lbrace
    \begin{array}{r@{\;}ll}
      -\Delta u_{0,1}^\delta - k_0^2 u_{0,1}^\delta & = 0 &\quad  \mbox{in} \; \OmegaTop \cup \OmegaBottom,\\[0.5em]
      \left[ u_{0,1}^\delta \right]_{\Gamma} &= g_{0,1}  &\quad \mbox{on} \, \Gamma ,\\[0.5em]
      \left[ \partial_{x_2} u_{0,1}^\delta \right]_{\Gamma} & = h_{0,1}  &\quad \mbox{on} \, \Gamma, \\[0.5em]
      \nabla u_{0,1}^\delta \cdot \bn & = 0 &\quad \mbox{on} \, \Gamma_N,
      \\[0.5em]
      \nabla u_{0,1}^\delta \cdot \bn - \imath k_0 u_{0,1}^\delta & = 0 &\quad \mbox{on} \, \Gamma_R,
    \end{array}
  \right. \quad \mbox{and} \quad
  \left \lbrace
    \begin{array}{rcll}
      -\Delta U_{0,1,\pm}^\delta & = &  0 &\mbox{in} \; \widehat{\Omega}^\pm,\\
      \partial_\bn U_{0,1,\pm}^\delta & = & 0  & \mbox{on} \; \partial
                                               \widehat{\Omega}^\pm,\\
    \end{array}
  \right.
\end{equation}
with
$
\label{eq:rhs_g_0_1_h_0_1}
g_{0,1} = \mathcal{D}_{1} \, \partial_{x_1} \langle u_{0,0}
\rangle_{\Gamma} \, + \, \mathcal{D}_{2} \, \langle \partial_{x_2}
u_{0,0} \rangle_{\Gamma}$
and
$ h_{0,1} = \mathcal{N}_{1} \, \langle u_{0,0} \rangle_{\Gamma} +
\mathcal{N}_{2} \, \partial_{x_1}^2 \langle u_{0,0} \rangle_{\Gamma}
+ \mathcal{N}_{3} \, \partial_{x_1} \langle \partial_{x_2} u_{0,0}
\rangle_{\Gamma} $.
Thanks to
\eqref{Expansionu00}-\eqref{definitionU00}-\eqref{ExpansionU10}, the
matching condition \eqref{eq:matching_conditions_global_main}
written up to order $\delta$, can be written as
\begin{equation}
  \label{eq:matching_conditions_order_1_pm_y1_positive2}
  \delta\,  u_{0,1}^\delta
  ( r^\pm, \theta^\pm) \approx \delta
  \frac{\ell_{1}^\pm(u_{0,0}) 
    (k_0/2)^{\lambda_1}}{\Gamma(\lambda_1+1)} 
    (r^\pm )^{\lambda_1-1}
  w_{1,1,\pm}(\theta^\pm) +  \delta \,
   U_{0,1,\pm}^\delta
  \left(\frac{r^\pm}{\delta},\theta^\pm \right)
\end{equation}
outside the periodic layer. Analogously to Sections~\ref{sec:constr-limit-terms}
and~\ref{sec:constr-terms-u_10}, we will start with the construction
of the macroscopic far field $u_{0,1}^\delta$. Then, we will define
the near field term $U_{0,1}^\delta$. As we have already seen in the
previous sections, we can rebuild \emph{a posteriori} the boundary
layer corrector $\Pi_{0,1}^\delta$, but for the sake of brevity, from now on, we omit this reconstruction. 

\subsubsection{Construction of  the macroscopic term $u_{0,1}^\delta$} 
First, we remark that $u_{0,1}^\delta$ should contain a singular contribution of order $(r^\pm )^{\lambda_1-1}$
in order to cancel out the first term in the right-hand side of 
\eqref{eq:matching_conditions_order_1_pm_y1_positive2}. In fact, we shall see  (and this is a crucial point) that this singular contribution appears to be a consequence of the transmission condition in~\eqref{eq:macroscopic_u_0_1}-(left). Besides, according to Remark~\ref{RemarkConstruction1},  $u_{0,1}^\delta$ has no other singular behavior (any other singular behavior would stem from $U_{0,1,\pm}^\delta$ and could not be compensated at this stage).\\

Let us now investigate Problem~\eqref{eq:macroscopic_u_0_1}-(right). In view of the asymptotic behaviour of $u_{0,0}$~\eqref{Expansionu00} in the vicinity of the two corners, the functions $g_{0,1}$ and $h_{0,1}$ blow up at the extremities of $\Gamma$. Indeed, 
\begin{multline*}
 g_{0,1}(r^\pm) = \ell_1^\pm(u_{0,0})  \, \frac{
   (k_0/2)^{\lambda_1}}{\Gamma(\lambda_1 +1)} \, \left( \mp \,
   \lambda_1 \mathcal{D}_1  \langle  w_{1,0,\pm} \rangle \mp \mathcal{D}_2  \, \langle \partial_{\theta^\pm} w_{1,0,\pm} \rangle \right) \,  (r^\pm)^{\lambda_1 -1} \\  + O((r^\pm)^{\lambda_2 -1} )
=  \, \ell_1^\pm(u_{0,0}) \frac{   (k_0/2)^{\lambda_1}}{\Gamma(\lambda_1+1)} [ (r^\pm)^{\lambda_1 -1} w_{1,1,\pm}] + O((r^\pm)^{\lambda_2 -1} ) ,
\end{multline*}
where  the functions $w_{1,1,\pm}$ are defined in Appendix~\eqref{sec:defin-prof-funct-wn} (note that $V_{1,1} = W_{1}^{\mathfrak{t}}$ and $V_{1,2} = W_{1}^{\mathfrak{n}}$).  Similarly,
\begin{align*}
h_{0,1}(r^\pm) = \ell_1^\pm(u_{0,0}) \frac{   (k_0/2)^{\lambda_1}}{\Gamma(\lambda_1 +1)}   \left[ \partial_{x_2}  \left( (r^\pm)^{\lambda_1 -1}  w_{1,1,\pm} \right) \right] + O((r^\pm)^{\lambda_2 -2} )\ .
\end{align*}
We shall construct $u_{0,1}^\delta$ by lifting explicitly the singular part of $g_{0,1}$ and $h_{0,1}$.  To do so, we consider the function
\begin{equation}\label{DefinitionGrandJ1moins1}
  \mJ_{1,-1}^\pm(r^\pm, \theta^\pm) = 
  J_{\lambda_1-1}(k_0 r^\pm) \, w_{1,1,\pm}(\theta^\pm)
\end{equation}
that satisfies the homogeneous Helmholtz equation in
$\OmegaTop \cup \OmegaBottom$. According to the asymptotic of the
Bessel function of the first kind $J_{\lambda_1-1}$ (using
Equation~(9.1.10) of Ref.~\cite{MR0167642}), we notice that
$g_{0,1}\approx \frac{k_0}{2 \lambda_1} \ell_1^\pm(u_{0,0}) [
\mJ_{1,-1}^\pm]$
and
$h_{0,1}\approx \frac{k_0}{2 \lambda_1} \ell_1^\pm(u_{0,0})
[\partial_{x_2} \mJ_{1,-1}^\pm]$
in the neighborhood of the extremities of $\Gamma$. It means that
$\frac{k_0}{2 \lambda_1} \ell_1^\pm(u_{0,0})\mJ_{1,-1}^\pm$ is
potentially a good candidate to lift the singular parts of the
$g_{0,1}$ and $h_{0,1}$.  It is then natural to define
$u_{0,1}^\delta$ as
\begin{equation}\label{definitionu01delta}
u_{0,1}^\delta = u_{0,1} := \frac{k_0}{2 \lambda_1} \left(  \ell_1^+(u_{0,0}) \chi_\LBottom^+  \mJ_{1,-1}^+ + \ell_1^-(u_{0,0}) \chi_\LBottom^-  \mJ_{1,-1}^- \right) + \hat{u}_{0,1},
\end{equation}
where $\chi_\LBottom^\pm(\mx) = 1 - \chi(2 r^\pm / \LBottom)$ and the function  $\hat{u}_{0,1}$ is the unique solution  in $\Hone(\OmegaTop \cup \OmegaBottom)$ of the following problem:
  \begin{equation}
    \label{eq:macroscopic_v_0_1}
    \left \lbrace
      \begin{array}{r@{\;}ll}
        -\Delta  \hat{u}_{0,1} - k_0^2  \hat{u}_{0,1} & =  \hat{f}_{0,1} &\quad  \mbox{in} \; \OmegaTop \cup \OmegaBottom,\\[0.5em]
        \left[  \hat{u}_{0,1}\right]_{\Gamma} & = \hat{g}_{0,1}   &\quad \mbox{on} \, \Gamma ,\\[0.5em]
        \left[ \partial_{x_2}  \hat{u}_{0,1} \right]_{\Gamma} & = \hat{h}_{0,1}  &\quad \mbox{on} \, \Gamma , \\[0.5em]
        \nabla  \hat{u}_{0,1}\cdot \bn & = 0 &\quad \mbox{on} \, \Gamma_N,
        \\[0.5em]
        \nabla  \hat{u}_{0,1} \cdot \bn - \imath k_0  \hat{u}_{0,1}& = 0 &\quad \mbox{on} \, \Gamma_R,
      \end{array}
    \right.
    \mbox{with}
    \quad
    \begin{array}{l}
    \hat{f}_{0,1} =  \dsp \frac{k_0}{2 \lambda_1}  \sum_{\pm} \ell_{1}^\pm(u_{0,0})
    \commu{\Delta}{\chi_\LBottom^\pm} \mJ_{1,-1}^\pm,\\
    \hat{g}_{0,1}  = g_{0,1} -  \dsp \frac{k_0}{2 \lambda_1}   \sum_{\pm}  \ell_{1}^\pm(u_{0,0}) \chi_\LBottom^ \pm [ \mJ_{1,-1}^\pm] , \\
\hat{h}_{0,1} = h_{0,1} -  \dsp \frac{k_0}{2 \lambda_1}   \sum_{\pm}  \ell_{1}^\pm(u_{0,0}) \chi_\LBottom^ \pm [\partial_{x_2} \mJ_{1,-1}^\pm].
    \end{array} 
  \end{equation}
 Here $\commu{\Delta}{\chi_\LBottom^\pm}$ denotes the commutator operator given by
  $\commu{\Delta}{\chi_\LBottom^\pm} v = v \Delta \chi_\LBottom^\pm + 2
  \nabla \chi_\LBottom^\pm \cdot \nabla v$ (for any sufficiently smooth function $v$).
 The existence and uniqueness of $\hat{u}_{0,1}$ in $\Hone(\OmegaTop \cup \OmegaBottom)$  is ensured by Proposition~\ref{prop:existence_uniqueness_macro} since $\hat{f}_{0,1}$ is compactly supported,  $\hat{g}_{0,1} \in \HonehalfG$
  and $\hat{h}_{0,1} \in \LtwoG$. 

Moreover, the asymptotic expansion of $u_{0,1}$ in the matching zones is given by
\begin{multline}
  \label{eq:expansion_u_0_1_expansion}
  u_{0,1}(r^\pm, \theta^\pm) = \ell_{1}^\pm(u_{0,0}) 
\frac{(k_0/2)^{\lambda_1} }{\Gamma(\lambda_1+1)} (r^\pm)^{\lambda_1-1}
w_{1,1,\pm}(\theta^\pm) +  \ell_{0}^\pm(u_{0,1}) \\ + \ell_{2}^\pm(u_{0,0})  \frac{(k_0/2)^{\lambda_2} }{\Gamma(\lambda_2+1)} (r^\pm)^{\lambda_2-1}  w_{2,1,\pm}(\theta^\pm)
  \\+  \ell_{1}^\pm(u_{0,1}) \frac{(k_0/2)^{\lambda_1}}{\Gamma(\lambda_1+1)}
  \left( r^\pm \right)^{\lambda_1}  w_{1,0,\pm}(\theta^\pm) \\ + \ell_{3}^\pm(u_{0,0})  \frac{(k_0/2)^{\lambda_3} }{\Gamma(\lambda_3+1)} (r^\pm)^{\lambda_3-1}  w_{3,1,\pm}(\theta^\pm)
+ O\left( r^\pm \right), 
\end{multline}
where the quantities $\ell_{0}^\pm(u_{0,1})$ and
$\ell_{1}^\pm(u_{0,1})$ are complex constants. Obviously, the first
term of \eqref{eq:expansion_u_0_1_expansion} compensates the first
term of the right hand side of the matching
condition~\eqref{eq:matching_conditions_order_1_pm_y1_positive2}. The
presence of the terms in factor of $\ell_{2}^\pm(u_{0,0})$ and
$\ell_{3}^\pm(u_{0,0})$ results from the transmission condition (see
Section~3.3 in Ref.~\cite{Delourme.Semin.Schmidt:2015} for a similar
asymptotic). %
If $\Theta<\frac{3\pi}{2}$ and so $\lambda_3 > 2$, the last listed
term of the expansion~\eqref{eq:expansion_u_0_1_expansion} is
negligible with respect to $O(r^\pm)$.
\subsubsection{Construction of $U_{0,1,\pm}^\delta$}
Plugging the asymptotic expansion~\eqref{eq:expansion_u_0_1_expansion} of $u_{0,1}$  into the matching condition~\eqref{eq:matching_conditions_order_1_pm_y1_positive2} written in term of the microscopic variable (ignoring the terms in factor of $\delta^s$, $s>1$, which will be taken into account latter), we obtain
$$
\delta U_{0,1,\pm}^\delta(R^\pm, \theta^\pm)  \approx \delta  \ell_{0}^\pm(u_{0,1}).
$$  
We then see that $U_{0,1,\pm}^\delta$ should behave like $\ell_0^\pm(u_{0,1})$ at infinity. Thus, we define $U_{1,0,\pm}^\delta$ as
\begin{equation}\label{DefinitionU01}
  U_{0,1,\pm}^\delta = U_{0,1,\pm} = \ell_0^\pm(u_{0,1}).
\end{equation}
\subsection{Construction of the terms $u_{2,0}^\delta$ and $U_{2,0,\pm}^\delta$}
\label{sec:constr-terms-u_20}
Reminding that $u_{2,0}^\delta$ fulfills the jump conditions \eqref{eq:macroscopic_u_n_0},
$u_{2,0}^\delta$ and $U_{2,0,\pm}^\delta$ satisfy
\begin{equation}
  \label{eq:macroscopic_u_2_0}
  \left \lbrace
    \begin{array}{r@{\;}ll}
      -\Delta u_{2,0}^\delta - k_0^2 u_{2,0}^\delta & = 0 &\quad  \mbox{in} \; \OmegaTop \cup \OmegaBottom,\\[0.5em]
      \left[ u_{2,0}^\delta \right]_{\Gamma} =\left[ \partial_{x_2} u_{2,0}^\delta \right]_{\Gamma}&= 0  &\quad \mbox{on} \, \Gamma ,\\[0.5em]
     % \left[ \partial_{x_2} u_{2,0}^\delta \right]_{\Gamma} & = 0  &\quad \mbox{on} \, \Gamma, \\[0.5em]
      \nabla u_{2,0}^\delta \cdot \bn & = 0 &\quad \mbox{on} \, \Gamma_N,
      \\[0.5em]
      \nabla u_{2,0}^\delta \cdot \bn - \imath k_0 u_{2,0}^\delta & = 0 &\quad \mbox{on} \, \Gamma_R,
    \end{array}
  \right. \quad \mbox{and} \quad
   \left \lbrace
  \begin{array}{rcll}
      -\Delta U_{2,0,\pm}^\delta & = &  0 &\mbox{in} \; \widehat{\Omega}^\pm,\\
      \partial_\bn U_{2,0,\pm}^\delta & = & 0  & \mbox{on} \; \partial
      \widehat{\Omega}^\pm,\\
   \end{array}
   \right.
   \end{equation}
together with the matching condition~
\eqref{eq:matching_conditions_global_main} written up to
order $\delta^{\lambda_2}$,
%\begin{equation}
%u_{0,0} + \delta u_{0,1} + \delta^{\lambda_2} \, u_{2,0}^\delta \approx U_{0,0} + \delta^{\lambda_1} U_{1,0} + \delta \, U_{0,1} + \delta^{\lambda_2} U_{2,0}^\delta,
%\end{equation}
which, outside the thin periodic layer gives
\begin{multline}\label{eq:raccordOrdrelambda2}
 \frac{\ell_{2}^\pm(u_{0,0})
  (k_0/2)^{\lambda_2}}{\Gamma(\lambda_2+1)}  \left(  \left( r^\pm
\right)^{\lambda_2} w_{2,0,\pm}(\theta^\pm) + \delta  (r^\pm)^{\lambda_2-1}  w_{2,1,\pm}(\theta^\pm) \right)+  \delta^{\lambda_2} u_{2,0}^\delta \\ \approx    \delta^{\lambda_2} \frac{\ell_{1}^\pm(u_{0,0})
    (k_0/2)^{\lambda_1}}{\Gamma(\lambda_1+1)} \mathscr{L}_{-1}(S_1^\pm) \left( r^\pm
  \right)^{-\lambda_1} w_{-1,0,\pm}(\theta^\pm)  + \delta^{\lambda_2} U_{2,0}^\delta.
\end{multline}
Here, we used the asymptotic expansions~\eqref{eq:macroscopic_u_0_0_general_representation2}-\eqref{eq:expansion_u_0_1_expansion} for the far field terms $u_{0,0}$ and $u_{0,1}$, the definition~\eqref{definitionU00}-and \eqref{DefinitionU01} of the near field terms $U_{0,0}$ and $U_{0,1}$, and the asymptotic expansion~\eqref{ExpansionU10} of $U_{1,0}$.  Predicably, the matching process carried out in the previous subsections makes the expression of~\eqref{eq:raccordOrdrelambda2} relatively simple.

\subsubsection{Construction of  the macroscopic term $u_{2,0}^\delta$}
In view of the right-hand side of~\eqref{eq:raccordOrdrelambda2} (and, here again, Remark~\ref{RemarkConstruction1}), we
remark that $u_{2,0,\pm}^\delta$ should have a single singular contribution of the
form 
$\frac{\ell_{1}^\pm(u_{0,0})
    (k_0/2)^{\lambda_1}}{\Gamma(\lambda_1+1)} \mathscr{L}_{-1}(S_1^\pm) \left( r^\pm
  \right)^{-\lambda_1} w_{-1,0,\pm}(\theta^\pm)$.
As done for $u_{0,1}$ in Section~\ref{sec:constr-terms-u_01}, we shall construct $u_{2,0,\pm}^\delta$ by lifting explicitly its singular behaviour. We remark that $(r^\pm)^{-\lambda_1} w_{-1,0,\pm}$  does not satisfy the homogeneous Helmholtz equation in $\OmegaTop \cup \OmegaBottom$ (by construction it satisfies the homogeneous Laplace equation). However, we can substitute it with a multiple of the function
\begin{equation}
  \label{eq:singular_macro_s_2_-1}
  \mY_1^\pm(r^\pm, \theta^\pm) =
  Y_{\lambda_1}(k_0 r^\pm) w_{-1,0,\pm}(\theta^\pm), 
\end{equation}
which  behaves like 
 $
  - \frac{\Gamma(\lambda_1)}{\pi}  \left( \frac{k_0 }{2}
  \right)^{-\lambda_1} \left( r^\pm \right)^{-\lambda_1} w_{-1,0,\pm}(\theta^\pm)
 $
in the vicinity of the two corners and satisfies the homogeneous Helmholtz equation in $\OmegaTop \cup \OmegaBottom$. It this then natural to define $u_{2,0}^\delta$  as
\begin{equation}
u_{2,0}^\delta = u_{2,0} := \sum_{\pm} \ell_{2,0,-1}^\pm(u_{0,0}) \chi_\LBottom^\pm \mY_1^\pm + \hat{u}_{2,0}, \quad \ell_{2,0,-1}^\pm(u_{0,0})  := -\pi  \left( \frac{\ell_{1}^\pm(u_{0,0}) \mathscr{L}_{-1}(S_1^\pm)}{\Gamma(\lambda_1) \Gamma(\lambda_1+1)} \right) \left( \frac{k_0}{2}\right)^{\lambda_2},
\end{equation}
the cut-off functions $ \chi_\LBottom^\pm$  being defined in~\eqref{definitionu01delta} and the function $\hat{u}_{2,0}$ being the only $\Hone(\Omega)$ solution to the following problem:
  \begin{equation}
    \label{eq:macroscopic_v_2_0}
    \left\lbrace
      \begin{array}{r@{\;}ll}
        -\Delta \hat{u}_{2,0} - k_0^2 \hat{u}_{2,0} & = \hat{f}_{2,0} &\quad  \mbox{in} \; \OmegaTop \cup \OmegaBottom,\\[0.5em]
        \left[ \hat{u}_{2,0}\right]_{\Gamma}= \left[ \partial_{x_2} \hat{u}_{2,0} \right]_{\Gamma}  &= 0 &\quad \mbox{on} \, \Gamma ,\\[0.5em]
     %   \left[ \partial_{x_2} \hat{u}_{2,0} \right]_{\Gamma} & = 0 &\quad \mbox{on} \,
     %                                                                 \Gamma , \\[0.5em]
        \nabla \hat{u}_{2,0}\cdot \bn & = 0 &\quad \mbox{on} \, \Gamma_N,
        \\[0.5em]
        \nabla\hat{u}_{2,0} \cdot \bn - \imath k_0 \hat{u}_{2,0} & = 0 &\quad \mbox{on} \, \Gamma_R,
      \end{array}
    \right. \quad    \hat{f}_{2,0} := \sum_{\pm} \ell_{2,0,-1}^\pm(u_{0,0}) \commu{\Delta}{\chi_\LBottom^\pm} \mY_1^\pm .
  \end{equation}
 The function $\hat{f}_{2,0}$ being in $\Ltwo(\Omega)$ (it is compactly supported), Proposition~\ref{prop:existence_uniqueness_macro} ensures the well-posedness of~\eqref{eq:macroscopic_v_2_0} in $\Hone(\Omega)$. 
In the vicinity of the two corners, 
$\hat{f}_{2,0}$ vanishes, so that  
\begin{equation*}
  \hat{u} =  \sum_{m=0}^\infty \ell_m^\pm(u_{2,0}) J_{\lambda_m} (k_0 r^\pm)
  w_{m,0,\pm}(\theta^\pm), \quad \ell_m^\pm(u_{2,0}) \in \IC.
\end{equation*}
Using the radial decomposition of the
Bessel functions, coupled with the formula
\begin{equation*}
  Y_{\lambda_1}(k_0 r) = \frac{J_{\lambda_1}(k_0
 r) \cos(\lambda_1 \pi) - J_{-\lambda_1}(k_0
  r)}{\sin(\lambda_1 \pi)} \quad \text{(see Equation~(9.1.2) of Ref.~\cite{MR0167642})},
\end{equation*}
we see that
\begin{multline}\label{Expansionu20}
 u_{2,0}(r^\pm, \theta^\pm) =  \frac{\ell_{1}^\pm(u_{0,0})
    (k_0/2)^{\lambda_1}}{\Gamma(\lambda_1+1)} \mathscr{L}_{-1}(S_1^\pm) \left( r^\pm
  \right)^{-\lambda_1} w_{-1,0,\pm}(\theta^\pm) +
  \ell_{0}^\pm(u_{2,0})
  \\  +  \ell_{2,0,1}(u_{2,0})   \left( r^\pm \right)^{\lambda_1}
  w_{1,0,\pm}(\theta^\pm) + O(r^{\max(\lambda_2, -\lambda_1+2)}),
\end{multline}
where
\begin{equation}
\ell_{2,0,1}(u_{2,0}) =  \frac{\ell_{1}^\pm(u_{2,0})
    (k_0/2)^{\lambda_1}}{\Gamma(\lambda_1+1)} + 
    \frac{\ell_{2,0,-1}^\pm(u_{0,0}) \cos(\lambda_1 \pi)
      (k_0/2)^{\lambda_1}}{\sin(\lambda_1 \pi) \Gamma(\lambda_1+1)}.
\end{equation}
By construction, the first term of the right hand side of~\eqref{eq:raccordOrdrelambda2} is counterbalanced by the first term of \eqref{Expansionu20} multiplied by $\delta^{\lambda_2}$.
\subsubsection{Construction of $U_{2,0,\pm}^\delta$}
Writing the matching
condition~\eqref{eq:raccordOrdrelambda2}
with respect to the microscopic variable and taking into account \eqref{Expansionu20}, we obtain
\begin{equation}
  \nonumber
  \delta^{\lambda_2}  \left( \frac{\ell_{2}^\pm(u_{0,0})
      (k_0/2)^{\lambda_2}}{\Gamma(\lambda_2+1)}  \left(  \left( R^\pm
      \right)^{\lambda_2} w_{2,0,\pm}(\theta^\pm) +
      (R^\pm)^{\lambda_2-1}  w_{2,1,\pm}(\theta^\pm) \right)+
    \ell_{0}^\pm(u_{2,0}) \right)  \approx \delta^{\lambda_2} U_{2,0}^\delta.
\end{equation}

We then see that $U_{2,0,\pm}^\delta$ has to grow up like
\begin{equation*}
 \frac{\ell_{2}^\pm(u_{0,0})
  (k_0/2)^{\lambda_2}}{\Gamma(\lambda_2+1)}  \left(  \left( R^\pm
\right)^{\lambda_2} w_{2,0,\pm}(\theta^\pm) +   (R^\pm)^{\lambda_2-1}  w_{2,1,\pm}(\theta^\pm) \right) +  \ell_{0}^\pm(u_{2,0}). 
\end{equation*}
Of course,
$\left(R^\pm \right)^{\lambda_2} w_{2,0,\pm}(\theta^\pm) +
(R^\pm)^{\lambda_2-1} w_{2,1,\pm}(\theta^\pm) $
does not satisfy the homogeneous
problem~\eqref{eq:macroscopic_u_2_0}-(right). However,
Proposition~\eqref{prop:existence_uniqueness_NF_Sn} ensures the
existence of a function $S_2^\pm$, that satisfies
\eqref{eq:macroscopic_u_2_0}-(right) and such that
$S_2^\pm - \left( R^\pm \right)^{\lambda_2} w_{2,0,\pm}(\theta^\pm)
+(R^\pm)^{\lambda_2-1} w_{2,1,\pm}(\theta^\pm)$
tends to $0$ as $R^\pm$ tends to infinity
($\lambda_2 \notin \mathbb{N}$). Consequently, it is natural to define
$U_{2,0,\pm}^\delta$ as
\begin{equation} \label{DefinitionU20}
  U_{2,0,\pm}^\delta = U_{2,0,\pm} = \frac{\ell_{2}^\pm(u_{0,0})
  (k_0/2)^{\lambda_2}}{\Gamma(\lambda_2+1)} S_2^\pm + \ell_{0}^\pm(u_{2,0}).
\end{equation}
Outside the periodic layer, $ U_{2,0,\pm} $ admits the following asymptotic expansion at infinity
\begin{multline}\label{ExpansionU20}
U_{2,0,\pm} =\frac{\ell_{2}^\pm(u_{0,0})
  (k_0/2)^{\lambda_2}}{\Gamma(\lambda_2+1)}  \left(R^\pm \right)^{\lambda_2} w_{2,0,\pm}(\theta^\pm) \\+ \frac{\ell_{2}^\pm(u_{0,0})
  (k_0/2)^{\lambda_2}}{\Gamma(\lambda_2+1)} (R^\pm)^{\lambda_2-1}  w_{2,1,\pm}(\theta^\pm) + \ell_{0}^\pm(u_{2,0})  \\
  + \frac{\ell_{2}^\pm(u_{0,0})
  (k_0/2)^{\lambda_2}}{\Gamma(\lambda_2+1)} (R^\pm)^{-\lambda_1}  \mathscr{L}_{-1}(S_2^\pm) w_{-1,0,\pm}(\theta^\pm) +O((R^\pm)^{\lambda_2-2} \ln R^\pm ).
\end{multline}

\subsection{Construction of the terms $u_{1,1}^\delta$ and $U_{1,1,\pm}^\delta$}
\label{sec:constr-terms-u_11}
Reminding that $u_{1,1}^\delta$ fulfills the jump conditions
\eqref{eq:macroscopic_u_n_1} and that $u_{1,0}=0$ (see~\eqref{Defintionu10delta}),
$u_{1,1}^\delta$ and $U_{1,1,\pm}^\delta$ satisfy the following problems
\begin{equation}
  \label{eq:macroscopic_u_1_1}
  \left \lbrace
    \begin{array}{r@{\;}ll}
      -\Delta u_{1,1}^\delta - k_0^2 u_{1,1}^\delta & = 0 &\quad  \mbox{in} \; \OmegaTop \cup \OmegaBottom,\\[0.5em]
      \left[ u_{1,1}^\delta \right]_{\Gamma} = \left[ \partial_{x_2} u_{1,1}^\delta \right]_{\Gamma} &= 0 &\quad \mbox{on} \, \Gamma ,\\[0.5em]
   %   \left[ \partial_{x_2} u_{1,1}^\delta \right]_{\Gamma} & = 0 &\quad \mbox{on} \, \Gamma, \\[0.5em]
      \nabla u_{1,1}^\delta \cdot \bn & = 0 &\quad \mbox{on} \, \Gamma_N,
      \\[0.5em]
      \nabla u_{1,1}^\delta \cdot \bn - \imath k_0 u_{1,1}^\delta & = 0 &\quad \mbox{on} \, \Gamma_R,
    \end{array}
  \right. \quad \mbox{and} \quad
  \left \lbrace
    \begin{array}{rcll}
      -\Delta U_{1,1,\pm}^\delta & = &  0 &\mbox{in} \; \widehat{\Omega}^\pm,\\
      \partial_\bn U_{1,1,\pm}^\delta & = & 0  & \mbox{on} \; \partial
                                               \widehat{\Omega}^\pm,\\
    \end{array}
  \right.
\end{equation}
Outside the thin periodic layer, the matching condition~\eqref{eq:matching_conditions_global_main} written up to order $\delta^{\lambda_1+1}$ gives
%\begin{equation}
%u_{0,0} + \delta u_{0,1} + \delta^{\lambda_2} \, u_{2,0} + \delta^{\lambda_1 +1} u_{1,0}^\delta \approx U_{0,0} + \delta^{\lambda_1} U_{1,0} + \delta \, U_{0,1} + \delta^{\lambda_2} U_{2,0} + \delta^{\lambda_1 +1} U_{1,0}^\delta
%\end{equation}
%which, 
  \begin{equation}
   \delta  \frac{\ell_{1}^\pm(u_{0,1})
      (k_0/2)^{\lambda_1}}{\Gamma(\lambda_1+1)}  \left( r^\pm \right)^{\lambda_1} w_{1,0,\pm}(\theta^\pm)+  \delta^{\lambda_1+1} u_{1,1}^\delta(r^\pm, \theta^\pm)  \approx
    \delta^{\lambda_1+1} U_{1,1,\pm}^\delta  \left(\frac{r^\pm}{\delta}, \theta^\pm\right). 
  \end{equation}
  
A analogous analysis than the one made in Section~\ref{sec:constr-terms-u_10} yields 
\begin{equation} 
u_{1,1}^\delta= u_{1,1} = 0 \quad \mbox{and} \quad  U_{1,1,\pm}^\delta = U_{1,1,\pm} = \frac{\ell_{1}^\pm(u_{0,1})
    (k_0/2)^{\lambda_1}}{\Gamma(\lambda_1+1)}  S_1^\pm.
\end{equation} 
% and 
%\begin{equation} \label{DefinitionU11}
%  U_{1,1,\pm}^\delta = U_{1,1,\pm} = \frac{\ell_{1}^\pm(u_{0,1})
 %   (k_0/2)^{\lambda_1}}{\Gamma(\lambda_1+1)}  S_1^\pm.
%\end{equation}
Far from the periodic layer, the asymptotic behaviour of $U_{1,1,\pm}$ is given by
\begin{equation}\label{AsymptoticU11}
U_{1,1,\pm}(\R^\pm, \theta^\pm) = \frac{\ell_{1}^\pm(u_{0,1})
    (k_0/2)^{\lambda_1}}{\Gamma(\lambda_1+1)}  (R^\pm)^{\lambda_1} w_{1,0,\pm}(\theta^\pm) + O((R^\pm)^{\lambda_1 -1}).
\end{equation}
\subsection{Construction of the terms $u_{3,0}^\delta$ and $U_{3,0,\pm}^\delta$ for $\Theta > \frac{3\pi}{2}$ }
\label{sec:constr-terms-u_30}
Reminding that $u_{3,0}^\delta$ fulfills the jump conditions~\eqref{eq:macroscopic_u_n_0},
$u_{3,0}^\delta$ and $U_{3,0,\pm}^\delta$ satisfy
\begin{equation}
  \label{eq:macroscopic_u_3_0}
  \left \lbrace
    \begin{array}{r@{\;}ll}
      -\Delta u_{3,0}^\delta - k_0^2 u_{3,0}^\delta & = 0 &\quad  \mbox{in} \; \OmegaTop \cup \OmegaBottom,\\[0.5em]
      \left[ u_{3,0}^\delta \right]_{\Gamma} = \left[ \partial_{x_2} u_{3,0}^\delta \right]_{\Gamma} &= 0  &\quad \mbox{on} \, \Gamma ,\\[0.5em]
      %\left[ \partial_{x_2} u_{3,0}^\delta \right]_{\Gamma} & = 0  &\quad \mbox{on} \, \Gamma, \\[0.5em]
      \nabla u_{3,0}^\delta \cdot \bn & = 0 &\quad \mbox{on} \, \Gamma_N,
      \\[0.5em]
      \nabla u_{3,0}^\delta \cdot \bn - \imath k_0 u_{3,0}^\delta & = 0 &\quad \mbox{on} \, \Gamma_R,
    \end{array}
  \right. \quad \mbox{and} \quad
   \left \lbrace
  \begin{array}{rcll}
      -\Delta U_{3,0,\pm}^\delta & = &  0 &\mbox{in} \; \widehat{\Omega}^\pm,\\
      \partial_\bn U_{3,0,\pm}^\delta & = & 0  & \mbox{on} \; \partial
      \widehat{\Omega}^\pm,\\
   \end{array}
   \right.
   \end{equation}
Outside the periodic layer, 
collecting the asymptotic representation~\eqref{Expansionu00}-\eqref{Expansionu00}-\eqref{eq:expansion_u_0_1_expansion} of the far field terms, the defintions \eqref{definitionU00}-\eqref{DefinitionU01} of $U_{0,0}$ and $U_{0,1}$ and the asymptotic expansions~\eqref{ExpansionU10}-\eqref{ExpansionU20}-\eqref{AsymptoticU11} of $U_{1,0}$, $U_{2,0}$, $U_{1,1}$, the matching condition~\eqref{eq:matching_conditions_global_main} written up to
order $\delta^{\lambda_3}$ becomes
\begin{multline}
  \label{eq:matching_conditions_order_lambda3_pm_y1_positive2}
  \frac{\ell_{3}^\pm(u_{0,0})
    (k_0/2)^{\lambda_3}}{\Gamma(\lambda_3+1)}   \left( \left( r^\pm \right)^{\lambda_3}
    w_{3,0,\pm}(\theta^\pm)  + \delta (r^\pm)^{\lambda_3-1}  w_{3,1,\pm}
  \right) \\ + \delta^{\lambda_2}  
  \ell_{2,0,1}(u_{2,0})  \left( r^\pm \right)^{\lambda_1} 
  w_{1,0,\pm} \\ + \delta^{\lambda_3}u_{3,0}^\delta   \approx  
  \delta^{\lambda_3} \left( \sum_{i=1}^2 \frac{\ell_{3-i}^\pm(u_{0,0})
      (k_0/2)^{\lambda_{3-i}}}{\Gamma(\lambda_{3-i}+1)}  (r^\pm)^{-\lambda_i}  \mathscr{L}_{-i}(S_{3-i}^\pm) w_{-i,0,\pm}  \right)  + \delta^{\lambda_3}U_{3,0}^\delta.
\end{multline}

\subsubsection{Construction of  the macroscopic term $u_{3,0}^\delta$}
In view of the right hand side of~\eqref{eq:matching_conditions_order_lambda3_pm_y1_positive2}, we
remark that $u_{3,0}^\delta$ has two singular contributions of the
form $(r^\pm)^{-\lambda_2}$ and $(r^\pm)^{-\lambda_1}$. Defining
\begin{equation}
  \label{eq:singular_macro_Y2}
  \mY_2^\pm(r^\pm, \theta^\pm) =
  Y_{\lambda_2}(k_0 r^\pm) w_{-2,0,\pm}(\theta^\pm), \quad  \ell_{3,0,-i}^\pm(u_{0,0}) = - \pi \mathscr{L}_{-i}(S_{3-i}^\pm) \ell_{3-i}^\pm(u_{0,0})
  \frac{(k_0/2)^{\lambda_3}}{\Gamma(\lambda_i) \Gamma(\lambda_{3-i}+1)},
\end{equation}
 the function $\sum_{i=1}^2\delta^{\lambda_3/2}
\ell_{3,0,-i}^\pm(u_{0,0}) \mY_i^\pm(r^\pm, \theta^\pm)$ ($\mY_1^\pm$
defined in~\eqref{eq:singular_macro_s_2_-1}) can counterbalance the
first two terms of the right hand side of
\eqref{eq:matching_conditions_order_lambda3_pm_y1_positive2}. This
remark leads us to define $u_{3,0}^\delta$ as
\begin{equation}
u_{3,0}^\delta = u_{3,0}  := \hat{u}_{3,0} + \sum_{\pm} \sum_{i=1}^2  \ell_{3,0,-i}^\pm(u_{0,0}) \chi_\LBottom^\pm \mY_i^\pm,
\end{equation} 
where $\chi_\LBottom^\pm$ is defined in~\eqref{definitionu01delta} and the function $\hat{u}_{3,0}$ is the unique function of $\Hone(\Omega)$ satisfying 
  \begin{equation}
    \label{eq:macroscopic_v_3_0}
    \left\lbrace
      \begin{array}{r@{\;}ll}
        -\Delta \hat{u}_{3,0} - k_0^2 \hat{u}_{3,0} & = \hat{f}_{3,0} &\quad  \mbox{in} \; \OmegaTop \cup \OmegaBottom,\\[0.5em]
        \left[ \hat{u}_{3,0}\right]_{\Gamma}= \left[ \partial_{x_2} \hat{u}_{3,0} \right]_{\Gamma}  &= 0 &\quad \mbox{on} \, \Gamma ,\\[0.5em]
     %   \left[ \partial_{x_2} \hat{u}_{3,0} \right]_{\Gamma} & = 0 &\quad \mbox{on} \,
     %                                                                 \Gamma , \\[0.5em]
        \nabla \hat{u}_{3,0}\cdot \bn & = 0 &\quad \mbox{on} \, \Gamma_N,
        \\[0.5em]
        \nabla\hat{u}_{3,0} \cdot \bn - \imath k_0 \hat{u}_{3,0} & = 0 &\quad \mbox{on} \, \Gamma_R,
      \end{array}
    \right.
  \end{equation}
  with
  \begin{equation*}
    \hat{f}_{3,0} := \sum_{i=1}^2 \sum_{\pm} \ell_{3,0,-i}^\pm(u_{0,0}) \commu{\Delta}{\chi_\LBottom^\pm} \mY_i^\pm .
  \end{equation*}
 The well-posedness of~\eqref{eq:macroscopic_v_3_0} directly follows from Proposition~\ref{prop:existence_uniqueness_macro}. In the matching zones, 
 \begin{multline}\label{Expansionu30}
   u_{3,0}(r^\pm, \theta^\pm) = -
   \frac{\ell_{3,0,-2}^\pm(u_{0,0})
    (k_0/2)^{-\lambda_2}}{\sin(\lambda_2 \pi) \Gamma(1-\lambda_2)}
  \left( r^\pm \right)^{-\lambda_2} w_{-2,0,\pm}(\theta^\pm) \\ -
   \frac{\ell_{3,0,-1}^\pm(u_{0,0})
    (k_0/2)^{-\lambda_1}}{\sin(\lambda_1 \pi) \Gamma(1-\lambda_1)}
  \left( r^\pm \right)^{-\lambda_1} w_{-1,0,\pm}(\theta^\pm) +
   \ell_{0}^\pm(u_{3,0}) + O(r^{-\lambda_2+2}),
\end{multline}

\subsubsection{Construction of $U_{3,0,\pm}^\delta$}
Writing the matching
condition~\eqref{eq:matching_conditions_order_lambda3_pm_y1_positive2} in term of the microscopic variables 
and taking into account \eqref{Expansionu30}, we obtain
\begin{multline}
\delta^{\lambda_3}  \left( \frac{\ell_{3}^\pm(u_{0,0})
    (k_0/2)^{\lambda_3}}{\Gamma(\lambda_3+1)}   \left( \left( r^\pm \right)^{\lambda_3}
  w_{3,0,\pm}(\theta^\pm)  + (R^\pm)^{\lambda_3-1}  w_{3,1,\pm}(\theta^\pm)  \right) \right. \\  + \left. 
    \ell_{2,0,1}(u_{2,0})  \left( R^\pm \right)^{\lambda_1}
  w_{1,0,\pm}(\theta^\pm)  +  \ell_{0}^\pm(u_{3,0})  \right) \approx \delta^3 U_{3,0}^\delta.
\end{multline}
As in Section~\ref{sec:constr-terms-u_20}, it is natural to define $U_{3,0,\pm}^\delta$ as
\begin{multline} \label{DefinitionU30}
  U_{3,0,\pm}^\delta =  U_{3,0,\pm} :=  \frac{\ell_{3}^\pm(u_{0,0})
  (k_0/2)^{\lambda_3}}{\Gamma(\lambda_3+1)} S_3^\pm  \\ + \frac{\big(
    \ell_{1}^\pm(u_{2,0}) \sin \lambda_1\pi +
    \ell_{2,0,-1}^\pm(u_{2,0}) \cos \lambda_1\pi \big)
    (k_0/2)^{\lambda_2}}{\Gamma(\lambda_1+1) \sin \lambda_1\pi}
  S_1^\pm + \ell_{0}^\pm(u_{3,0}).
\end{multline}
\subsection{The 'automatic' matching inside the layer }\label{SectionRaccordCouche}
We end this part by showing, that far and near field expansions automatically match in the matching areas. For the sake of concision, we consider the case $\theta \in (\pi, 3\pi/2)$ and we only investigate the matching area located in the vicinity of the right corner $\mathbf{x}_{O}^+$. \\

\noindent On the one hand, collecting the results of the present section and Appendix~\ref{SectionTransmissionConditionsBoundaryLayer}, we see that the boundary layer correctors are given by 
\begin{equation}\label{BLdef}
  \begin{aligned}
    \dsp \Pi_{0,0} (x_1, \mX) & = \langle u_{0,0} \rangle_{\Gamma}(x_1)
    V_0(\mX), \quad \Pi_{1,0} = 0, \\ \dsp \Pi_{0,1}(x_1, \mX) (x_1,
    \mX)  & = \langle u_{0,1} \rangle_\Gamma(x_1)\,    V_{0}(\mX)  \,
    {+\, \partial_{x_1} \langle u_{0,0} \rangle_\Gamma(x_1) \,
      V_{1,1}(\mX)}  \\ & \dsp \qquad \;+ \; \langle \partial_{x_2} u_{0,0}
    \rangle_\Gamma(x_1) \, V_{1,2}(\mX), \\
    Pi_{2,0} (x_1, \mX)  & =    \langle u_{2,0} \rangle_{\Gamma}(x_1)\,  V_0(\mX), \quad \quad \Pi_{1,1} = 0. 
  \end{aligned}
\end{equation}  
Then, the asymptotic expansion for the boundary layer in the matching areas can be directly written introducing the asymptotic formula~\eqref{Expansionu00}-\eqref{eq:expansion_u_0_1_expansion}-\eqref{Expansionu20} of the macroscopic terms $u_{0,0}$, $u_{0,1}$ and $u_{2,0}$ into~\eqref{BLdef}.  Writting the obtained asymptotic expansions in term of the microscopic variables, noticing that $V_0 = W_0^{\mathfrak{t}}$, $V_{1,1} = W_1^{\mathfrak{t}}$ and  $V_{1,2} = W_1^{\mathfrak{n}}$ (defined in~\eqref{ProblemWpt}-\eqref{ProblemWpn}), and summing over $(n,q) \in \N_2$, we obtain
   \begin{equation}
   \begin{aligned}
  &  \sum_{(n,q)\in \N_2}  \delta^{\lambda_n +q} \, \Pi_{n,q}   \approx 
   \ell_{0}^\pm(u_{0,0}) W_0^{\mathfrak{t}} \\
& \qquad  + \delta^{\lambda_1} \left\{  \frac{\ell_{1}^\pm(u_{0,0}) 
    (k_0/2)^{\lambda_1}}{\Gamma(\lambda_1+1)} \right\} \big\{ |X_1|^{\lambda_1} 
 p_{1,0,+}+ |X_1|^{\lambda_1-1}   p_{1,1,+}  \\ & \qquad\qquad+ |X_1|^{-\lambda_1}  \mathscr{L}_{-1}(S_1^\pm)  p_{-1,0,+} W_0^{\mathfrak{t}}   \big\} \\
 &   \qquad+ \delta  \, \ell_{0}^\pm(u_{0,1})  W_0^{\mathfrak{t}}  \\
    &   \qquad+ \delta^{\lambda_2} \left\{  \frac{\ell_{2}^\pm(u_{0,0})
    (k_0/2)^{\lambda_2}}{\Gamma(\lambda_2+1)} \right\} \left\{ |X_1|^{\lambda_2} 
  p_{2,0,+} + |X_1|^{\lambda_2-1}  p_{2,1,+}    \right\}  +  \delta^{\lambda_2} \ell_{0}^\pm(u_{2,0}) \, W_0^{\mathfrak{t}} \\
 &  \qquad+ \delta^{\lambda_1+1} \left\{   \frac{\ell_{1}^\pm(u_{0,1})(k_0/2)^{\lambda_1}}{\Gamma(\lambda_1+1)}  \right\} |X_1|^{\lambda_1} 
   p_{1,0,+},
   \end{aligned} \label{FarFieldLayerFin}
    \end{equation} 
the function $p_{n,r,t,\pm}$ being defined in~\eqref{sec:defin-prof-funct_p}.  \\

\noindent On the other hand, using the definitions~\eqref{definitionU00}-\eqref{DefinitionU10}-\eqref{DefinitionU01}-\eqref{DefinitionU20} of the near field terms, the truncated series of the near field  is given by
\begin{multline}\label{NearFieldLayer}
 \sum_{(n,q)\in \N_2}  \delta^{\lambda_n +q} \, \Pi_{n,q}    =    \ell_0^+(u_{0,0}) + \delta^{\lambda_1}  \frac{\ell_{1}^+(u_{0,0})
    (k_0/2)^{\lambda_1}}{\Gamma(\lambda_1+1)}  S_1^+  +  \delta \ell_0^+(u_{0,1})\\
     + \delta^{\lambda_2} \left( \frac{\ell_{2}^+(u_{0,0})
  (k_0/2)^{\lambda_2}}{\Gamma(\lambda_2+1)} S_2^+ + \ell_{0}^+(u_{2,0}) \right)  
    + \delta^{\lambda_1 + 1} \frac{\ell_{1}^\pm(u_{0,1})
    (k_0/2)^{\lambda_1}}{\Gamma(\lambda_1+1)}  S_1^+ .
 \end{multline}
 Introducing the asymptotic expansions~\eqref{eq:block_decomposition_Sn} of the functions $S_1^+$ and $S_2^+$ in the vicinity of the periodic layer into~\eqref{NearFieldLayer}, we see that the near field expansions~\eqref{NearFieldLayer} and ~\eqref{FarFieldLayerFin} coïncide (up to a given order). 
%\begin{remark}
 % In the particular case $\Theta = \frac{3\pi}{2}$, $\lambda_3=2$ and
 % Proposition~\ref{prop:existence_uniqueness_NF_Sn} shows that the
  %function $S_3^\pm$, and therefore the near-field $U_{3,0}^\delta$,
 % has a $\ln R^\pm$ behavior. Writing it in the intermediate variable
 % gives a $\ln \frr^\pm - \frac{1}{2} \ln \delta^\pm$
 % contribution. This contribution has to be canceled by a
 % $\ln r^\pm$ contribution in the far field $u_{3,0,\pm}^\delta$. Not
 % only will the function $u_{3,0,\pm}^\delta(\sqrt{\delta} \frr^\pm)$
 % have a contribution in $Y_0(\sqrt{\delta} \frr^\pm)$, but this
  %function will also have a real dependency with respect to $\ln \delta$.
%\end{remark}

\section{Error estimates}
\label{SectionErrorEstimates}
\label{sec:error_estimates}

To finish this paper, we give the sketch of the proof of
Theorem~\ref{theo:error_estimate_optimal}.  As usual for this kind of
work (see \eg Section~6 of Ref.~\cite{ResearchReportAKB}, Section~3
of Ref.~\cite{fente2}, Section~2.5 of Ref.~\cite{Jol-Sem-2008}),
the proof of the previous result is based on the construction of an
approximation $u_{N_0}^\delta$ of $u^\delta$ in the whole domain
$\Omega^\delta$. To do so, we define the following four truncated
series (at order $N_0$), corresponding to the truncated series of the
macroscopic terms, the boundary layer terms and the near field terms:
\begin{enumerate}
\item[-] The truncated series $u_{\text{macro},N_0}^\delta$ of the macroscopic terms:  the macroscopic
  approximation is defined by 
  \begin{equation}
    \label{eq:u_macro_N0}
    u_{\text{macro},N_0}^\delta(\mx) = \chi_{\text{macro}}^\delta(\mx)
    \sum_{(n,q)\in \IN_{N_0}} \delta^{\lambda_n+q} u_{n,q}^\delta(\mx),
  \end{equation}
  where the set $\IN_{N_0}$ is the set of indexes $(n,q)\in \N^2$
    for which $\lambda_n+q < N_0$, and the macroscopic cut-off function
  $ \chi_{\text{macro}}^\delta$ is given by
  \begin{multline}
    \label{eq:chi_macro}
    \chi_{\text{macro}}^\delta(\mx) = \chi_+ \left(
      \frac{x_1-L}{\delta} \right) \chi_- \left( \frac{x_1+L}{\delta}
    \right) \chi \left( \frac{x_2}{\delta} \right) \\ + \sum_{\pm}
    \chi_{\text{macro},\pm} \left( \frac{x_1 \mp L}{\delta},
      \frac{x_2}{\delta} \right) \left( 1 - \chi_\pm \left(
        \frac{x_1 \mp L}{\delta} \right) \right).
  \end{multline}
 We notice that the function $ \chi_{\text{macro}}^\delta$  is equal to $1$ for $|x_1| > \LBottom$ and coincides
  with $\chi \left( \frac{x_2}{\delta} \right)$ in the region
  $|x_1| < \LBottom - \delta$ (The cut-off functions $\chi$ and
  $\chi_\pm$ are defined in \eqref{defchi} and
  \eqref{eq:defchipm}, while the cut-off functions $\chi_{\text{macro},+}$, represented on Fig.~\ref{fig:chi_macro_plus}, satisfies \eqref{eq:chi_macro_property}). 
 \item[-] The truncated  series $\Pi_{N_0}^\delta$ of the periodic correctors is
  given by
  \begin{equation}
    \label{eq:Pi_N0}
    \Pi_{N_0}^\delta(\mx) = \chi_+ \left(
      \frac{x_1-L}{\delta} \right) \chi_- \left( \frac{x_1+L}{\delta}
    \right) \chi \left( \frac{2x_2}{\min(\HBottom,\HTop)} \right) 
    \sum_{(n,q)\in \IN_{N_0}} \delta^{\lambda_n+q} \Pi_{n,q}^\delta(\mx).
  \end{equation}
  The use of the function
  $\chi_+ \left( \frac{x_1-L}{\delta} \right) \chi_- \left(
    \frac{x_1+L}{\delta} \right)$ permits us to localize the function
  $\Pi_{N_0}^\delta(\mx)$ in the domain $|x_1|<\LBottom$ while the
  introduction of the function $\chi \left(
    \frac{2x_2}{\min(\HBottom,\HTop)} \right)$ ensures that
  $\Pi_{N_0}^\delta(\mx)$ satisfies Neumann boundary condition on
  $\Gamma_N$.
\item[-]  The truncated near field series $U_{N_0,\pm}^\delta$ are given
  by
  \begin{equation}
    \label{eq:U_N0_pm}
    U_{N_0,\pm}^\delta = \sum_{(n,q)\in \IN_{N_0}}
    \delta^{\lambda_n+q} U_{n,q,\pm}^\delta \left(
      \frac{\mx-\mx_O^\pm}{\delta} \right).
  \end{equation}
%\end{itemize}
\end{enumerate}

Based, on these truncated series, the global approximation $u_{N_0}^\delta$ is defined by
\begin{equation}
  \label{eq:global_approximation}
  u_{N_0}^\delta = \chi_+^\delta \, U_{N_0,+}^\delta \ +
  \ \chi_-^\delta \, U_{N_0,-}^\delta \ + \
  (1-\chi_+^\delta-\chi_-^\delta) \, (u_{\text{macro},N_0}^\delta
  + \Pi_{N_0}^\delta),
\end{equation}
where $\chi_\pm^\delta(\mx) = \chi ( {\lvert \mx-\mx_O^\pm
      \rvert}/{\sqrt{\delta}} )$.
We point out that  $u_{N_0}^\delta$ coincides
with $U_{N_0,\pm}^\delta$ in the vicinity of the two corners, with
$\Pi_{N_0}^\delta$ in the vicinity of the layer and with
$u_{\text{macro},N_0}^\delta$ away from the corner and the periodic
layer. 
 \begin{remark}
The overall approximation $u_{N_0}^\delta$ can be computed for any real number $\N_0$ as soon as the terms of the the far and near field expansions are defined. In Section~\ref{sec:constr-first-terms}, we only constructed the first terms of these expansions, but the next order terms can naturally be derived using the same methodology. 
\end{remark}
The overall approximation being constructed, it remains to evaluate the $\Hone$-norm of the error
$e_{N_0}^\delta = u^\delta - u_{N_0}^\delta$ in $\Omega^\delta$. It is
in fact sufficient to estimate the residue $(\Delta+(k^\delta)^2) e_{N_0}^\delta$ and
the Neumann trace $\partial_\bn e_{N_0}^\delta$. Indeed, the estimation
of $\xnorm{e_{N_0}^\delta}{\Hone(\Omega^\delta)}$ directly results
from a straightforward modification of the uniform stability result 
\eqref{eq:prop_norm_u_norm_f}
(Proposition~\ref{prop:existence_uniqueness_u_delta}): there exists a
constant $C>0$ independent of $\delta$ (but depending on other
  parameters such as $N_0$ and hole shape) such that, for $\delta$ small enough,
\begin{equation}
  \label{eq:estimate_e_by_residue}
  \xnorm{e_{N_0}^\delta}{\Hone(\Omega^\delta)} \leqslant C \big(
  \xnorm{(\Delta+(k^\delta)^2) e_{N_0}^\delta}{\Ltwo(\Omega^\delta)} + \xnorm{\partial_\bn
    e_{N_0}^\delta}{\Ltwo(\Gamma^\delta)} \big).
\end{equation}
Similarly to the proof of Proposition~6.3 in
Ref.~\cite{ResearchReportAKB}, we decompose the error of the
residue into a modeling error (measuring how the truncated far and
near field expansions fail to satisfies the Helmholtz equation and the
Neumann boundary condition) and a matching error (measuring the
difference between the far and near field expansions in the matching
areas), and we obtain the following proposition:
\begin{proposition}
  Let $N_0 \in \IR$. There exists a constant $C \geq 0$, a constant $\kappa = \kappa(N_0) > 0$ and a constant
  $\delta_0 > 0$ such that, for any $\delta \in (0,\delta_0)$,
  \begin{equation}
    \label{eq:prop_error_estimate}
    \xnorm{(\Delta+(k^\delta)^2) e_{N_0}^\delta}{\Ltwo(\Omega^\delta)} + \xnorm{\partial_\bn
      e_{N_0}^\delta}{\Ltwo(\Gamma^\delta)} \leqslant C (\ln
    \delta)^\kappa \delta^{\frac{N_0}{2}-\frac{5}{2}}.
  \end{equation}
As a consequence, there exists
  a constant $C > 0$, a constant $\kappa = \kappa(N_0) > 0$ and a constant
  $\delta_0 > 0$ such that, for any $\delta \in (0,\delta_0)$,
  \begin{equation}
    \label{eq:lema:error_estimate_global_error}
    \xnorm{e_{N_0}^\delta}{\Hone(\Omega^\delta)} \leqslant C (\ln
    \delta)^\kappa \delta^{\frac{N_0}{2}-\frac{5}{2}}.
  \end{equation}
\end{proposition}
Finally, since $e_{N_0}^\delta$ coincides with $u^\delta - \sum_{(n,q) \in
  \IN_{N_0}} \delta^{\lambda_n+q} u_{n,q}^\delta$ in $\Omega_\alpha$
for $\delta$ small enough, Theorem~\ref{theo:error_estimate_optimal} follows
from \eqref{eq:lema:error_estimate_global_error} and the triangular inequality.

\section*{Acknowledgment}

The authors would like to thank Robert Gruhlke (TU Berlin) for helpful
discussions and the support of the implementation in Concepts related
to the project. They gratefully acknowledge the financial support by
the Einstein Foundation Berlin (grant number IPF-2011-98).

\begin{appendix}

\section{Proof of Proposition~\ref{prop:existence_uniqueness_u_delta}}\label{AppendixStability}
The variational formulation associated with~\eqref{eq:perturbed_Helmholtz} writes as follows: find $u^\delta \in \Hone(\Omega^\delta)$ such that, 
\begin{equation}\label{FormVar}
\forall v \in \Hone(\Omega^\delta), \quad \quad a^\delta( u^\delta , v ) = \langle \partial_n u_{\text{inc}}  - \imath k_0 u_{\text{inc}}, v \rangle_{H^{-1/2}(\Gamma_R^+), H^{1/2}(\Gamma_R^+)} 
\end{equation}
where 
\begin{equation*}
a^\delta( u , v )  = \int_{\Omega^\delta} \nabla u  \cdot \overline{\nabla v} \, dx  - \int_{\Omega^\delta} (k^\delta)^2 u \, \overline{ v} \, dx \; - \; \imath k_0 \int_{\Gamma_R}  u \overline{\nabla v}  dx, 
\end{equation*}  
and $\langle \cdot, \cdot \rangle_{H^{-1/2}(\Gamma_R^+), H^{1/2}(\Gamma_R^+)} $ stands for the duality pairing between $H^{-1/2}(\Gamma_R^+)$ and  $H^{1/2}(\Gamma_R^+)$ extending the $L^2(\Gamma_R^+)$ iner-product.  
It is easily seen that Problem~\eqref{eq:perturbed_Helmholtz} is a Fredholm-type problem (Theorem 6.6 in ~\cite{BrezisAnglais}). Let us prove that is has a unique solution. Assume that 
$$
a^\delta( u , v )   = 0  \quad \forall v \in \Hone(\Omega^\delta).
$$
Then, taking $v = u^\delta$ leads to $u^\delta = 0$ on $\Gamma_R^\pm$. Since $u^\delta = 0$ satisfies a Robin-type boundary condition on $\Gamma_R$, we deduce that $\partial_n u^\delta =0$ on $\Gamma_R^\pm$. It then follows from the unique continuation theorem that $u^\delta =0$ in the whole domain $\Omega^\delta$.\\

\noindent It remains  to prove the uniform stability estimate~\eqref{eq:prop_norm_u_norm_f}. The proof is by contradiction. If \eqref{eq:prop_norm_u_norm_f} does not hold, there exists a sequence $\delta_n$ going to $0$ as $n$ tends to $+ \infty$, and a sequence $u_n \in \Hone(\Omega^\delta)$ such that
 \begin{equation}
 \| u_n \|_{\Hone(\Omega^{\delta_n})} = 1 \quad \mbox{and} \quad  \forall v \in \Hone(\Omega^{{\delta_n}}),  \lim_{n\rightarrow+\infty}  a^{\delta_n}( u_n , v ) = 0.
 \end{equation} 
 First, we construct an extension of $\tilde{u}_n$ of $u_n$ belonging to $\Hone(\Omega)$ (see \eg example 1 in \cite{RauchTaylor})  that satisfies 
 $$
 1 \leq \| \tilde{u}_n \|_{\Hone(\Omega)} \leq C  \| u_n \|_{\Hone(\Omega^\delta)} \quad \mbox{and} \quad \quad  \tilde{u}_n = u_n \; \mbox{on} \; \Omega^{\delta_n}.
 $$  
 Then, for any $v \in \Hone(\Omega)$,  
 $$
  \lim_{n\rightarrow +\infty} a^{\delta_n}( \tilde{u}_n , v )  =  \lim_{n\rightarrow +\infty} \left( a^{\delta_n}( {u}_n , v )  +    \int_{\Omega_{\text{hole}}^{\delta_{n}}} \nabla \tilde{u}_{n} \cdot \overline{\nabla v} \, dx \; - \;  (k^\delta)^2   \tilde{u}_{n} \overline{v} \, dx  \right)= 0.
  $$ 
 Indeed, since the measure of  $\Omega_{\text{hole}}^{\delta_{n}}$ tends to $0$ as $\delta$ tends to $0$, for any $v \in \Hone(\Omega)$, $\dsp \lim_{n\rightarrow +\infty} \| v \|_{\Hone(\Omega_{\text{hole}}^{\delta_{n}})} =0$.
 
 \noindent Besides,  $\tilde{u}_n$ being bounded in $\Hone(\Omega)$, there exists a function $u_\ast \in \Hone(\Omega)$  such that, up to a subsequence, $\tilde{u}_n$ weakly tends to $u_\ast$  in $\Hone(\Omega)$ as $n$ tends to $+\infty$. As a result,
  $$
  \lim_{n\rightarrow +\infty} a^{\delta_n}( \tilde{u}_n , v )  =  \int_{\Omega} \nabla u_\ast \cdot \overline{\nabla v} \, dx  \; - \;  (k_0)^2 \int_{\Omega} u_\ast  \overline{ v} dx -i k_0 \int_{\Omega}   u_\ast  v ds = 0.
  $$ 
  Naturally, it implies that $u_\ast=0$.  In particular, it
  $\lim_{n\rightarrow\infty} \| \tilde{u}_n \|_{L^2(\Omega)} =0$, which in turn implies that 
  $
  \lim_{n\rightarrow\infty} \| \nabla \tilde{u}_n \|_{L^2(\Omega)} = 0,
  $
  and contradicts the fact that $ \| \tilde{u}_n \|_{\Hone(\Omega)} \geq 1$.

\section{Technical results for the near field singularities}
\label{appendix_singularities}

\subsection{The variational framework associated with the near field problems}
\label{sec:study-non-vari-2}

The near field terms $U_{n,q,\pm}^\delta$ satisfy 
  Laplace problems (see \eqref{NearFieldEquation})of the form
\begin{equation}\label{eq:nearfield_global}
  \left\lbrace\quad
    \begin{aligned}
      -\Delta U & = F \quad \mbox{in} \; \widehat{\Omega}^\pm,\\
      \partial_\bn U & = G \quad \mbox{on} \;  \partial \widehat{\Omega}^\pm.
    \end{aligned}
  \right. 
\end{equation}

As described in Section~{3.5} in \cite{CalozVial}, the standard
variational space to solve problem~\eqref{eq:nearfield_global} is
\begin{equation}
  \label{eq:classical_norm_NF}
  \mathfrak{V}(\widehat{\Omega}^\pm) = \left\{ v \in \Hone_{\text{loc}}(\widehat{\Omega}^\pm), \;
    \nabla v \in L^2(\widehat{\Omega}^\pm), \quad \frac{v}{{(1+R^\pm) \ln(2+R^\pm)}}
    \in  L^2(\widehat{\Omega}^\pm) \right\},
\end{equation}
which, equipped with the norm
\begin{equation*}
  \| v\|_{\mathfrak{V} (\widehat{\Omega}^\pm) } = \left( \left\|  {v}/{{(1+R^\pm) \ln(2+R^\pm)}} \right\|_{
      \Ltwo(\widehat{\Omega}^\pm)}^2 + \|  \nabla v\|_{
      \Ltwo(\widehat{\Omega}^\pm)}^2 \right)^{1/2}
\end{equation*}
is a Hilbert space.  Based on a variational formulation, we can prove
the following well-posedness result (see Proposition~3.22 and
Corollary~3.23 of Ref.~\cite{CalozVial} for the proof):
\begin{proposition}
  \label{PropositionProblemeChampProche}
  Assume that
  ${(1+R^\pm) \ln(2+R^\pm)} F \in \Ltwo(\widehat{\Omega}^\pm)$,
  ${(1+R^\pm)^{1/2} \ln(2+R^\pm)} G \in
  \Ltwo(\partial \widehat{\Omega}^\pm)$, {and the compatibility condition
    \begin{equation}
      \label{eq:CompatibilityProblemeChampProche}
      \int_{\widehat{\Omega}^\pm} F + \int_{\partial \widehat{\Omega}^\pm}
      G = 0
    \end{equation}
    is satisfied.}
  Then, problem~\eqref{eq:nearfield_global} has a solution
  $u \in \mathfrak{V}(\widehat{\Omega}^\pm)$, unique up to an additive constant.
\end{proposition}

%% Non-classical

\subsection{Absence of logarithmic singularity}
\label{sec:absence-logar-sing}

As explained in Section~\ref{sec:constr-singularities}, we are interesting in building solutions to the homogeneous problem (\ie $F=G=0$) associated with~\eqref{eq:nearfield_global}  that blow up at infinity.  One natural question is to know if such a solution can blow up like $\ln R^\pm$ at infinity. The negative answer is given in the following Lemma:
\begin{lemma}
  \label{lema:no_singularity_alone}
The problem
  \begin{equation}\label{eq:nearfield_global_star}
    \left\lbrace\quad
      \begin{aligned}
        -\Delta U & = \Delta \mathcal{U}_{0,1,\pm} \quad \mbox{in} \; \widehat{\Omega}^\pm,\\
        \partial_\bn U & = - \partial_\bn \mathcal{U}_{0,1,\pm} \quad
        \mbox{on} \; \partial \widehat{\Omega}^\pm
      \end{aligned}
    \right. 
  \end{equation}
  has no solution in $\mathfrak{V}(\widehat{\Omega}^+)$. As a
  consequence, it is not possible to construct a solution to the homogeneous problem~\eqref{eq:nearfield_global} (\ie $F=G=0$)  that has a logarithmic blow up as $R^\pm$ tends toward infinity.
\end{lemma}
\begin{proof}
We first remark that $(1+R^\pm) \ln(2+R^\pm) \Delta \mathcal{U}_{0,1,\pm} \in
  \Ltwo(\widehat{\Omega}^\pm)$ and $(1+R^\pm)^{1/2} \ln(2+R^\pm) \partial_\bn  \mathcal{U}_{0,1,\pm} \in
  \Ltwo(\partial \widehat{\Omega}^\pm)$. Then, thanks to Proposition~\ref{PropositionProblemeChampProche}, if Problem~\eqref{eq:nearfield_global_star} has a solution in $\mathfrak{V}(\widehat{\Omega}^+)$, the right hand side of~\eqref{eq:nearfield_global_star} has to satisfy the compatibility condition~\eqref{eq:CompatibilityProblemeChampProche}. We shall see that this compatibility condition does not hold.\\
   
  \noindent Following the proof of Theorem~3.25 in
  Ref.~\cite{CalozVial}, we shall construct a sequence of domains
  $\widehat{\Omega}_k^+$ that tends to $\widehat{\Omega}^+$ as $k$
  tends to $+\infty$. To do so, we introduce $M_0 \in (0,1)$ such that
  the vertical segment $\{ -M_0 \} \times (-1,1)$ does not intersect
  the obstacle $\widehat{\Omega}_\hole$, and we consider the sequence
  $(M_k)_{k \in \IN^\ast}$ defined by $M_k = M_0 + k$. By
  construction, the vertical segment $\{ -M_k \} \times (-1,1)$ does
  not intersect any hole of the domain $\widehat{\Omega}^+$. Then, we
  define,
  \begin{equation} 
  \widehat{\Omega}_k^+ = \widehat{\Omega}^+ \cap \mathcal{B}(0,M_k) \quad \mbox{and} \quad \Gamma_k^+ = (\partial \widehat{\Omega}^+) \cap \mathcal{B}(0,M_k).
  \end{equation} 
  %and
  %$\Gamma_M^+ = (\partial \widehat{\Omega}^+) \cap \mathcal{B}(0,M)$.
 We have
   \begin{equation}
    \int_{\widehat{\Omega}^+} \Delta \mathcal{U}_{0,1,+} -
    \int_{\partial \widehat{\Omega}^+} \partial_\bn
    \mathcal{U}_{0,1,+} = \lim_{k \to \infty} \int_{\widehat{\Omega}^+_k} \Delta \mathcal{U}_{0,1,+} -
    \int_{\Gamma_k^+} \partial_\bn
    \mathcal{U}_{0,1,+}
  \end{equation}
  Applying the Green formula (to the first integral of the right hand side of the previous equality) gives
   % \label{eq:decomposition_partial_Omega_M_+}
   % \partial \widehat{\Omega}^+_M = \Gamma_M^+ \cup \mathcal{I}_M^+,
    %\qquad \mathcal{I}_M^+ = \big\lbrace 
    %\mX^+ \in \widehat{\Omega}^+, \lvert \mX^+ \rvert = M \big\rbrace,
  %\end{equation}
  %so that we can rewrite \eqref{eq:limit_integrals} as
  \begin{equation}
    \label{eq:limit_integrals_2}
    \int_{\widehat{\Omega}^+} \Delta \mathcal{U}_{0,1,+} -
    \int_{\partial \widehat{\Omega}^+} \partial_\bn
    \mathcal{U}_{0,1,+} = \lim_{k \to \infty}
    \int_{0}^{\Theta} \partial_R^\pm \mathcal{U}_{0,1,+}(M_k, \theta^+) M_k d \theta^+.
   %\quad \mathcal{I}_M^+ = \big\lbrace 
    %\mX^+ \in \widehat{\Omega}^+, \lvert \mX^+ \rvert =M \big\rbrace
  \end{equation}
  %\bdsn{ \\Adrien, je ne comprends pas pourquoi tu as besoin de ces arguments. IL me semble que la $I$+ ne croise pas les %trous ? pourquoi je me trompe ? \\} 
%  \bdsn{--------------------------------------------------------------------------------------------------------------------------------\\}
 % We can parametrize the integral on $I_M^+$ by the angular variable
 % $\theta^+$, and we call $I_M^+$ its domain of validity. We can
 % easily remark that the length $\lvert I_M^+ \rvert$ of $I_M^+$ satisfies
 % \begin{equation*}
 %   \Theta - 2 \sin^{-1} \frac{2}{M} \leq \lvert I_M^+ \rvert \leq \Theta
  %\end{equation*}
 % and tends to $\Theta$, as $M$ tends to infinity. 
 % \bdsn{\\-------------------------------------------------------------------------------------------------------------------------------\\}
  %Then
  %\begin{equation}
   % \label{eq:limit_integrals_3}
  %  \int_{\mathcal{I}_M^+} \partial_\bn
   % \mathcal{U}_{\star,1,+}
  %  = \int_{I_M^+} \partial_{R^+} \mathcal{U}_{\star,1,+}(M,\theta^+)\,
  %  M d\theta^+ 
 %\end{equation}
  But,  for large $R^\pm$, 
  \begin{equation}
    \nonumber
   \mathcal{U}_{0,1,+}(M,\theta^+) = \ln R^\pm +
   \chi_{\text{macro},+}(X_1^+, X_2^+)  \frac{1}{R^\pm} w_{0,1,+}(\ln
   R^\pm) + \chi_-(X_1^+) \lvert X_1^+
  \rvert^{-1} p_{0,1,+}(\ln \lvert X_1^+ \rvert, X_1^+,
  X_2^+) 
  \end{equation}
  where
  \begin{equation}
    \nonumber
  p_{0,1,+} (\ln \lvert X_1^+ \rvert, X_1^+,
  X_2^+) =g_{0,0,1,+}^{\mathfrak{n}}(\ln \lvert X_1^+ \rvert) \;
  W_1^{\mathfrak{n}}\left(X_1^+, X_2^+
  \right) + \sum_{p=0}^1 \, \;
g_{0,1-p,p,+}^{\mathfrak{t}}(\ln \lvert X_1^+ \rvert)\;
  W_p^{\mathfrak{t}} \left(X_1^+,  X_2^+ \right),
  \end{equation}
 $g_{0,1-p,p,+}^{\mathfrak{t}}$ and $g_{0,0,1,+}^{\mathfrak{n}}$ having a polynomial dependence with respect to $\ln |X_1^+|$. Then, a direct computation shows that 
  \begin{equation*}
    \partial_{R^+} \mathcal{U}_{0,1,+}(R^+,\theta^+) = \frac{1}{M} + O\left(
      \frac{\ln R^+}{(R^+)^2} \right), \quad \text{uniformly w.r.t }\theta^+.
  \end{equation*}
Consequently, taking the limit of the  integral in \eqref{eq:limit_integrals_2} gives
  \begin{equation}
    \label{eq:limit_integrals_4}
    \int_{\widehat{\Omega}^+} \Delta \mathcal{U}_{0,1,+} -
    \int_{\partial \widehat{\Omega}^+} \partial_\bn
    \mathcal{U}_{0,1,+} = \Theta \neq 0,
  \end{equation}
  which means that compatibility condition
  \eqref{eq:CompatibilityProblemeChampProche} is not satisfied. \\
  
\noindent  Finally, the absence of logarithmic singularity is proved
by contradiction: assume that such a function exists. We denote it by
$S_{\text{log}}$. Then, in view of Theorem~4.1 of Ref.~\cite{Nazarov205},  $S_{\text{log}}$ can be decomposed as $S_{\text{log}} = \mathcal{U}_{0,1,+}+ \hat{S}_{\text{log}}$,  $\hat{S}_{\text{log}}$ being in  $\mathfrak{V}(\widehat{\Omega}^+)$. Noticing that $\hat{S}_{\text{log}}$ satisfies Problem~\eqref{eq:nearfield_global_star} that has no solution in $\mathfrak{V}(\widehat{\Omega}^+)$, we obtain a contradiction.\end{proof} 

 \begin{lemma}
    \label{lema:compatibility_condition_Unp}
    Let $n \in \IZ^*$ and $p(n)=\max(1,1 + \lceil \lambda_n \rceil)$. If $n < 0$ or $\lambda_n \not\in \IN$, then
    \begin{equation}
      \label{eq:compatibility_condition_Unp}
      \int_{\widehat{\Omega}^+} \Delta \mathcal{U}_{n,p(n),\pm} -
      \int_{\partial \widehat{\Omega}^+} \partial_\bn
      \mathcal{U}_{n,p(n),\pm} = 0
    \end{equation}
  \end{lemma}
  \begin{proof}
   As in  the proof of
    Lemma~\ref{lema:no_singularity_alone}, we will define a domain
    $\widehat{\Omega}^{\pm}_k$ such that
    $\lim_{k \to \infty} \widehat{\Omega}^{\pm}_k =
    \widehat{\Omega}^{\pm}$. As previously, we consider $M_0 \in (0,1)$ such that the
    vertical segment $\{ -M_0 \} \times (-1,1)$ does not intersect the
    obstacle $\widehat{\Omega}_\hole$, and we consider the sequence
    $M_k$, $k \in \IN^\ast$, by $M_k = M_0 + k$. By construction, the
    vertical segment $\{ -M_k \} \times (-1,1)$ does not intersect 
    any hole of the domain $\widehat{\Omega}^+$. We define the
    boundary $I_k$, 
 %   \begin{equation}
%{I}_{M_k} =  I_k^c  \cup I_k^p \quad I_k^c  =  \{ (- M_k ,X_2^+) \in \R^2, -2 < X_2 < 2 \} \quad  I_k^p = \{ (R_k^+ \cos(\theta^+), R_k^+ \sin(\theta^+)) \in \R^2,  R_k^+ = \sqrt{M_k^2+4},  \theta^+ 
%\end{equation}

\begin{equation*}
I_k = \lbrace (R^+_{k}(\theta^+) \cos \theta^+,R^+_{k}(\theta^+)  \sin \theta^+) \in
      \widehat{\Omega}^+ , 0 < \theta^+ < \Theta \rbrace, 
    \end{equation*}
    where the function $R^+_{k}(\theta^+)$ is given by
    \begin{equation*}
      R^+_{k}(\theta^+) =
      \begin{cases}
       - M_k / \cos \theta^+, & \quad  \lvert \theta^+ - \pi \rvert
        \leq \theta_k, \\
         \sqrt{M_k^2+4}, & \quad \text{otherwise.}
      \end{cases}  \quad \theta_k = \sin^{-1} ( 2 / \sqrt{M_k^2+4} ).
    \end{equation*}
     For $\lvert \theta^+ - \pi \rvert
        \geq \theta_k$, ${I}_{k} $ coincides with a portion of the circle of radius $ \sqrt{M_k^2+4}$ and of center $(0,0)$ while for $\lvert \theta^+ - \pi \rvert \geq \theta)k$, ${I}_{k} $ coincides with the segment $\{ (-M_k, X_2^+) -2 \leq X_2^+ \leq 2\}$. 
   % Then,  let $\widehat{\Omega}^{\pm}_{M_k}$ be the domain defined by
    %\begin{equation*}
     % \widehat{\Omega}^{\pm}_{M_k} = \lbrace (R^+,\theta^+) \in
     % \widehat{\Omega}^+ , 0 < \theta^+ < \Theta , 0 < R^+ <
     % R^+_{M_k}(\theta^+) \rbrace.
    %\end{equation*}
        Again, analogously to the proof of Lemma~\ref{lema:no_singularity_alone}, we have,
    \begin{equation}
      \label{eq:limit_integrals_Unp}
      \int_{\widehat{\Omega}^+} \Delta \mathcal{U}_{n,p(n),+} -
      \int_{\partial \widehat{\Omega}^+} \partial_\bn
      \mathcal{U}_{n,p(n),+} = \lim_{k \to \infty} J_k^n \quad  J_k^n = \int_{{I}_{k}} \partial_\bn
      \mathcal{U}_{n,p(n),+} d\sigma
      \end{equation}
    where, since
    $(1+R^+) \ln(2+R^+) \Delta \mathcal{U}_{n,p(n),+} \in
    \Ltwo(\widehat{\Omega}^+)$ and $(1+R^+)^{1/2} \ln(2+R^+) \partial_\bn \mathcal{U}_{n,p(n),+}  \in
  \Ltwo(\partial \widehat{\Omega}^+)$
    the limit of $J_k^n$ is finite.  But,  applying Lemma~\ref{lema:calculInfernal1} and  Lemma~\ref{lema:calculInfernal2} below, we can prove
   \begin{equation}
{J}_k^n = \sum_{m=0}^{\lfloor \lambda_n \rfloor} \sum_{\ell=0}^{L}
      C_{m\ell} M_k^{\lambda_n - m} (\ln M_k)^\ell + o(1).
  \end{equation}
If $n<0$, we immediately deduce that ${J}_k^n$ tends to $0$ as $k$ tends toward infinity.  For $n>0$, since $\lambda_n \notin \N$, $\lambda_n -m \neq 0$. But, since the limit is finite, the coefficients $C_{m\ell}$ have to vanish and we conclude that $\dsp \lim_{k \rightarrow \infty}{J}_k^n =0$. 
   \end{proof}
   \begin{lemma}\label{lema:calculInfernal1}
     For any $(n,p)\in \Z \times \R$, there exists a sequence 
     $(C_{n,p,t,q})_{t\in \N, q\in \N, q\leq p}$ such that, for any $s \in \N$, 
     \begin{multline}
       \label{eq:integration_Ik_wnp}
       \int_{{I}_{k}^+} \partial_\bn (R^+)^{\lambda_n-p}
       w_{n,p,+}(\ln R^+,\theta^+) \chi_{\text{macro},+}(X_1^+,X_2^+) d\sigma(\mX)
       \\ = \sum_{t=0}^s \sum_{q=0}^p
       C_{n,p,t,q} (M_k)^{\lambda_n-p-t}
       (\ln M_k)^q +  o\big( (M_k)^{\lambda_n-p-s} \big).
     \end{multline}
   \end{lemma}
   \begin{proof}
     We decompose $I_k^+$ into its circular part 
     \begin{equation}\label{definitionI1}
     {I}_{1}= \left\{ \left(R_k \cos \theta^+, R_k\sin \theta^+ \right) \in \R^2,  \theta^+ \in (0, \pi -\theta_k) \cup  ( \pi +\theta_k, \Theta) \right\},
     \end{equation}
     $R_k  = \sqrt{M_k^2 + 4}$, and its straight part 
     \begin{equation}\label{definitionI2}
     {I}_{2}  = \left\{ (- M_k, X_2^+) \in \R^2, X_2^+ \in (-2, 2) \right\},
     \end{equation}
     and we study the integral over these two parts separately.

     \textbf{Integration over ${I}_{1}$:} On this part, the normal derivative
     is $\partial_\bn = \partial_{R^+}$ and 
     $\chi_{\macro,+}=1$. Using the explicit form~\eqref{eq:definition_w_star_p} of the 
     function $w_{n,p,+}(\ln R^+, \theta^+)$, we see that
     \begin{multline}
       \nonumber
       J_1 = \int_{{I}_{k,c}} \partial_\bn \left\{  (R^+)^{\lambda_n-p}
         w_{n,p,+}(\ln R^+,\theta) \chi_{\text{macro},+} \right\}
       d\sigma \\ = (R_k^+)^{\lambda_n-p}   \sum_{q=0}^p    ( \ln {R_k^+})^q   \int_{|\theta^+ - \pi| \geq \theta_k}
       v_{n,p,q,+}(\theta^+) d\theta^+.          
     \end{multline}
     where  the functions $v_{n,p,q,+}$ are smooth on the intervals $(0, \pi -\theta_k) $   and $( \pi -\theta_k, \Theta)$.  On $(0, \pi -\theta_k)$ (resp.  $( \pi -\theta_k, \Theta)$), we denote by $V_{n,p,q,+}$ the primitive of  $v_{n,p,q,+}$ that vanishes at $0$ (resp. $\Theta$).  The function $V_{n,p,q,+}$ is smooth on both $(0, \pi -\theta_k) $ and $( \pi -\theta_k, \Theta)$.
     \begin{equation}
       J_1  =    (R_k^+)^{\lambda_n-p}   \sum_{q=0}^p    ( \ln {R_k^+})^q   (V_{n,p,q,+}(\pi -\theta_k) - V_{n,p,q,+}(\pi +\theta_k))
     \end{equation}
     Then, we use Taylor expansion of $V_{n,p,q,+}$ at the point $\theta = \pi^\pm$ ($V_{n,p,q,+}$ is not continuous at $\pi$)  
     $$
     V_{n,p,q,+}(\pi -\theta_k) = \sum_{r=0}^N  \frac{V_{n,p,q,+}^{(r)}(\pi^+)}{r!}  \left(\theta_k\right)^r +o((\theta_k)^N)
     $$ 
     and  the following expansions to conclude:
     \begin{align*}
       &  \forall s \in \R, \; \exists \, ( \alpha_{i,s})_{i\in \N}, \; \forall N\in \N, \quad     R_k^s =   \sum_{i=0}^N   M_k^{s-i}  \alpha_{i,s}  + o((M_k)^{s-N}),  \\
       &  \forall m \in \R, \; \exists \, ( \beta_{i,m,\ell})_{i\in
         \N, \ell\in\N, \ell\leq m}, \; \forall N \in \N,  \\ &\quad  (\ln R_k)^m =\sum_{i=0}^N \sum_{\ell=0}^m (\ln M_k)^\ell \beta_{i,m,\ell} M_k^{-i}  + o((M_k)^{-N}), \\
       & \forall m \in \R,  \; \exists \, ( \gamma_{i,m})_{i\in \N},  \; \forall N \in \N, \quad   (\theta_k)^m = \sum_{i=0}^N \gamma_{i,m} M_k^{-i} + o((M_k)^{-N}).
     \end{align*}
     
     \textbf{Integration over $I_2$:}
     On this part,  $\partial_\bn = - \partial_{X_1}^+$, and $\chi_\macro(\bX^+)  = \chi(X_2)$. Then,
     \begin{equation}
       \nonumber
       J_2 = \int_{(-2,-1)\cup(1,2)} - \partial_{X_1^+}
       \mathtt{v}_{n,p}(\bX)      \chi(X_2^+) dX_2^+,  \quad
       \mbox{with } \mathtt{v}_{n,p}(\bX)  =   (R^+)^{\lambda_n-p} \sum_{q=0}^p \big( \ln R^+ \big)^q
       w_{n,p,q,+}(\theta^+) 
     \end{equation}
     We remind that $
     \mathtt{v}_{n,p}$ is harmonic in both
     $\OmegaTop$ and $\OmegaBottom$.
     We shall compute the integral over $(1,2)$,
     the computation of the integral over $(-2,-1)$
     being similar. First, since $\partial_{X_1^+}
     \mathtt{v}_{n,p}(\bX^+) = \cos \theta^+ \partial_{R^+}
     \mathtt{v}_{n,p}(R^+,\theta^+) - \frac{1}{R^+}
     \sin(\theta^+) \partial_{\theta^+} \mathtt{v}_{n,p}(R^+,
     \theta^+)$, there exists smooth functions $v_{n,p,q,+}$ such that
     \begin{equation}\label{dX1vnp}
       \partial_{X_1^+} \mathtt{v}_{n,p}(\bX^+) = (R^+)^{\lambda_n-p-1} \sum_{q=0}^p \big( \ln R^+ \big)^q  v_{n,p,q,+}(\theta^+). 
     \end{equation}
     Since $R^+ = M_k \sqrt{1 + \frac{X_2^2}{M_k^2}}$ and $\theta^+ = \tan^{-1} \left( \frac{X_2^+}{X_1^+} \right)$ to obtain the following asymptotic formula (reminding that $X_2$ is bounded):
     \begin{align*}
       & \forall  s \in \Z, \exists (\tilde{\alpha}_{i,s})_{i\in \N}, \forall N \in \N, (R^+)^s = \sum_{i=0}^N \tilde{\alpha}_{i,s} (X_2^+)^{2i}  \, (M_k)^{s-2i}+ o((M_k)^{s-2N}) \\
       &   \forall m \in \R, \; \exists \, (
         \tilde{\beta}_{i,m,\ell})_{i\in \N, \ell\in\N, \ell\leq m},
         \; \forall N \in \N, \\ & \quad  (\ln R^+)^m =\sum_{i=0}^N \sum_{\ell=0}^m (\ln M_k)^\ell \tilde{\beta}_{i,m,\ell} {(X_2^+)}^{2i} M_k^{-2i}  + o((M_k)^{-2N}),\\
       & \exists ({\gamma}_i)_{i \in \N}, \forall N \in \N, v_{n,p,q,+}(\theta^+) = \sum_{i=0}^N \tilde{\gamma_i} \, {(X_2)}^i \,  (M_k)^{-i}  + o((M_k)^{-N})
     \end{align*}
     Introducing the previous formulas into~\eqref{dX1vnp}, integrating exactly with respect to $X_2$ gives the desired formula.  
   \end{proof}
   \begin{lemma}\label{lema:calculInfernal2}
     For any $(n,q)\in \Z \times \R$ and for any $s \in \N$, there exists a sequence 
     $(C'_{n,q,t,r})_{t \leqslant s, r\leq q}$ such that
     \begin{multline}
       \label{eq:integration_Ik_pnp}
       \int_{{I}_{k}^+} \partial_\bn (X_1^+)^{\lambda_n-q}
       p_{n,q,+}(\ln \lvert X_1^+ \rvert,X_1^+,X_2^+) \chi_-(X_1^+) d\sigma(\mX)
        \\= \sum_{t=0}^s \sum_{r=0}^q
       C'_{n,q,t,r} (M_k)^{\lambda_n-q-t}
       (\ln M_k)^r +  o\big( (M_k)^{\lambda_n-q-s} \big).
     \end{multline}
   \end{lemma}
   \begin{proof}
     As in the proof of
     Lemma~\ref{lema:calculInfernal1},  we decompose $I_k^+$ into its
     circular part $I_1$and its straight part $I_2$ (cf.~\eqref{definitionI1}-\eqref{definitionI2}), and we study the
     integral over these two parts separately.

     \textbf{Integration over $I_2$:}
     on this part, $X_1^+=-M_k$ $\partial_\bn = - \partial_{X_1}^+$, and
     $\chi_-(X_1^+) = \chi(M_k) = 1$ for $k \geqslant 2$. Then,
     \begin{equation*}
       J_2 = \int_{-2}^2 - \partial_{X_1^+} \left( (X_1^+)^{\lambda_n-q}
       p_{n,q,+}(\ln |X_1^+|,\bX^+) \right)dX_2^+. 
     \end{equation*}
     We use expression of $p_{n,q,+}$ given by \eqref{definitionPmr}
     in Appendix~\ref{sec:defin-prof-funct_p}, which yields to consider intergrals of the form
     $$
     \int_{-2}^2 (X_1^+)^{\lambda_n-q-1} \left( \ln X_1^+ \right)^\kappa W(X_1^+, X_2^+) dX_2^+ \quad \kappa \in \N,
     $$ 
     where the functions $W$  are one periodic with respect to $X_1^+$ ($W(X_1^+, X_2^+) = W(-M_0, X_2^+)$). Moreover, since, by assumption the line $X_1^+ = -M_0$ does not intersect any obstacle, the functions $W(-M_0, X_2^+)$ are continuous and bounded for $X_2^+\in [-2, 2]$. Then, integrating exactly with respect to $X_2$ gives the desired formula.
     
     \textbf{Integration over $I_1$:}
     On this part, the normal derivative
     is given by
     \begin{equation}\label{Partialn}
     \partial_\bn = \partial_{R^+} =
     \frac{X_1^+}{R_k} \partial_{X_1^+} +
     \frac{X_2^+}{R_k} \partial_{X_2^+}, \quad R_k = \sqrt{M_k^2 + 4}.
     \end{equation}
      Here again, we separate this integral
     into three arcs:
     \begin{equation}
       \nonumber
       I_1^1 = \left\{ \left(R_k\cos \theta^+,
           R_k \sin \theta^+ \right) \in \R^2,  \theta^+
         \in (0, \pi/2 + \theta'_k) \cup ( \frac{3\pi}{2} - \theta'_k,
         \Theta) \right\}, \quad \theta'_k = \sin^{-1}(1/R_k)
     \end{equation}
      \begin{multline}
        \nonumber
       I_1^2 = \left\{ \left(R_k\cos \theta^+,
           R_k \sin \theta^+ \right) \in \R^2,  \theta^+
         \in (\pi/2 + \theta'_k, \pi - \theta''_k) \cup ( \pi+\theta''_k, \frac{3\pi}{2} - \theta'_k) \right\}, \\ \theta''_k = \sin^{-1} (\frac{\alpha \ln M_k} { R_k})
     \end{multline}
        \begin{equation*}
       I_1^3 = \left\{ \left({R_k} \cos \theta^+,
           {R_k} \sin \theta^+ \right) \in \R^2,  \theta^+
         \in (\pi - \theta''_k, \pi - \theta_k) \cup ( \pi+\theta_k,
         \pi+\theta''_k) \right\}.
     \end{equation*}
The parameter $\alpha \in \R_+^\ast$ (defining $\theta''_k$) will be fixed later. We remark that $R_k \cos(\pi/2+ \theta_k') = -1$ and $R_k \sin (\pi \mp \theta''_k) = \pm \alpha  \ln M_k$.\\
   
   \noindent   \textit{Integration over $I_1^1$:} this integration is trivial,
     because on this integration domain, $X_1^+ \geqslant R_k \cos(\pi/2+ \theta_k') \geqslant -1$ and
     therefore $\chi_-(X_1^+)=0$.\\

     \noindent \textit{Integration over $I_1^2$:} because the functions $p_{n,q,+}$ are exponentially decaying with respect to $X_2^+$,  we shall prove that the corresponding 
     integral is $o(M_k)^{\lambda_n-p-s}$. First, we remind that 
     the profile functions
     $W_i^{\mathfrak{t}}$ and  $W_i^{\mathfrak{n}}$ are in $\mathcal{V}^+(\mathcal{B})$ and are smooth on $I_1^2$. It follows that  $W_i^{\mathfrak{t}}$, $W_i^{\mathfrak{n}}$ and their derivatives can be bounded by
     $C_i \exp(- \pi X_2^+)$. Since  $\lvert X_1^+ \rvert \leq R_k$, it follows that there exists $C>0$ such that
     \begin{multline}
       \label{eq:integration_pnp_I12}
      J_1^2 =  \int_{I_1^2} \partial_\bn \left\{ (X_1^+)^{\lambda_n-p}
       p_{n,p,+}(\ln \lvert X_1^+ \rvert,X_1^+,X_2^+) \chi_-(X_1^+) \right\}
       d\sigma(\mX) \\ \leqslant C \, \big( R_k  \big)^{\lambda_n-p} \lvert
       \ln R_k \rvert^p \int_{I_1^2} \exp( - \pi X_2^+) d\sigma(\bX)
     \end{multline}
     We parametrize then the arc $I_1^2$ by
     $X_2 \in \pm \left(\alpha \ln M_k, R_k \cos( \tilde{\theta}_k)\right)$, which means that $X_1^+ = - \sqrt{R_k^2 -X_2^2}$ and
     $d\sigma(\bX) = \lvert R_k/X_1^+\rvert dX_2$. In addition, since 
     $\lvert X_1^+ \rvert \geqslant 1$ on $I_1^2$ (by construction),
     $\lvert R_k^+/X_1^+ \rvert \leqslant
     R_k$.
     Therefore, 
      \begin{equation}
        \nonumber
        J_1^2 \leq    \big(R_k \big)^{\lambda_n-p+1} \lvert
        C   \ln R_k \rvert^p \int_{\alpha \ln M_k}^\infty
        \exp(-\pi X_2^+) dX_2^+ = \frac{C}{\pi} \big(R_k \big)^{\lambda_n-p+1} \lvert
        \ln R_k \rvert^p M_k^{-\pi \alpha},
     \end{equation}
     which is equivalent to $M_k^{\lambda_n-p+1-\pi\alpha} \lvert \ln M_k \rvert^p$ as $k$ tends toward infinity.
     In the end, choosing $\alpha = (s+2)/\pi$, we see that $J_1^2 = o\big( (M_k)^{\lambda_n-p-s} \big)$.\\

     \textit{Integration over $I_1^3$:} on this integration domain,
     $\chi_-(X_1^+)=1$, and
     $\lvert X_2^+\rvert \in (2,  \alpha \ln M_k)$. It follows that
     $X_2 / R_k$ is uniformly bounded by $\alpha \ln
     M_k / M_k$, which tends to $0$ as $M_k$ tends to
     infinity.  Combining formula~\eqref{Partialn} and  the definition~\eqref{definitionPmr} of the function $p_{n,q,+}$, we see that we have to evaluate the two following kinds of integrals: 
     \begin{equation}\label{J13K13}
     J_1^3 = \int_{I_1^3} \frac{ (X_1^+)^{\lambda_n-q}}{R_k}
     \ln(|X_1^+|)^\kappa W(X_1^+, X_2^+) d\sigma, \quad K_1^3 = \int_{I_1^3} \frac{X_2^+ (X_1^+)^{\lambda_n-q}}{R_k} \ln(|X_1^+|)^\kappa W(X_1^+, X_2^+) d\sigma 
     \end{equation}
     where $\kappa \in \N$ and $W \in\mathcal{V}^+(\mathcal{B})$ (exponentially decaying with respect to $X_2^+$) is a (generic) $1$-periodic function in $X_1^+$. In fact, $W$ stands for either the profile functions $W_{i}^{\mathfrak{t}}$ and $W_{i}^{\mathfrak{n}}$ (defined in~\eqref{ProblemWpt}-\eqref{ProblemWpn}) or their partial derivatives with respect to $X_1^+$ and $X_2^+$.  Consequently, $W$  
     admits the following Fourier series decomposition for  $|X_2^+| > 2$:
    \begin{equation}
       \label{eq:Fourier_representation_W_i_a}
     \exists R \in \N,   \exists\, (c_{r,p,\pm})_{r \leqslant R, p \in
         \IZ^\ast}, \quad  W(\mX) = \sum_{r=0}^R \sum_{p
       \not=0} c_{r,p,\pm} \exp(\imath 2 \pi p X_1^+),
     (X_2^+)^r \exp(-2\pi p |X_2^+|),  
   \end{equation}
     the coefficients
     $c_{r,p,\pm}$ being super-algebraically convergent as  $p \to \pm \infty$, \ie
          \begin{equation}
       \label{eq:rapid_convergence}
  \forall r \in \N,  r \leq R,   \forall \beta \in \R, \quad   \sum_{p \not=0} p^\beta c_{r,p,\pm} \exp(-4\pi
       p) < \infty.
     \end{equation}
         Since $X_1^+ = - \sqrt{R_k^2-(X_2^+)^2} = -M_k \sqrt{1+
       \frac{4-X_2^2}{M_k^2}}$, similarly to the proof of
     Lemma~\ref{lema:calculInfernal1}, the following expansions hold:
     \begin{equation}\label{Taylor1}
        \forall  s \in \Z, \exists (\tilde{\alpha}_{i,j,s})_{(i,j)\in
          \N^2},
         \forall N \in \N,\quad (X_1^+)^s = \sum_{i=0}^N \sum_{j=0}^i
         \tilde{\alpha}_{i,j,s} (X_2^+)^{2j}  \, (M_k)^{s-2i}+
         M_k^s o((\ln M_k / M_k)^{2N}) 
         \end{equation}
         \begin{multline}\label{Taylor2}
        \forall m \in \R, \; \exists \, (
        \tilde{\beta}_{i,j,m,\ell})_{(i,j,\ell)\in \N^3, \ell\leq m},
        \; \forall N \in \N, \\  (\ln |X_1^+|)^m =\sum_{i=0}^N
        \sum_{j=0}^i \sum_{\ell=0}^m (\ln M_k)^\ell \tilde{\beta}_{i,j,m,\ell} (X_2^+)^{2j} M_k^{-2i}  + o\left( \frac{(\ln M_k)^{m+2N}}{ M_k^{2N}}\right).
     \end{multline}
   Then, here again, we parameterize the arc $I_1^3$ by
     $X_2^+ \in \pm(2,\alpha \ln M_k)$. Expanding
     $\lvert R_k / X_1^+\rvert$ with respect to $X_2^+$, we obtain
     \begin{equation}\label{Taylor3}
       \exists (\tilde{\gamma}_{i,j})_{(i,j)\in \N^2}, \forall N \in \N,\quad
       d\sigma(\mX) = \sum_{i=0}^N \sum_{j=0}^i \tilde{\gamma}_{i,j} (X_2^+)^{2j}
       (M_k)^{-2i} dX_2^+          +  o(\left(\frac{\alpha \ln M_k}{M_k}\right)^{-2N}) dX_2^+, 
     \end{equation}
     and
     \begin{equation}\label{Taylor4}
       \exists (\tilde{\delta}_{1,i,j})_{(i,j)\in \N^2}, \forall N \in \N,\quad
       \quad X_1^+ / R_k = \sum_{i=0}^N \sum_{j=0}^i \tilde{\delta}_{1,i,j} (X_2^+)^{2j}
       (M_k)^{-2i} + o(\left(\frac{\alpha \ln M_k}{M_k}\right)^{-2N}),
     \end{equation}
     \begin{equation}\label{Taylor5}
       \exists (\tilde{\delta}_{2,i})_{i\in \N}, \forall N \in \N,
       \quad X_2^+ / R_k = \sum_{i=0}^N \tilde{\delta}_{2,i} X_2^+
       (M_k)^{-2i-1}  + o( \frac{\alpha \ln M_k}{ (M_k)^{2N+1}}).
     \end{equation}
   It remains to expand  $W(X_1^+, X_2^+)$, expressing $X_1^+$ in terms of $X_2^+$. More specifically, thanks to~\eqref{eq:Fourier_representation_W_i_a}, we have to compute $\exp(2 \imath \pi X_1^+)$  for any $p\in \Z^\ast$.  Since  $M_k = M_0 + k$, 
     \begin{equation}
       \nonumber
       \exp(2 \imath \pi p X_1^+) = \exp(-2 \imath \pi p \sqrt{R_k^2
         -X_2^2} )  = \exp(-2 \imath \pi p M_0)
       \exp\Bigg(-2 \imath \pi p M_k \bigg(
       \sqrt{1+\frac{4-(X_2^+)^2}{M_k^2}} - 1 \bigg)\Bigg).
     \end{equation}
     Then, using that
     $M_k \bigg( \sqrt{1+\frac{4-(X_2^+)^2}{M_k^2}} - 1 \bigg) =
     O(\ln^2 M_k / M_k)$
     which tends to $0$ as $M_k$ tends to $0$, for $p$ fixed, we can
     make a Taylor expansion of this exponential term with respect to
     $X_2$:
     \begin{multline}
       \label{eq:expansion_exp}
       \exists (\tilde{\zeta}_i)_{i\in \N}, \quad \forall N \in \N,
       \exp\Bigg(-2 \imath \pi p M_k \bigg(
       \sqrt{1+\frac{4-(X_2^+)^2}{M_k^2}} - 1 \bigg)\Bigg) \\ =
       \sum_{n=0}^N \frac{(-2 \imath \pi p)^n}{n!} \left[ M_k \bigg(
         \sqrt{1+\frac{4-(X_2^+)^2}{M_k^2}} - 1 \bigg) \right]^n +
       \mathcal{R}_N(p) \phi((\ln^2 M_k / M_k)^{N}) 
     \end{multline}
     where the remainder $\mathcal{R}_N(p)$ is polynomial with respect
     to $p$ and behaves like $(2\pi p)^N/(N!)$ for $N$ fixed as
     $p \to \infty$, and the function $\phi(x)$ is $o(x)$ as
     $x \to 0$. In \eqref{eq:expansion_exp}, expanding the polynomial
     sum with respect to $X_2^+$ and neglecting the terms in
     $o(M_k)^{-N}$ gives
     \begin{multline}
       \label{eq:expansion_exp2}
       \exists (\tilde{\zeta}_{i,j})_{(i,j)\in \N^2}, \quad \forall N
       \in \N, \exp\Bigg(2 \imath \pi p M_k \bigg(
       \sqrt{1+\frac{4-(X_2^+)^2}{M_k^2}} - 1 \bigg)\Bigg) \\
       1 + \sum_{n=1}^N \frac{(2 \imath \pi p)^n}{n!} M_k^{n}
       \sum_{i=1}^{\lfloor(N+n)/2\rfloor} \sum_{j=0}^i
       \tilde{\zeta}_{i,j} (X_2^+)^{2j} (M_k)^{-2i} + \frac{(2\pi
         p)^N}{N!} \tilde{\phi}((\ln^2 M_k / M_k)^N),
     \end{multline}
     where $\tilde{\phi}(x)$ is also $o(x)$ as $x \to 0$.
     Finally, we insert \eqref{eq:expansion_exp2} in
     \eqref{eq:Fourier_representation_W_i_a} and we obtain
     \begin{multline}
       \label{eq:Fourier_representation_W_i_a_2}
       W(\mX) = \sum_{r=0}^R \sum_{p
         \not=0} c_{r,p,\pm} \left( 1 + \sum_{n=1}^N
         \frac{(2 \imath \pi p)^n}{n!} M_k^{n}
         \sum_{\ell=1}^{\lfloor(N+n)/2\rfloor} \sum_{j=0}^\ell
         \tilde{\zeta}_{\ell,j} (X_2^+)^{2j} (M_k)^{-2\ell} \right)
       \\ (X_2^+)^r \exp(-2\pi p |X_2^+|) \\ + \sum_{r=0}^R \sum_{p
         \not=0} c_{r,p,\pm} (X_2^+)^r \exp(-2\pi p
       |X_2^+|) \mathcal{R}_N(p) \tilde{\phi}((\ln^2 M_k / M_k)^{N})
     \end{multline}
     To estimate the remainder in
     \eqref{eq:Fourier_representation_W_i_a_2}, we use that
     \begin{equation}
       \nonumber (X_2^+)^r \exp(-2\pi p |X_2^+|) \mathcal{R}_N(p)
       \tilde{\phi}((\ln^2 M_k / M_k)^{N}) \leqslant (\alpha \ln M_2^+)^R
       \exp(-4 \pi p) \mathcal{R}_N(p) \tilde{\phi}((\ln^2 M_k /
       M_k)^{N}),
     \end{equation}
     which, together with \eqref{eq:rapid_convergence} taking $\beta = -2-N$ gives     \begin{equation}
       \nonumber
       \sum_{r=0}^R \sum_{p
         \not=0} c_{r,p,\pm} (X_2^+)^r \exp(-2\pi p
       |X_2^+|) \mathcal{R}_N(p) \tilde{\phi}((\ln^2 M_k / M_k)^{N})
       = o \left( \frac{(\ln M_k)^{2N+R}} {(M_K)^N}  \right).
     \end{equation}
     Finally, we insert the expansions~\eqref{Taylor1}-\eqref{Taylor2}-\eqref{Taylor3}-\eqref{Taylor4}-\eqref{Taylor5}-\eqref{eq:Fourier_representation_W_i_a_2}  written up to order $N=\lambda_n-q-s$ into~\eqref{J13K13}. We obtain 
         \begin{multline}
       \label{integration_part_X3}
      J_1^3 
       =  \sum_{t=0}^s \sum_{r=0}^\kappa (M_k)^{\lambda_n-q-t} (\ln
       M_k)^r
       \sum_{i=0}^{Q} \sum_{p \not=0} c_{i,p,t,r,\pm} \int_{\pm 2}^{\pm
         \alpha \ln M_k} (X_2^+)^i \exp(-2\pi p |X_2^+|) dX_2^+
       \\ +
       o\big( (\ln M_k)^{\tilde{Q}} (M_k)^{\lambda_n-q-s} \big),
     \end{multline}
     where $Q$ and $\tilde{Q}$ are positive integers depending on $s$, $R$ and $\kappa$.
     Note also that the sum over $p$ converges using again
     \eqref{eq:rapid_convergence} with $\beta = -2-Q$.
     To conclude, it remains to estimate each integral that appears on
     \eqref{integration_part_X3}. A direct integration by parts gives,
     for any numbers $0 < a < b$,
     \begin{equation}
       \label{eq:Calcul_Infernal_2_IPP_2}
       \int_a^b (X_2^+)^i \exp(-2\pi p X_2^+) dX_2^+ = i! \sum_{k=0}^i
       \frac{(2\pi p)^{k-1-i}}{k!} \big( a^i \exp(-2 p \pi a) - b^i \exp(-2 p \pi b) \big).
     \end{equation}
     We use then \eqref{eq:Calcul_Infernal_2_IPP_2} for $a=2$ and
     $b=\alpha \ln M_k$, such that $b^i \exp(-2 \pi p b) = (\alpha \ln
     M_k)^i (M_k)^{-2 \pi p \alpha}$. Using that $\alpha = (s+2)/\pi$,
     the sum of $b^i \exp(-2 \pi b)$ over $i$ is negligible with
     respect to $(M_k)^{t-s}$ ($-t-s-4<0$). Then \eqref{eq:Calcul_Infernal_2_IPP_2} becomes
     \begin{equation}
       \label{eq:Calcul_Infernal_2_IPP}
       \int_2^{\alpha \ln M_k} (X_2^+)^i \exp(-2\pi p |X_2^+|)
     dX_2^+  = i! \exp(-4\pi p) \sum_{k=0}^i
     \frac{(2\pi p)^{k-1-i}2^i}{k!} + o((M_k)^{t-s}).
     \end{equation}
     Inserting \eqref{eq:Calcul_Infernal_2_IPP} in
     \eqref{integration_part_X3} gives the desired result for $J_1^3$, the analysis of $K_1^3$ being similar.
   \end{proof}

\section{Complete definition of the asymptotic blocks}
The definition of the asymptotic block
$\mathcal{U}_{n,p,\pm}$~\eqref{eq:ansatz_U_star_p} requires the
definition of the functions $w_{n,q,\pm}$ and $p_{n, q,\pm}$.  To do
that, we first need to introduce two families of boundary layer
functions $W_i^{\mathfrak{t}}$ and $W_i^{\mathfrak{n}}$.

\subsection{Two families of boundary layer profile functions $W_i^{\mathfrak{t}}$ and  $W_i^{\mathfrak{n}}$}

Let $W_i^{\mathfrak{t}} =0$ for any negative integer $i$, and, for $i\geq0$, we define $W_{i}^{\mathfrak{t}} \in \mathcal{V}^+(\mathcal{B})$ as the unique decaying solution to
\begin{equation}\label{ProblemWpt}
\left \lbrace
\begin{aligned}
\dsp - \Delta_{\mX} W_{i}^{\mathfrak{t}}(\mX) & \; = &&  F_{i}^{\mathfrak{t}}(\mX) + \frac{\mathcal{D}_{i}^{\mathfrak{t}}  }{2}
\dsp [g_0(\mX)] +  \frac{\mathcal{N}_{i}^{\mathfrak{t}}}{2}
\dsp [g_1(\mX)]  \quad \mbox{in} \; \mathcal{B}, \\
\partial_{\mathbf{n}} W_{i}^{\mathfrak{t}} &\;  = &&  G_i^{\mathfrak{t}}(\mX)\quad \mbox{on} \;  \partial
\widehat{\Omega}_\hole,\\[1ex]
\partial_{X_1} W_{i}^{\mathfrak{t}} (0,X_2) & \; = &&  \partial_{X_1} W_{i}^{\mathfrak{t}}(1,X_2), \quad X_2 \in
\R,
\end{aligned} 
\right.
\end{equation}
where $G_i^{\mathfrak{t}}(\mX) =  -W_{i-1}^{\mathfrak{t}} \mathbf{e}_1\cdot {\mathbf{n}}$ and
\begin{multline}\label{DefFpt}
F_{i}^{\mathfrak{t}}(\mX) = 2 \partial_{X_1} W_{i-1}^{\mathfrak{t}}(\mX) \,+\, W_{i-2}^{\mathfrak{t}}(\mX)
\,+\, (-1)^{\lfloor i/2 \rfloor} \left(
  2\, \langle g_i(\mX)\rangle \, \delta_{i}^{\text{even}} \right)  \\+ \sum_{k=2}^{i-1}
(-1)^{\lfloor k/2 \rfloor}  \frac{[g_k(\mX)]}{2}
\delta_{k}^{\text{even}}  \mathcal{D}_{i-k}^{\mathfrak{t}} + \sum_{k=2}^{i-1}
(-1)^{\lfloor k/2 \rfloor}  \,
\frac{[g_k(\mX)]}{2} \, \delta_{k}^{\text{odd}} \, \mathcal{N}_{i-k+1}^{\mathfrak{t}}.
\end{multline}
In~\eqref{DefFpt}, the constants $\mathcal{D}_i^{\mathfrak{t}}$ and $\mathcal{N}_i^{\mathfrak{t}}$ are given by

\begin{equation}\label{DefDptNpt}
\dsp \mathcal{D}_i^{\mathfrak{t}} = \int_{\mathcal{B}} F_{i}^{\mathfrak{t}}\,
\mathcal{D} + \int_{\partial
\widehat{\Omega}_\hole} G_{i}^{\mathfrak{t}}\, \mathcal{D} ,  \quad \quad \; \mathcal{N}_i^{\mathfrak{t}} = -\int_{\mathcal{B}} F_{i}^{\mathfrak{t}}\mathcal{N} - \dsp \int_{\partial
\widehat{\Omega}_\hole} G_{i}^{\mathfrak{t}}\mathcal{N}. 
\end{equation}
and, for $k \in \N$,
$
\langle  g_k(\mX) \rangle := \tfrac12 [  \Delta, \chi_+ + \chi_-] \left( \frac{X_2^k}{k!}\right)$, 
$\left[  g_k(\mX)   \right] := [  \Delta, \chi_+ - \chi_-] \left( \frac{X_2^k}{k!}\right)$.
Moreover, $\delta_k^{\text{odd}}$ is equal to the remainder of the euclidian
division of $k$ by $2$  (\ie $\delta_k^{\text{odd}} $ is equal to $1$ if $k$ is odd
and equal to $0$ if $k$ is even), $\delta_k^{\text{even}} = 1 -
\delta_k^{\text{odd}}$ and,  $\lfloor r \rfloor$ denotes the floor of a
real number $r$. \\

\noindent Similarly, let $W_i^{\mathfrak{n}} = 0$, for $i\leq 0$. Then, for $i\geq 1$, we define $W_{i}^{\mathfrak{n}} \in \mathcal{V}^+(\mathcal{B})$ as the unique decaying solution to
\begin{equation}\label{ProblemWpn}
\left\lbrace
\begin{aligned}
\dsp - \Delta_{\mX} W_{i}^{\mathfrak{n}}(\mX) & \; = &&  F_{i}^{\mathfrak{n}}(\mX) + \frac{\mathcal{D}_{i}^{\mathfrak{n}}  }{2}
\dsp [g_0(\mX)] +  \frac{\mathcal{N}_{i}^{\mathfrak{n}}}{2}
\dsp [g_1(\mX)]  \quad \mbox{in} \; \mathcal{B}, \\
\partial_{\mathbf{n}} W_{i}^{\mathfrak{n}} & \;= &&G_{i}^{\mathfrak{n}}(\mX)     \quad  \mbox{on} \;  \partial
\widehat{\Omega}_\hole,\\
\partial_{X_1} W_{i}^{\mathfrak{n}} (0,X_2) & \;= &&  \partial_{X_1} W_{i}^{\mathfrak{n}}(1,X_2), \quad X_2 \in
\R,
\end{aligned} 
\right.
\end{equation}
where $G_i^{\mathfrak{n}}(\mX) =  -W_{i-1}^{\mathfrak{n}} \mathbf{e}_1\cdot {\mathbf{n}}$ and
\begin{multline}\label{DefFpn}
F_{i}^{\mathfrak{n}}(\mX) = 2 \partial_{X_1} W_{i-1}^{\mathfrak{n}}(\mX) \,+\, W_{i-2}^{\mathfrak{n}}(\mX)
\,+\, (-1)^{\lfloor i/2 \rfloor} \left(
  2 \, \langle g_i(\mX)\rangle \, \delta_{i}^{\text{odd}} \right)  \\+ \sum_{k=2}^{i-1}
(-1)^{\lfloor k/2 \rfloor}  \frac{[g_k(\mX)]}{2}
\delta_{k}^{\text{even}} \mathcal{D}_{i-k}^{\mathfrak{n}} + \sum_{k=2}^{i-1}
(-1)^{\lfloor k/2 \rfloor}  \, \frac{[g_k(\mX)]}{2}
\delta_{k}^{\text{odd}}  \, \mathcal{N}_{i-k+1}^{\mathfrak{n}},
\end{multline}
 the constants $\mathcal{D}_i^{\mathfrak{n}}$ and $\mathcal{N}_i^{\mathfrak{n}}$ being given by
%\begin{equation}\label{DefDpnNpn}
%\mathcal{D}_p^{\mathfrak{n}}= \int_{\mathcal{B}} F_{p}^{\mathfrak{n}}(\mX)
%\mathcal{D}(\mX) d\mX, \quad \mathcal{N}_p^{\mathfrak{n}} = -\int_{\mathcal{B}} F_{p}^{\mathfrak{n}}(\mX)
%\mathcal{N}(\mX) d\mX.
%\end{equation}
%\begin{equation}\label{DefDpnNpn}
%\left.
%\begin{array}{l}
%\dsp \mathcal{D}_p^{\mathfrak{n}} = \int_{\mathcal{B}} F_{p}^{\mathfrak{n}}(\mX)
%\mathcal{D}(\mX) d\mX + \int_{\partial
%\widehat{\Omega}_\hole} G_{p}^{\mathfrak{n}}(\mX) \mathcal{D}(\mX)  d\sigma, \\ \dsp \mathcal{N}_p^{\mathfrak{n}} = -\int_{\mathcal{B}} F_{p}^{\mathfrak{n}}(\mX)
%\mathcal{N}(\mX) d\mX - \dsp \int_{\partial
%\widehat{\Omega}_\hole} G_{p}^{\mathfrak{n}}(\mX) \mathcal{N}(\mX)  d\sigma. 
%\end{array}
%\right.
%\end{equation}
\begin{equation}\label{DefDpnNpn}
\dsp \mathcal{D}_i^{\mathfrak{n}} = \int_{\mathcal{B}} F_{i}^{\mathfrak{n}}\,
\mathcal{D} + \int_{\partial
\widehat{\Omega}_\hole} G_{i}^{\mathfrak{n}}\, \mathcal{D} ,  \quad \quad \; \mathcal{N}_i^{\mathfrak{n}} = -\int_{\mathcal{B}} F_{i}^{\mathfrak{n}}\mathcal{N} - \dsp \int_{\partial
\widehat{\Omega}_\hole} G_{i}^{\mathfrak{n}}\mathcal{N}. 
\end{equation}

\begin{remark} The well posedness of Problem~\eqref{ProblemWpn} and
  Problem~\eqref{ProblemWpt} results from the application of
  Proposition~\ref{prop:layer_existence_uniqueness_problem_strip}
  noticing that the right-hand sides of Problem~\eqref{ProblemWpn} and
  Problem~\eqref{ProblemWpt} satisfy the
  conditions~\eqref{eq:prop_layer_existence_uniqueness_compatibility_D}-\eqref{eq:prop_layer_existence_uniqueness_compatibility_N}.
\end{remark}

\subsection{Definition of the profile functions $w_{n,q,\pm}$}
\label{sec:defin-prof-funct-wstar}
\label{sec:defin-prof-funct-wn}

We shall construct the functions $w_{n,q,\pm}$ as
\begin{equation}
  \label{eq:definition_w_star_p}
  w_{n,q,\pm}(\ln R^\pm,\theta^\pm) = \sum_{s=0}^q w_{n,q,s,\pm}(\theta^\pm) (\ln
  R^\pm)^s,  \quad q \in \N, \quad  w_{n,q,s,\pm} \in
  \mathcal{C}^\infty(\overline{I_1^\pm})\cap \mathcal{C}^\infty(\overline{I_2^\pm}),
\end{equation} 
where $I_1^\pm = (a^\pm, \gamma^\pm)$, $I_2^\pm = (\gamma^\pm, b^\pm)$
with $a^+ = 0$, $\gamma^+ = \pi$, $b^+ = \Theta$, and,
$a^- = \pi-\Theta$, $\gamma^- = 0$, $b^- = \pi$. The construction is
done by induction on $q$. The functions $w_{n,0,\pm}$ have already
been defined in~\eqref{eq:defintion_w_0_star}:
\begin{equation}
  \nonumber
  w_{0,0,\pm}(\ln R^\pm, \theta^\pm) = \ln R^\pm,   \quad  w_{n,0,+}(\theta^+) = \cos (\lambda_n \theta^+), \quad
  w_{n,0,-}( \theta^-) = \cos \left( \lambda_n (\theta^--\pi)\right),
\end{equation}

For $q \geq 1$, we construct $w_{n,q,\pm}$ of the form
\eqref{eq:definition_w_star_p} such that the function
\begin{equation*}
  \ttv_{n,q,\pm}(R^\pm, \theta^\pm) = (R^\pm)^{\lambda_n -q}
  w_{n,q,\pm}(\ln R^\pm, \theta^\pm)
\end{equation*}
satisfies
\begin{equation}
  \label{eq:problem_v_star_q_plus}
  \left\lbrace
    \begin{aligned}
      \Delta \ttv_{n,q,\pm} &  \; = &&  0 \; \mbox{in}\;
      \mathcal{K}_1^\pm \cap \mathcal{K}_2^\pm      , \\
      \partial_\theta \ttv_{n,q,\pm}(a^\pm)   = \partial_\theta \ttv_{n,q,\pm}(b^\pm) & \; = && 0 , \\
      \dsp [\ttv_{n,q,\pm}(R^\pm, \gamma^\pm)]_{\partial \mathcal{K}_1^\pm
        \cap \partial \mathcal{K}_2^\pm } & \;  = &&
      (R^\pm)^{\lambda_n-q}  \,
      \tta_{n,q,\pm}(\ln R^\pm)  ,  \\
      \dsp [\partial_{\theta^{\pm}} \ttv_{n,q,\pm}(R^\pm, \gamma^\pm)]_{\partial
        \mathcal{K}_1^\pm \cap \partial \mathcal{K}_2^\pm } & \; =
      &&
      (R^\pm)^{\lambda_n-q} \, \ttb_{n,q,\pm}(\ln R^\pm), 
\end{aligned} \right. \quad \;  \forall q \in \IN^\ast,
\end{equation}
where, for $j = \lbrace 1,2 \rbrace$,
$
%  \label{eq:definitionCones1et2}
  \mathcal{K}_j^{\pm}  = \left\{ (R^\pm \cos \theta^\pm,  R^\pm \sin
    \theta^\pm) \in \mathcal{K}^\pm, \ R^\pm \in \IR^\ast, \theta^\pm \in I_j^\pm \right\}
$, 
and,
\begin{align}
  \label{eq:a_star_p_pm}
  & \tta_{n,q,\pm}(\ln R^\pm)= \sum_{r=0}^{q-1} \left( \mathcal{D}_{q-r}^{\mathfrak{t}}
  g_{n,r,q-r,\pm}^{\mathfrak{t}}(\ln R^\pm) +  \mathcal{D}_{q-r}^{\mathfrak{n}}
  g_{n,r,q-r,\pm}^{\mathfrak{n}}(\ln R^\pm) \right), \\
  \label{eq:b_star_p_pm}
  &  \ttb_{n,q,\pm}(\ln R^\pm)= \sum_{r=0}^{q-1} \left( \mathcal{N}_{q+1 - r}^{\mathfrak{t}}
  h_{n,r,q-r,\pm}^{\mathfrak{t}}(\ln R^\pm) +  \mathcal{N}_{q+ 1 -r}^{\mathfrak{n}}
  h_{n,r,q-r,\pm}^{\mathfrak{n}}(\ln R^\pm) \right).
\end{align}
The  reals coefficients $\mathcal{D}_i^{\mathfrak{t}}$,
$\mathcal{D}_i^{\mathfrak{n}}$, $\mathcal{N}_i^{\mathfrak{t}}$ and
$\mathcal{N}_i^{\mathfrak{t}}$ are defined in~\eqref{DefDptNpt}-\eqref{DefDpnNpn}. The functions $g_{n,r,t,\pm}^{\mathfrak{t}}$,
$g_{n,r,t,\pm}^{\mathfrak{n}}$ are defined by the following relations:
for $r\in \N$, $t \in \N$, 
\begin{align*}
%  \label{eq:g_n_star_p_pm_t}
  & (R^\pm)^{\lambda_n-r -t}  \; g_{n,r,t,\pm}^{\mathfrak{t}}(\ln R^+)
    \\ & \quad = (\mp 1)^t
    \frac{\partial^t}{(\partial R^\pm)^t}  \left[ (R^\pm)^{\lambda_n-r} \,
    \langle w_{n,r,\pm}(\gamma^\pm, \ln R^\pm) \rangle 
    _{\partial \mathcal{K}_1^\pm \cap \partial \mathcal{K}_2^\pm} \right],
  \\
  %\label{eq:g_n_star_p_pm_n}
  & (R^\pm)^{\lambda_n-r -t} \; g_{n,r,t,\pm}^{\mathfrak{n}}(\ln R^+)
  \\ & \quad = (\mp 1)^t
    \frac{\partial^{t-1}}{( \partial R^\pm)^{t-1}}  \left[ (R^\pm)^{\lambda_n-r-1} \,
    \langle \partial_{\theta^\pm} w_{n,r,\pm}(\gamma^\pm, \ln R^\pm) \rangle
    _{\partial \mathcal{K}_1^\pm \cap \partial \mathcal{K}_2^\pm}
    \right], (t\geq 1)
\end{align*}
$g_{n,r,0,\pm}^{\mathfrak{n}} =0$, $
  h_{n,r,t,\pm}^{\mathfrak{t}} = \mp g_{n,r,t+1,\pm}^{\mathfrak{t}}$ and
  $h_{n,r,t,\pm}^{\mathfrak{n}} = \mp g_{n,r,t+1,\pm}^{\mathfrak{n}}$.\\

%The recursive procedure to construct the terms $w_{n,p,\pm}$ is
%the following: assume that the terms $w_{n,q,\pm}$ are known for
%any $q \leq p-1$. To construct $w_{n,p,\pm}$
%we need to know the source terms $\tta_{n,p,\pm}$ and
%$\ttb_{n,p,\pm}$, which require the computation of
%$g_{n,r,p-r,\pm}^\mathfrak{t}$ and
%$g_{n,r,p-r,\pm}^\mathfrak{n}$ for $r \leq p-1$. But,
%$g_{n,r,p-r,\pm}^\mathfrak{t}$ and
%$g_{n,r,p-r,\pm}^\mathfrak{n}$ only depend on the function
%$w_{n,r,\pm}$, which, thanks to the induction hypothesis is
%known. Similarly, to construct $\ttb_{n,p,\pm}$, we need to define
%$h_{n,r,p-r,\pm}^\mathfrak{t}$ and
%$h_{n,r,p-r,\pm}^\mathfrak{n}$ for $r\leq p-1$. These terms only
%depend on $w_{n,r,\pm}$, and, as a consequence, are known.  

\noindent The
existence $w_{n,q,\pm}$ of the form~\eqref{eq:definition_w_star_p}
results from the following Lemma (see also Chapter 3
in Ref.~\cite{TheseAbdelkader} and the Section 6.4.2 in Ref.~\cite{MR1469972}
for the proof).
\begin{lemma}
  \label{lema:technical_lemma_sol}
  Let $j\in \N$, $\lambda \in \R$, $(\tta,\ttb) \in \R^2$ and 
  \begin{equation*}
    N = \begin{cases} 
      j & \mbox{if}  \; \lambda \notin \frac{\pi}{\Theta}\IZ, \\
      j+1 & \mbox{if} \; \lambda \in \frac{\pi}{\Theta}\IZ^\ast, \\
      j+2 & \mbox{if} \; \lambda = 0.
    \end{cases} 
  \end{equation*}
  There exist $N+1$ functions
  $g_k \in \mathcal{C}^\infty(\overline{I_1^\pm}) \cap
  \mathcal{C}^\infty(\overline{I_2^\pm} )$,
  ($0 \leq k \leq N$), such that the function
  $
    \ttv(R^\pm, \theta^\pm) = (R^\pm)^\lambda \left(  \sum_{k=0}^{N}
      (\ln R^\pm)^k g_k(\theta^\pm) \right) 
  $  satisfies
  \begin{equation}
    \label{eq:ProblemeLemmeTechniqueCone}
    \left \lbrace
      \quad \begin{aligned}
        \Delta \ttv &\; = &&  0 \; \mbox{in}\; \mathcal{K}_1^\pm \cap \mathcal{K}_2^\pm  , \\
        \partial_\theta \ttv(R^\pm,a^\pm) =   \partial_\theta \ttv(R^\pm,b^\pm) & \; = &&  0, \\ 
      %  \partial_\theta \ttv(R^\pm,b^\pm) & \; = &&  0, \\
        [\ttv(R^\pm,\gamma^\pm)]_{\partial \mathcal{K}_1^\pm
          \cap \partial \mathcal{K}_2^\pm} & \; = &&
        \tta \, (R^\pm)^\lambda \, \ln (R^\pm)^j,  \\
        [\partial_\theta \ttv(R^\pm,\gamma^\pm)]_{\partial \mathcal{K}_1^\pm
          \cap \partial \mathcal{K}_2^\pm} & \; =&& \ttb \, (R^\pm)^\lambda
        \, \ln (R^\pm)^j.
      \end{aligned} 
    \right.
  \end{equation}
\end{lemma}

\begin{remark} If $\lambda_n - q \in \frac{\pi}{\Theta}\IZ^\ast$, the function
  $w_{n,q,\pm}$ is not uniquely defined by
  \eqref{eq:definition_w_star_p} because we can add any multiple of
  the function $\theta^\pm \mapsto w_{\frac{\Theta}{\pi} (\lambda_n -
    q),0,\pm} (\theta^\pm)$. In that case, we restore the uniqueness
  taking the orthogonal projection of $w_{n,q,0,\pm}$  with respect to
  $w_{\frac{\Theta}{\pi} (\lambda_n - q),0,\pm}$, \ie
  \begin{equation*}
    \label{ConditionOrthogonalitePlus_1}
    \int_{a^\pm}^{b^\pm} w_{n, q, 0,\pm }(\ln R^\pm, \theta^\pm)
    w_{\frac{\Theta}{\pi} (\lambda_n - q),0,\pm} (\theta^\pm) d \theta^\pm = 0,
    \quad \big( \lambda_n - q \in
    \frac{\pi}{\Theta}\IZ^\ast , q\geq 1 \big).
  \end{equation*}
Similarly, for $n>0$, if $\lambda_n - q = 0$, the function $w_{n,q,\pm}$ is not uniquely defined by
  \eqref{eq:definition_w_star_p}, because we can add any multiple of
  the functions $1$ and
  $ \ln R^\pm$. Here again, the uniqueness is restored by imposing $
    \int_{a^\pm}^{b^\pm} w_{n, q, 0,\pm }
 d \theta^\pm =    \int_{a^\pm}^{b^\pm} w_{n, q, 1,\pm }
   d \theta^\pm = 0$.
\end{remark}
\subsection{Definition of the profile functions $p_{n,q,\pm}$}
\label{sec:defin-prof-funct_p}
Finally, the functions $p_{n,q,\pm}$ are given by
\begin{equation} \label{definitionPmr}
p_{n,q,\pm}(\ln |X_1^\pm|,\mathbf{X}^\pm) =
\sum_{i=0}^q \, \;
g_{n,q-i,i,+}^{\mathfrak{t}}(\ln |X_1^\pm|)\;
  W_i^{\mathfrak{t}} (\mathbf{X}^\pm)
  +\sum_{i=1}^q \;
  g_{n,q-i,i,+}^{\mathfrak{n}}(\ln |\ln X_1^\pm|) \;
  W_i^{\mathfrak{n}}(\mathbf{X}^\pm
  ) .
\end{equation}

\end{appendix}

\def\cprime{$'$}


\begin{thebibliography}{10}

\bibitem{Ammari}
T.~Abboud and H.~Ammari.
\newblock Diffraction at a curved grating: {TM} and {TE} cases, homogenization.
\newblock {\em J. Math. Anal. Appl.}, 202(3):995--1026, 1996.

\bibitem{MR0167642}
Milton Abramowitz and Irene~A. Stegun.
\newblock {\em Handbook of mathematical functions with formulas, graphs, and
  mathematical tables}, volume~55 of {\em National Bureau of Standards Applied
  Mathematics Series}.
\newblock For sale by the Superintendent of Documents, U.S. Government Printing
  Office, Washington, D.C., 1964.

\bibitem{Achdou}
Y.~Achdou.
\newblock Etude de la r\'eflexion d'une onde \'electromagn\'etique par un
  m\'etal recouvert d'un rev\^etement m\'etallis\'e.
\newblock Technical report, INRIA, 1989.

\bibitem{AchdouCR}
Y.~Achdou.
\newblock {Effect of a thin metallized coating on the reflection of an
  electromagnetic wave }.
\newblock {\em C. R. Acad. Sci. Paris, Ser. I}, {314}({3}):{217--222}, January
  {1992}.

\bibitem{AchdouPironneauValentin}
Y.~Achdou, O.~Pironneau, and F.~Valentin.
\newblock Effective boundary conditions for laminar flows over periodic rough
  boundaries.
\newblock {\em J. Comput. Phys.}, 147(1):187--218, 1998.

\bibitem{ArtolaCessenat}
M.~Artola and M.~Cessenat.
\newblock Diffraction d'une onde \'electromagn\'etique par une couche composite
  mince accol\'ee \`a\ un corps conducteur \'epais. {I}. {C}as des inclusions
  fortement conductrices.
\newblock {\em C. R. Acad. Sci. Paris, Ser.~I}, 313(5):231--236, 1991.

\bibitem{ArtolaCessenat2}
M.~Artola and M.~Cessenat.
\newblock {Scattering of an electromagnetic wave by a slender composite slab in
  contact with a thick perfect conductor. II. Inclusions (or coated material)
  with high conductivity and high permeability}.
\newblock {\em C.~R. Acad. Sci. Paris, Ser. I}, 313(6):381--385, 1991.

\bibitem{BendaliMakhloufTordeux}
A.~Bendali, A.~Makhlouf, and S.~Tordeux.
\newblock Field behavior near the edge of a microstrip antenna by the method of
  matched asymptotic expansions.
\newblock {\em Quart. Appl. Math.}, 69(4):691--721, 2011.

\bibitem{MR2573145}
V.~Bonnaillie-No{\"e}l, M.~Dambrine, S.~Tordeux, and G.~Vial.
\newblock Interactions between moderately close inclusions for the {L}aplace
  equation.
\newblock {\em Math. Models Meth. Appl. Sci.}, 19(10):1853--1882, 2009.

\bibitem{Bonnet.Drissi.Gmati:2005}
A.-S.~Bonnet-Ben~Dhia, D.~Drissi, and N.~Gmati.
\newblock Mathematical analysis of the acoustic diffraction by a muffler
  containing perforated ducts.
\newblock {\em Math. Models Meth. Appl. Sci.}, 15(7):1059--1090, 2005.

\bibitem{BreschMilisic2010}
D.~Bresch and V.~Milisic.
\newblock High order multi-scale wall-laws, {P}art {I}: the periodic case.
\newblock {\em Quart. Appl. Math.}, 68(2):229--253, 2010.

\bibitem{BrezisAnglais}
H.~Brezis.
\newblock {\em Functional analysis, {S}obolev spaces and partial differential
  equations}.
\newblock Universitext. Springer, New York, 2011.

\bibitem{CalozVial}
G.~Caloz, M.~Costabel, M.~Dauge, and G.~Vial.
\newblock Asymptotic expansion of the solution of an interface problem in a
  polygonal domain with thin layer.
\newblock {\em Asymptot. Anal.}, 50(1-2):121--173, 2006.

\bibitem{CiupercaJaiPoignard}
I.~S. Ciuperca, M.~Jai, and C.~Poignard.
\newblock Approximate transmission conditions through a rough thin layer: the
  case of periodic roughness.
\newblock {\em European J. Appl. Math.}, 21(1):51--75, 2010.

\bibitem{XavierArticle}
X.~Claeys.
\newblock On the theoretical justification of {P}ocklington's equation.
\newblock {\em Math. Models Meth. Appl. Sci.}, 19(8):1325--1355, 2009.

\bibitem{Claeys.Delourme:2013}
X.~Claeys and B.~Delourme.
\newblock High order asymptotics for wave propagation across thin periodic
  interfaces.
\newblock {\em Asymptot. Anal.}, 83(1--2):35--82, 2013.

\bibitem{conceptsweb}
{Concepts Development Team}.
\newblock {\em Webpage of {Numerical} {C++} {Library} {Concepts} 2}.
\newblock http://www.concepts.math.ethz.ch, 2016.

\bibitem{DaugeTordeuxVialVersionLongue}
M.~Dauge, S.~Tordeux, and G.~Vial.
\newblock Selfsimilar perturbation near a corner: matching versus multiscale
  expansions for a model problem.
\newblock In {\em Around the research of {V}ladimir {M}az'ya. {II}}, volume~12
  of {\em Int. Math. Ser. (N. Y.)}, pages 95--134. Springer-Verlag, New York,
  2010.

\bibitem{TheseDelourme}
B.~Delourme
\newblock{Mod\`eles et asymptotiques des interfaces fines et p\'riodiques en \'lectromagn\'etisme}
\newblock PhD thesis,  Universit\'e Pierre et Marie Curie, 2010.

\bibitem{ResearchReportAKB}
B.~Delourme, K.~Schmidt, and A.~Semin.
\newblock When a thin periodic layer meets corners: asymptotic analysis of a
  singular poisson problem.
\newblock Technical report, June 2015.

\bibitem{Delourme.Semin.Schmidt:2015}
B.~Delourme, K.~Schmidt, and A.~Semin.
\newblock On the homogenization of thin perforated walls of finite length.
\newblock {\em Asymptotic Analysis}, 97(3-4):211-264, 2016.

\bibitem{Frauenfelder.Lage:2002}
Ph. Frauenfelder and Ch. Lage.
\newblock Concepts -- an object-oriented software package for partial
  differential equations.
\newblock {\em ESAIM: Math. Model. Numer. Anal.}, 36(5):937--951, 2002.

\bibitem{Goldstein:1982}
C.I. Goldstein.
\newblock {A finite element method for solving Helmholtz type equations in
  waveguides and other unbounded domains}.
\newblock {\em Math. Comp.}, 39(160):309--324, 1982.

\bibitem{Ilin}
A.~M. Il{\cprime}in.
\newblock {\em Matching of asymptotic expansions of solutions of boundary value
  problems}, volume 102 of {\em Translations of Mathematical Monographs}.
\newblock American Mathematical Society, Providence, RI, 1992.
\newblock Translated from the Russian by V. Minachin [V. V. Minakhin].

\bibitem{fente1}
P.~Joly and S.~Tordeux.
\newblock Matching of asymptotic expansions for wave propagation in media with
  thin slots. {I}. {T}he asymptotic expansion.
\newblock {\em Multiscale Model. Simul.}, 5(1):304--336 (electronic), 2006.

\bibitem{fente2}
P.~Joly and S.~Tordeux.
\newblock Matching of asymptotic expansions for waves propagation in media with
  thin slots. {II}. {T}he error estimates.
\newblock {\em ESAIM: Math. Model. Numer. Anal.}, 42(2):193--221, 2008.

\bibitem{Jol-Sem-2008}
P.~Joly and A.~Semin.
\newblock Construction and analysis of improved kirchoff conditions for
  acoustic wave propagation in a junction of thin slots.
\newblock {\em ESAIM Proceeding}, 25:44--67, Dec 2008.

\bibitem{MR1469972}
V.~A. Kozlov, V.~G. Mazya, and J.~Rossmann.
\newblock {\em Elliptic boundary value problems in domains with point
  singularities}, volume~52 of {\em Mathematical Surveys and Monographs}.
\newblock American Mathematical Society, Providence, RI, 1997.

\bibitem{Madureira}
A.L. Madureira and F.~Valentin.
\newblock Asymptotics of the {P}oisson problem in domains with curved rough
  boundaries.
\newblock {\em SIAM J. Math. Anal.}, 38(5):1450--1473 (electronic), 2006/07.

\bibitem{TheseAbdelkader}
A.~Makhlouf.
\newblock {\em Justification et am\'elioration de mod\`eles d'antenne patch par
  la m\'ethode des d\'eveloppements asymptotiques raccord\'es}.
\newblock PhD thesis, Institut National des Sciences Appliqu\'ees de Toulouse,
  2008.

\bibitem{LivreNazarov1}
V.~Maz'ya, S.~Nazarov, and B.~Plamenevskij.
\newblock {\em Asymptotic theory of elliptic boundary value problems in
  singularly perturbed domains. {V}ol. {I}}, volume 111 of {\em Operator
  Theory: Advances and Applications}.
\newblock Birkh\"auser Verlag, Basel, 2000.
\newblock Translated from the German by Georg Heinig and Christian Posthoff.

\bibitem{Nazarov205}
S.~A. Nazarov.
\newblock The {N}eumann problem in angular domains with periodic and parabolic
  perturbations of the boundary.
\newblock {\em Tr. Mosk. Mat. Obs.}, 69:182--241, 2008.

\bibitem{Nicaise}
S.~Nicaise and A.-M. S{\"a}ndig.
\newblock General interface problems. {I}, {II}.
\newblock {\em Math. Meth. Appl. Sci.}, 17(6):395--429, 431--450, 1994.

\bibitem{Panasenko81}
G.~Panasenko
\newblock{\em High order asymptotics of solutions of problems on the contact of periodic structures.}
\newblock{\em Sbornik: Mathematics}, 38(4), 465-494, 1981.

\bibitem{PavliotisStuart2008}
G.A.~Pavliotis and A.M.~Stuart.
\newblock {\em Multiscale Methods: Averaging and Homogenization}.
\newblock Springer, 2008.

\bibitem{poirier2006impedance}
J.-R. Poirier, A.~Bendali, and P.~Borderies.
\newblock {Impedance boundary conditions for the scattering of time-harmonic
  waves by rapidly varying surfaces}.
\newblock {\em IEEE Trans. Antennas and Propagation}, 54(3):995--1005, 2006.

\bibitem{Poirier}
J.-R. Poirier, A.~Bendali, P.~Borderies, and S.~Tournier.
\newblock High order asymptotic expansion for the scattering of fast
  oscillating periodic surfaces.
\newblock In {\em Proc. 9th Int. Conf. on Mathematical and Numerical Aspects of
  Waves Propagation (Waves 2009), Pau, France}, 2009.

\bibitem{RauchTaylor}
J.~Rauch, M.~Taylor.
\newblock Potential and scattering theory on wildly perturbed domains
\newblock Journal  of functional analysis 18, 27-59, 1975.

\bibitem{SanchezPalencia}
E.~S{\'a}nchez-Palencia.
\newblock {\em Nonhomogeneous media and vibration theory}, volume 127 of {\em
  Lecture Notes in Physics}.
\newblock Springer-Verlag, Berlin, 1980.

\bibitem{RapportSanchezPalencia}
E.~S{\'a}nchez-Palencia.
\newblock Un probl\`eme d'\'ecoulement lent d'un fluide incompressible au
  travers d'une paroi finement perfor\'ee.
\newblock In {\em Homogenization methods: theory and applications in physics
  ({B}r\'eau-sans-{N}appe, 1983)}, volume~57 of {\em Collect. Dir. \'Etudes
  Rech. \'Elec. France}, pages 371--400. Eyrolles, Paris, 1985.

\bibitem{Schmidt.Kauf:2009}
K.~Schmidt and P.~Kauf.
\newblock Computation of the band structure of two-dimensional photonic
  crystals with {\em hp} finite elements.
\newblock {\em Comput. Methods Appl. Mech. Engrg.}, 198:1249--1259, March 2009.

\bibitem{Schwab1998}
C.~Schwab.
\newblock {\em $p$- and $hp$-finite element methods: Theory and applications in
  solid and fluid mechanics}.
\newblock Oxford University Press, Oxford, UK, 1998.

\bibitem{VanDyke}
M.~Van~Dyke.
\newblock {\em Perturbation methods in fluid mechanics}.
\newblock Applied Mathematics and Mechanics, Vol. 8. Academic Press, New York,
  1964.

\end{thebibliography}
\end{document}